\documentclass[11pt]{article}
\usepackage{latexsym,amssymb,amsmath,amsfonts,amsthm,mathrsfs}
\usepackage{graphicx}
\usepackage{dcolumn}
\usepackage{bm}
\usepackage{color}
\usepackage{epsfig}
\usepackage{multirow}

\makeatletter
\newif\ifmsbmloaded@
\def\loadmsbm{\msbmloaded@true
  \font\tenmsb=msbm10 scaled 1\@ptsize00
  \font\sevenmsb=msbm7 scaled 1\@ptsize00
  \font\fivemsb=msbm5 scaled 1\@ptsize00
  \alloc@8\fam\chardef\sixt@@n\msbfam
  \textfont\msbfam=\tenmsb
  \scriptfont\msbfam=\sevenmsb
  \scriptscriptfont\msbfam=\fivemsb
  }
\def\nonmatherr@#1{\errmessage%
{LateX error: \string#1\space allowed only in math mode}}
\def\Bbb{\relax\ifmmode\expandafter\Bbb@\else
  \expandafter\nonmatherr@\expandafter\Bbb\fi}
\def\Bbb@#1{{\Bbb@@{#1}}}
\def\Bbb@@#1{\fam\msbfam\relax#1}
\@addtoreset{equation}{section}

\makeatother \loadmsbm

\def\u1{u_1}
\def\b1{b_1}

\newcommand{\beq}{\begin{equation}}
\newcommand{\eeq}{\end{equation}}
\newcommand{\ben}{\begin{eqnarray}}
\newcommand{\een}{\end{eqnarray}}
\newcommand{\beno}{\begin{eqnarray*}}
\newcommand{\eeno}{\end{eqnarray*}}

\newtheorem{thm}{Theorem}[section]
\newtheorem{lem}{Lemma}[section]
\newtheorem{remk}{Remark}[section]
\newtheorem{prop}{Proposition}[section]
\newtheorem{cor}{Corrolary}[section]
\newtheorem{defin}{Definition}[section]

\begin{document}
\title{\bf  $\alpha$-Modulation Spaces (I) Scaling, Embedding and Algebraic Properties }

\author{\bf Jinsheng Han \\
{\it \small Department of Mathematics, Shanghai Jiao Tong University, Shanghai 200240, P. R. China.}\\
   \bf  Baoxiang Wang \\
    {\it \small LMAM, School of Mathematics, Peking University and BICMR, Beijing 100871, P. R. China.}\\
    {\small (pkuhan@sjtu.edu.cn, wbx@pku.edu.cn)}     }
\date{}
 \maketitle

\renewcommand{\thefootnote}{\fnsymbol{footnote}}

\vspace{-1.2in} \vspace{.9in} \vspace{0.2cm} {\bf Abstract.} First, we consider some fundamental properties including dual spaces, complex interpolations of $\alpha$-modulation spaces $M^{s,\alpha}_{p,q}$ with $0<p,q \le \infty$. Next, necessary  and sufficient conditions for the scaling property and the inclusions between $\alpha_1$-modulation and $\alpha_2$-modulation spaces are obtained. Finally, we give some criteria for $\alpha$-modulation spaces constituting multiplication algebra. As a by-product, we show that there exists an $\alpha$-modulation space which is not an interpolation space between modulation and Besov spaces. In a subsequent paper, we will give some applications of $\alpha $-modulation spaces to nonlinear dispersive wave equations.

\vspace{0.2cm} {\bf Key words.} $\alpha$-modulation space, Dual space,
 Multiplication algebra, Scaling property,
Embedding.
 \vspace{0.2cm}

{\bf AMS subject classifications.} 42 B35, 42 B37, 35 A23

\section{Introduction and definition}

Frequency localization technique plays an important role in the modern theory of function spaces.
There are two kinds of basic partitions to the Euclidean
space $\mathbb{R}^n$, one is the dyadic decomposition $\mathbb{R}^n = \{\xi: |\xi|< 1\} \bigcup (\bigcup^\infty_{j=1} \{\xi : |\xi|\in [2^{j-1}, 2^j)\})$,  another is the uniform decomposition $\mathbb{R}^n = \bigcup_{k\in \mathbb{Z}^n} (k+[-1/2, 1/2)^n)$. According to these two kinds of decompositions in  frequency spaces, one can naturally introduce the dyadic decomposition operators
$\triangle_j \ (j\in\mathbb{Z}_+)$ whose symbol $\varphi_j$ is localized in $\{\xi: |\xi|\sim 2^j\}$, and the uniform decomposition operator  $\square_k \ (k\in\mathbb{Z}^n)$ whose symbol $\sigma_k$ is supported in $k+[-1,1]^n$. The difference between $\varphi_j$  and $\sigma_k$ is
that the diameters of ${\rm supp} \ \varphi_j$ and ${\rm supp} \ \sigma_k$ are $O(2^j)$ and $O(1)$, respectively. All tempered
distributions acted on these decomposition operators with finite
$\ell^q(L^p)$ (quasi)-norms constitute Besov space $B_{p,q}^s$ and
modulation space $M_{p,q}^s$, respectively.

The $\alpha$-modulation spaces
$M_{p,q}^{s,\alpha}$, introduced by Gr\"obner \cite{Gb}, are proposed to be intermediate function spaces
to connect modulation space and Besov space with respect to parameters $\alpha\in [0,1]$, which are formulated by
some new kind of $\alpha$-decomposition operators $\square_k^{\alpha} \ (k\in\mathbb{Z}^n)$. We denote by
$\eta_k^{\alpha}$ the symbol of $\square_k^{\alpha}$, whose essential characteristic is that
the diameter of its support set has
power growth as $\langle k\rangle^{\alpha/(1-\alpha)}$.

Modulation spaces are special $\alpha$-modulation spaces in the case $\alpha=0$, and Besov space can be regarded as the limit case of $\alpha$-modulation space when $\alpha \nearrow 1$. Modulation spaces were first introduced by
Feichtinger \cite{F} in the study of time-frequency analysis to
consider the decay property of a function in both physical and
frequency spaces. His original idea is to use the short-time Fourier transform
of a tempered distribution equipping with a mixed $L^{q}(L^{p})$-norm to generate modulation spaces
$M_{p,q}^s$. Gr\"{o}chenig's book \cite{Gc} systematically discussed
the theory of time-frequency analysis and modulation spaces.
In Gr\"{o}bner's doctoral thesis,   he used the
$\alpha$-covering to the frequency space $\mathbb{R}^n$ and a
corresponding bounded admissible partition of unity of order $p$
($p$-BAPU) to define $\alpha$-modulation spaces. Some recent    works have been devoted to the study of $\alpha$-modulation spaces (see \cite{BoNi06, BoNi062, DaFor08, For07, KoSuTo09, KoSuTo091} and references therein). Borup and Nielsen \cite{BoNi06} and Fornasier \cite{For07} constructed  Banach frames for $\alpha$-modulation spaces in the multivariate setting,  Kobayashi,  Sugimoto and  Tomita \cite{KoSuTo09, KoSuTo091} discussed the boundedness for a class of pseudo-differential operators with symbols in $\alpha$-modulation spaces.   Dahlke, Fornasier, Rauhut, Steidl and Teschke \cite{DaFor08} established the relationship between  the generalized coorbit theory and
$\alpha$-modulation spaces. The aim of the present paper is to describe some standard properties including the dual spaces, embeddings, scaling and algebraic structure of $\alpha$-modulation spaces.

Before stating the notion of $\alpha$-modulation spaces, we introduce some notations frequently used in this paper.
$A\lesssim B$ stands for $A\le C B$, and $A\sim B$ denote $A\lesssim B$ and $B\lesssim A$, where $C$ is a positive constant which can be different at different places.
Let $\mathscr{S}(\mathbb{R}^n)$  be the Schwartz space and $\mathscr{S}'(\mathbb{R}^n)$ be its strongly topological dual space. Suppose $f\in\mathscr{S}'({\mathbb{R}^n})$ and $\lambda>0$, we write $f_{\lambda}(\cdot)=f(\lambda\cdot)$. Let $X$ be a (quasi-)Banach space, we denote by $X^*$  the dual space of $X$.  For any $p\in [1,\infty]$, $p^*$ will stand for
the dual number of $p$, i.e., $1/p+1/p^*=1$. We denote by $L^p= L^p(\mathbb{R}^n)$  the Lebesgue space for which the norm is written by $\|\cdot\|_p$, and by $\ell^p$ the sequence Lebesgue space. We will write
$\langle x\rangle=(1+|x|^2)^{1/2}$. For any multi-index
$\delta=(\delta_1,\delta_2,\cdots,\delta_n)$, we denote
$D^{\delta}=\partial_1^{\delta_1}\partial_2^{\delta_2}\cdots\partial_n^{\delta_n}$.
It is convenient to divide $\mathbb{R}^n$ into $n$ parts
$\mathbb{R}^n_j,j=0,1,2,\cdots,n$:
\begin{equation*}
\mathbb{R}^n_j=\left\{x\in\mathbb{R}^n: |x_i|\leqslant|x_j|,\,
i=1,\cdots,j-1,j+1,\cdots,n\right\}.
\end{equation*}
 We write $J=(I-\Delta)^{s/2}$ and define the Sobolev space
$$
H^s(\mathbb{R}^n)=\left\{f\in\mathscr{S}'(\mathbb{R}^n): \|f\|_{H^s}=\|J^sf\|_2<\infty\right\}
$$
and
$$
\ell_{s,\alpha}^{q}(\mathbb{Z}^n;L^p)=\left\{\{g_k\}_{k\in\mathbb{Z}^n}:g_k\in\mathscr{S}'(\mathbb{R}^n),\left\|\langle
k\rangle^{\frac{s}{1-\alpha}}\|g_k\|_p\right\|_{\ell^q}<\infty\right\}.
$$
Without additional note, we will always assume that
$$
s\in\mathbb{R}, \ \ 0<p,q\leqslant \infty, \ \  0\leqslant\alpha< 1.
$$
Let us start with the third partition of unity on
frequency space for $\alpha\in[0,1)$ (see \cite{BoNi06}). We suppose $c<1$ and $C>1$ are two positive constants, which relate to the space dimension $n$, and a Schwartz function sequence $\{\eta_k^{\alpha}\}_{k\in\mathbb{Z}^n}$ satisfies
\begin{subequations}\label{eta}
\begin{align}
& |\eta_k^{\alpha}(\xi)|\gtrsim1, \quad  \mbox{if } \big|\xi-\langle k\rangle^{\frac{\alpha}{1-\alpha}}k\big|<c\langle k\rangle^{\frac{\alpha}{1-\alpha}}; \label{eta-1}\\
& \mbox{supp}\eta_k^{\alpha}\subset\big\{\xi:\big|\xi-\langle
k\rangle^{\frac{\alpha}{1-\alpha}}k\big|<C\langle
k\rangle^{\frac{\alpha}{1-\alpha}}\big\}; \label{eta-2}\\
& \sum\nolimits_{k\in\mathbb{Z}^n}\eta_k^{\alpha}(\xi)\equiv1, \quad
\forall\xi\in\mathbb{R}^n; \label{eta-3}\\
& \langle
k\rangle^{\frac{\alpha|\delta|}{1-\alpha}}\big|D^{\delta}\eta_k^{\alpha}(\xi)\big|\lesssim1,
\quad \forall\xi\in\mathbb{R}^n. \label{eta-4}
\end{align}
\end{subequations}
We denote
\begin{equation}
\Upsilon=\big\{\{\eta_k^{\alpha}\}_{k\in\mathbb{Z}^n}:\{\eta_k^{\alpha}\}_{k\in\mathbb{Z}^n}\,\mbox{satisfies \eqref{eta}}\big\}
\end{equation}
Corresponding to every sequence
$\{\eta_k^{\alpha}\}_{k\in\mathbb{Z}^n}\in\Upsilon$, one can construct an
operator sequence denoted by
$\{\square_k^{\alpha}\}_{k\in\mathbb{Z}^n}$, and
\begin{equation}
\square_k^{\alpha}=\mathscr{F}^{-1}\eta_k^{\alpha}\mathscr{F}.
\end{equation}
 $\Upsilon$ is nonempty. Indeed, let
$\rho$ be a smooth radial bump function supported in $B(0,2)$,
satisfying $\rho(\xi)=1$ as $|\xi|<1$, and $\rho(\xi)=0$ as
$|\xi|\geqslant2$. For any $k\in\mathbb{Z}^n$, we set
\begin{equation}
\rho_k^{\alpha}(\xi)=\rho\left( \frac{\xi-\langle
k\rangle^{\frac{\alpha}{1-\alpha}}k}{C \langle
k\rangle^{\frac{\alpha}{1-\alpha}} }\right) \label{rhok}
\end{equation}
and  denote
\begin{equation*}
\eta_k^{\alpha}(\xi)=\rho_k^{\alpha}(\xi)\left(\sum_{l\in\mathbb{Z}^n}\rho_l^{\alpha}(\xi)\right)^{-1}.
\end{equation*}
It is easy to verify that $\{\eta_k^{\alpha}\}_{k\in\mathbb{Z}^n}$
satisfies \eqref{eta}.  This type of decomposition on frequency space is a generalization of
the uniform decomposition and the dyadic decomposition. When
$0\leqslant\alpha<1$, on the basis of this decomposition, we define the
$\alpha$-modulation space by
\begin{align} \label{alphamodulation}
M_{p,q}^{s,\alpha}(\mathbb{R}^n)=\left\{f\in\mathscr{S}'(\mathbb{R}^n):\|f\|_{M_{p,q}^{s,\alpha}}= \left\|\{\square_k^{\alpha}f\}_{k\in\mathbb{Z}^n}\right\|_{\ell_{s,\alpha}^q(\mathbb{Z}^n;L^p)}<\infty\right\}.
\end{align}
Denote
$
\varphi(\xi)= \rho(\xi)- \rho(2\xi),
$ we may assume $\varphi(\xi)=1$ if $ 5/8 \le |\xi| \le 3/2$.
We introduce the function sequence $\{\varphi_k\}^\infty_{k=0}$:
\begin{align}\label{eq1.1.5}
\left\{
\begin{array}{l}
 \varphi_j (\xi)=\varphi (2^{-j} \xi), \quad j\in {\Bbb N},\\
\varphi_0 (\xi)=1-\sum ^\infty_{j=1}\varphi_j (\xi).
\end{array}
\right.
\end{align}
Define
\begin{align}\label{eq1.1.6} \index{$\triangle_k$}
\triangle_j = \mathscr{F}^{-1} \varphi_j \mathscr{F}, \quad j\in {\Bbb
N}\cup \{0\},
\end{align}
$\{\triangle_j\}^\infty_{j=0}$ is said to be the Littlewood--Paley (or dyadic) decomposition operators. Denote
\begin{align} \label{besov}
B_{p,q}^{s}(\mathbb{R}^n)=\left\{f\in\mathscr{S}'(\mathbb{R}^n):\|f\|_{B_{p,q}^{s}}= \left\|\{ 2^{sj} \triangle_j f\}_{j \in\mathbb{N} \cup \{0\}}\right\|_{\ell^q(L^p)}<\infty \right\}.
\end{align}
Strictly speaking, \eqref{alphamodulation} cannot cover the case $\alpha=1$, however, we will denote $M_{p,q}^{s,1} =
B_{p,q}^s$ for convenience.

The paper is organized as follows.  In Section 2, we
show some basic properties on $\alpha$-modulation spaces, their dual and complex interpolation spaces are presented there. In Section 3,
we discuss the scaling property. In Section 4, the
inclusions between $\alpha$-modulation spaces for different indices $\alpha$ (including
Besov spaces) are obtained. In
Section 5, we study the regularity conditions so that
$\alpha$-modulation spaces form  multiplication algebra. Finally, we show the necessity for the conditions of scalings, embeddings and algebra structures by constructing several counterexamples.

\section{Some basic properties}

In the sequel, we give some basic properties of
$M_{p,q}^{s,\alpha}$. We need the following
\begin{prop}[{\bf \cite{triebel}, Convolution in $L^p$ with $p<1$}]\label{convolution}
Let $0<p\le 1$. $L^p_{B(x_0,R)} =\{f\in L^p(\mathbb{R}^n): \ {\rm supp}f \subset B(x_0,R) \}$, $B(x_0,R)=\{x: |x-x_0|\le R\}$.
Suppose that $f, g \in L^p_{B(x_0,R)}$, then there exists a constant $C>0$ which is independent of $x_0$ and $R>0$ such that
$$
\|f*g\|_p \le C R^{n(1/p-1)} \|f\|_p \|g\|_p.
$$
\end{prop}

\begin{prop}\label{thm1.9}
{\rm ({\bf \cite{triebel}, Generalized Bernstein inequality})} \it Let $\Omega\subset
\mathbb{R}^n$ be a compact set, $0<r\le \infty.$ Let us denote
$\sigma_r=n( 1/(r\wedge 1)-1/2)$ and assume that $s>\sigma_r$. Then
there exists a constant $C>0$ such that
\begin{align} \label{1.35}
\|\mathscr{F}^{-1}\varphi \mathscr{F} f\|_r\le
C\|\varphi\|_{H^s}\|f\|_r
\end{align}
holds for all $f\in L^r_\Omega:=\{f\in L^p: \; {\rm supp}
\widehat{f} \subset \Omega\}$ and $\varphi\in H^s$. Moreover, if
$r\ge 1$, then \eqref{1.35} holds for all $f\in L^r$.
\end{prop}

\begin{prop}[{\bf Equivalent norm}]\label{equivalent-norm} Let
$\{\eta_k^{\alpha}\}_{k\in\mathbb{Z}^n},\{\widetilde{\eta}_k^{\alpha}\}_{k\in\mathbb{Z}^n}\in\Upsilon$,
then they genarate equivalent quasi-norms on $M_{p,q}^{s,\alpha}$.
\end{prop}

\begin{proof}
See \cite{BoNi06}.
\end{proof}

\begin{prop}[{\bf Embedding}]\label{embede}
Let $0< p_1\leqslant p_2 \le \infty$, $0< q_1, q_2 \le \infty$. We have \\
(i) if $q_1\leqslant q_2$ and $s_1\geqslant
s_2+n\alpha\big(\frac{1}{p_1}-\frac{1}{p_2}\big)$, then
\begin{equation}\label{embede-1}
M_{p_1,q_1}^{s_1,\alpha}\subset M_{p_2,q_2}^{s_2,\alpha};
\end{equation}
(ii) if $q_1>q_2$ and
$s_1>s_2+n\alpha\big(\frac{1}{p_1}-\frac{1}{p_2}\big)+n(1-\alpha)\big(\frac{1}{q_2}-\frac{1}{q_1}\big)$,
then
\begin{equation}\label{embede-2}
M_{p_1,q_1}^{s_1,\alpha}\subset M_{p_2,q_2}^{s_2,\alpha}.
\end{equation}
\end{prop}

\begin{proof}
From Bernstein's inequality it follows that
\begin{equation}\label{bernstein-alpha}
\|\square_k^{\alpha}f\|_{p_2}\lesssim\langle
k\rangle^{\frac{n\alpha}{1-\alpha}\big(\frac{1}{p_1}-\frac{1}{p_2}\big)}\|\square_k^{\alpha}f\|_{p_1}.
\end{equation}
Then (i) follows from $\ell^{q_1}\subset\ell^{q_2}$ and
\eqref{bernstein-alpha}. For (ii), we use H\"{o}lder's
inequality to obtain
\begin{equation*}
\begin{split}
\|f\|_{M_{p_2,q_2}^{s_2,\alpha}}&\lesssim\|\{\square_k^{\alpha}f\}_{k\in\mathbb{Z}^n}\|_{\ell_{s_1,\alpha}^{q_1}}\|\{1\}_{k\in\mathbb{Z}^n}\|_{\ell_{s_2-s_1+n\alpha(1/p_1-1/p_2),\alpha}^{q_1q_2/(q_1-q_2)}}.
\end{split}
\end{equation*}
For the second term in the right-hand side, we easily see that it
is finite by changing the summation to an integration.
\end{proof}

\begin{prop}[{\bf Completeness}]
(i) $M_{p,q}^{s,\alpha}$ is a quasi-Banach space, and is a Banach
space if $1\leqslant p\leqslant\infty$ and $1\leqslant
q\leqslant\infty$. \\
(ii) We have
\begin{equation}\label{complete}
\mathscr{S}(\mathbb{R}^n)\subset
M_{p,q}^{s,\alpha}(\mathbb{R}^n)\subset\mathscr{S}'(\mathbb{R}^n).
\end{equation}
Moreover, if $0<p,q<\infty$, then $\mathscr{S}(\mathbb{R}^n)$ is
dense in $M_{p,q}^{s,\alpha}$.
\end{prop}

\begin{proof}
 See \cite{BoNi06a}.
\end{proof}

\begin{prop}[{\bf Isomorphism}]\label{isom}
For any $\sigma\in\mathbb{R}$, the mapping
$J^{\sigma}:M_{p,q}^{s,\alpha}\rightarrow M_{p,q}^{s-\sigma,\alpha}$
is isomorphic.
\end{prop}

\begin{proof}
 See \cite{triebel}
\end{proof}

\begin{prop}\label{alpha=2-equivalent-besov}
$M_{2,2}^{s,\alpha}(\mathbb{R}^n)=H^s(\mathbb{R}^n)$  with equivalent norms.
\end{prop}

\begin{proof}
Plancherel's identity  implies the result, see \cite{Gb}.
\end{proof}

\subsection{Duality}

It is known that the dual space of Besov space $B^s_{p,q}$ is $B^{-s+n(1/(p\wedge 1)-1) }_{(p\vee 1)^*, (q\vee 1)^*}$ (see \cite{triebel}) and the dual space of modulation space $M^s_{p,q}$  is $M^{-s }_{(p\vee 1)^*, (q\vee 1)^*}$ (see \cite{WaHe07}). In this section we study the dual spaces of
$\alpha$-modulation spaces.

\begin{prop}\label{duality-sequence}
Suppose $1\leqslant p,q<\infty$. Then we have
\begin{equation*}
\big(\ell_{s,\alpha}^q(\mathbb{Z}^n;L^p)\big)^*=\ell_{-s,\alpha}^{q^*}(\mathbb{Z}^n;L^{p^*}).
\end{equation*}
More precisely,
$f\in\big(\ell_{s,\alpha}^q(\mathbb{Z}^n;L^p)\big)^*$ is equivalent
to that there exists a sequence
$\{f_k\}_{k\in\mathbb{Z}^n}\in\ell_{-s,\alpha}^{q^*}(\mathbb{Z}^n;L^{p^*})$
such that for any
$g=\{g_k\}_{k\in\mathbb{Z}^n}\in\ell_{s,\alpha}^q(\mathbb{Z}^n;L^p)$,
we have
\begin{equation*}
\langle f,g \rangle=\sum_{k\in\mathbb{Z}^n}\int_{\mathbb{R}^n}f_k(x)\overline{g_k(x)}dx,
\end{equation*}
with
$\|f\|_{(\ell_{s,\alpha}^q(\mathbb{Z}^n;L^p))^*}=\|\{f_k\}\|_{\ell_{-s,\alpha}^{q^*}(\mathbb{Z}^n;L^{p^*})}$.
\end{prop}
\noindent It is a direct consequence of Proposition 3.3 in \cite{WaHe07}.

\begin{lem}\label{lemma-sequence}
Let
$\{g_k\}_{k\in\mathbb{Z}^n}\in\ell_{s,\alpha}^q(\mathbb{Z}^n;L^p)$,
and $\Gamma$ be a subset of $\mathbb{Z}^n$, then we have
\begin{equation}\label{gggggggggggggggggllllllll}
\left\|\sum_{k\in\Gamma}\square_k^{\alpha}g_k\right\|_{M_{p,q}^{s,\alpha}}\lesssim
\|\{\square_k^{\alpha}g_k\}_{k\in\Gamma}\|_{\ell_{s,\alpha}^q(\mathbb{Z}^n;L^p)}.
\end{equation}
\end{lem}
\begin{proof}
We introduce
\begin{equation}
 {\Lambda}(k)=\{l\in\mathbb{Z}^n:\square_k^{\alpha} {\square}_l^{\alpha}f\neq 0 \}.
\end{equation}
We denote the constant in \eqref{eta-2} relating to
$\{ {\eta}_l^{\alpha}\}_{l\in\mathbb{Z}^n}$ by
$ {C}$, thus for every $l\in \Lambda(k)$, there holds
\begin{subequations}\label{quantitative-k-l}
\begin{align}
\langle k\rangle^{\frac{\alpha}{1-\alpha}}(k_j-C) &<\langle
l\rangle^{\frac{\alpha}{1-\alpha}}(l_j+ {C}),  \label{quantitative-k-l-1}\\
\langle k\rangle^{\frac{\alpha}{1-\alpha}}(k_j+C) &> \langle
l\rangle^{\frac{\alpha}{1-\alpha}}(l_j- {C})
\label{quantitative-k-l-2}
\end{align}
\end{subequations}
with $j=1,\cdots,n$. For the above $l$ and $k$, we conclude  that
\begin{equation}\label{quantitative-k-l-3333}
\langle k\rangle\sim\langle l\rangle.
\end{equation}
If $|k|\lesssim1 \; (\mbox{or}\; |l|\lesssim1)$, it is easy to see that
\eqref{quantitative-k-l-3333} follows from \eqref{quantitative-k-l}. Thus, it suffices to show  \eqref{quantitative-k-l-3333} in the case
$|k|\gg1 \; (\mbox{or}\;|l|\gg1)$. When $k\in\mathbb{R}^n_j$ with $k_j>0$,
from \eqref{quantitative-k-l-1}; whereas when $k\in\mathbb{R}^n_j$ but
with $k_j<0$, from \eqref{quantitative-k-l-2}$\times(-1)$, we see
$\langle k\rangle^{\frac{1}{1-\alpha}}\lesssim\langle
l\rangle^{\frac{1}{1-\alpha}}$, and symmetrically, we have $\langle
l\rangle^{\frac{1}{1-\alpha}}\lesssim\langle
k\rangle^{\frac{1}{1-\alpha}}$. Therefore, we get
\eqref{quantitative-k-l-3333}. Suppose both $l$ and $\widetilde{l}$
are in $\Lambda(k)$.   substituting  $l$ with
$\widetilde{l}$
(\ref{quantitative-k-l-1}) and (\ref{quantitative-k-l-2}) also hold.
It
follows that
\begin{equation*}
\left|\langle
l\rangle^{\frac{\alpha}{1-\alpha}}l_j-\langle\widetilde{l}\rangle^{\frac{\alpha}{1-\alpha}}\widetilde{l}_j\right|\lesssim\langle
k\rangle^{\frac{\alpha}{1-\alpha}}+\langle
l\rangle^{\frac{\alpha}{1-\alpha}}+\langle\widetilde{l}\rangle^{\frac{\alpha}{1-\alpha}}.
\end{equation*}
Then Taylor's theorem, combined with \eqref{quantitative-k-l-3333},
gives $|l_j-\widetilde{l}_j|\lesssim1$. It follows that
\begin{equation}\label{cardinal}
\#  {\Lambda}(k)\sim 1.
\end{equation}

 One has that the right hand side of
\eqref{gggggggggggggggggllllllll} is
\begin{equation}\label{vvvvvvvvvvvvvvvvvvvvvvvvvvvvv}
\begin{split}
\left(\sum_{l\in\mathbb{Z}^n}\langle
l\rangle^{\frac{sq}{1-\alpha}}\left\|\sum_{k\in\Lambda(l)\cap\Gamma}\square_l^{\alpha}\square_k^{\alpha}g_k\right\|_p^q\right)^{\frac1q}
&\lesssim\left(\sum_{l\in\mathbb{Z}^n}\langle
l\rangle^{\frac{sq}{1-\alpha}}\sum_{k\in\Lambda(l)\cap\Gamma}\|\square_l^{\alpha}\square_k^{\alpha}g_k\|_p^q\right)^{\frac1q}
\\
&\lesssim\left(\sum_{k\in\Gamma}\|\square_k^{\alpha}g_k\|_p^q\sum_{l\in\Lambda(k)}\langle
l\rangle^{\frac{sq}{1-\alpha}}\right)^{\frac1q}
\\
&\lesssim\left(\sum_{k\in\Gamma}\langle
k\rangle^{\frac{sq}{1-\alpha}}\|\square_k^{\alpha}g_k\|_p^q\right)^{\frac1q}.
\end{split}
\end{equation}
We remark that in the second inequality of \eqref{vvvvvvvvvvvvvvvvvvvvvvvvvvvvv}, by Young's inequality, or by Proposition \ref{convolution} we see that
$$
\|\square_k^{\alpha} {\square}_l^{\alpha}f\|_p\lesssim\| {\square}_k^{\alpha}f\|_p,
$$
which enable us to remove
$\square_l^{\alpha}$; and in the third inequality of
\eqref{vvvvvvvvvvvvvvvvvvvvvvvvvvvvv}, we have applied
\eqref{cardinal} to remove the summation on $l\in \Lambda (k)$.
\end{proof}

\begin{thm}\label{duality}
Suppose $0<p,q<\infty$, then we have
\begin{equation}
\big(M_{p,q}^{s,\alpha}\big)^*=M_{(1\vee p)^*,(1\vee
q)^*}^{-s+n\alpha\big(\frac{1}{1\wedge p}-1\big)}.
\end{equation}
\end{thm}

\begin{proof} The proof is separated into four cases.

\noindent{\bf Case 1:}\;$\boldsymbol{1\leqslant p,q<\infty}$. First,  we show that
$ \big(M_{p,q}^{s,\alpha}\big)^* \subset M_{p^*,q^*}^{-s,\alpha}$.  Noticing that
$$
M_{p,q}^{s,\alpha} \ni f \to \{ \Box^\alpha_k f\}  \in \ell_{s,\alpha}^q(\mathbb{Z}^n;L^p)
$$
is an isometric mapping from $M_{p,q}^{s,\alpha}$ onto a subspace $X=\{\{\Box^\alpha_k f\}: \ f\in M_{p,q}^{s,\alpha}\}$ of $\ell_{s,\alpha}^q(\mathbb{Z}^n;L^p)$, so, any continuous functional $g$ on $M_{p,q}^{s,\alpha}$ can be regarded as a bounded linear functional on $X$, which can be extended onto $\ell_{s,\alpha}^q(\mathbb{Z}^n;L^p)$ (the extension is  written as $\tilde{g}$) and the norm of $g$ is preserved. By Proposition \ref{duality-sequence}, there exists $ \{g_k\} \in\ell_{-s,\alpha}^{q^*}(\mathbb{Z}^n;L^{p^*})$ such that
\begin{align} \label{funct}
\langle \tilde{g}, \{f_k\} \rangle = \sum_{k\in \mathbb{Z}^n} \int \overline{g_k(x)} f_k(x) dx
\end{align}
holds for all $ \{f_k\} \in \ell_{s,\alpha}^q(\mathbb{Z}^n;L^p)$. Moreover,  $\|g\|_{(M_{p,q}^{s,\alpha})^*} = \|\{g_k\}\|_{\ell_{-s,\alpha}^{q^*}(\mathbb{Z}^n;L^{p^*})}$. Since $M^{s,\alpha}_{p,q}$ is isometric to $X$, we see that
\begin{align} \label{functal}
\langle g, \varphi \rangle = \langle  \tilde{g}, \{\Box^\alpha_k  \varphi \} \rangle = \int \sum_{k\in \mathbb{Z}^n} \overline{\Box^\alpha_k  g_k(x)} \varphi(x) dx, \quad \varphi \in  M_{p,q}^{s,\alpha},
\end{align}
Hence,
$g= \sum_{k\in \mathbb{Z}^n} \Box^\alpha_k g_k(x)$.  In view of Lemma \ref{lemma-sequence} and Young's inequality,
$$
\|g\|_{M_{p^*,q^*}^{-s,\alpha}} \lesssim   \|\{ g_k\}_{k\in\mathbb{Z}^n}\|_{\ell_{-s,\alpha}^{q^*}(\mathbb{Z}^n;L^{p^*})}= \|g\|_{ (M_{p,q}^{s,\alpha} )^*},
$$
which implies $\big(M_{p,q}^{s,\alpha}\big)^*\subset
M_{p^*,q^*}^{-s,\alpha}$. Next, we prove the reverse inclusion. For any $f\in M_{p^*,q^*}^{-s,\alpha} \subset \mathscr{S}'$, we show that
 $f\in\big(M_{p,q}^{s,\alpha}\big)^*$. Let $\varphi \in \mathscr{S}$.  We have
\begin{equation*}\label{duality-1-1}
\begin{split}
|\langle f,\varphi\rangle |
&\leqslant\|\{\square_k^{\alpha}f\}_{k\in\mathbb{Z}^n}\|_{\ell_{-s,\alpha}^{q^*}(\mathbb{Z}^n;L^{p^*})} \left\|\left\{\sum_{l\in\Lambda(k)}\square_l^{\alpha} \varphi\right\}_{k\in\mathbb{Z}^n}\right\|_{\ell_{s,\alpha}^q(\mathbb{Z}^n;L^p)}
\\
&\lesssim\|f\|_{M_{p^*,q^*}^{-s,\alpha}}\|\varphi\|_{M_{p,q}^{s,\alpha}}.
\end{split}
\end{equation*}
The principle of duality implies
$M_{p^*,q^*}^{-s,\alpha}\subset\big(M_{p,q}^{s,\alpha}\big)^*$.

In the following, we discuss the left three cases. From
\eqref{embede-1} in Proposition \ref{embede}, we know
$$
M_{p,q}^{s,\alpha}\subset M_{1\vee p,1\vee
q}^{s-n\alpha\big(\frac{1}{p\wedge 1} -1\big),\alpha}.
$$
This combined with the principle of duality gives
$$
\big(M_{p,q}^{s,\alpha}\big)^*\supset\left(M_{1\vee p,1\vee
q}^{s-n\alpha\big(\frac{1}{p\wedge 1}-1\big),\alpha}\right)^*=M_{(1\vee
p)^*,(1\vee q)^*}^{-s+n\alpha\big(\frac{1}{p\wedge 1}-1\big),\alpha}.
$$
Hence, only the reverse inclusion  needs to be proven.

\noindent{\bf Case 2.}\;$\boldsymbol{1\leqslant p<\infty,0<q<1.}$ For any
$f\in\big(M_{p,q}^{s,\alpha}\big)^*$, take any $k\in\mathbb{Z}^n$
and any $\varphi\in\mathscr{S}(\mathbb{R}^n)$, we have
\begin{equation}
\begin{split}
|\langle \square_k^{\alpha}f,\varphi \rangle| &=|\langle f,\square_k^{\alpha}\varphi \rangle| \\
&\leqslant\|f\|_{\big(M_{p,q}^{s,\alpha}\big)^*}\|\square_k^{\alpha}\varphi\|_{M_{p,q}^{s,\alpha}}
\\
&\lesssim\langle
k\rangle^{\frac{s}{1-\alpha}}\|f\|_{\big(M_{p,q}^{s,\alpha}\big)^*}\|\varphi\|_p.
\end{split}
\end{equation}
This implies $\big(M_{p,q}^{s,\alpha}\big)^*\subset
M_{p^*,\infty}^{-s,\alpha}$.

\noindent{\bf Case 3.}\;$\boldsymbol{0<p,q<1.}$ For any
$f\in\big(M_{p,q}^{s,\alpha}\big)^*$, take any $k\in\mathbb{Z}^n$,
we have
\begin{equation}
\begin{split}
|\square_k^{\alpha}f(x)|
&=|\langle f,\overline{\mathscr{F}^{-1}\eta_k^{\alpha}(x-\cdot)}\rangle |
\\
&=|\langle f,\mathscr{F}^{-1}\eta_k^{\alpha}(\cdot-x) \rangle| \\
&\lesssim\|f\|_{\big(M_{p,q}^{s,\alpha}\big)^*}\|\mathscr{F}^{-1}\eta_k^{\alpha}(\cdot-x)\|_{M_{p,q}^{s,\alpha}}
\\
&\lesssim\langle
k\rangle^{\frac{s}{1-\alpha}-\frac{n\alpha}{1-\alpha}\big(\frac1p-1\big)}\|f\|_{\big(M_{p,q}^{s,\alpha}\big)^*}.
\end{split}
\end{equation}
This implies $\big(M_{p,q}^{s,\alpha}\big)^*\subset
M_{\infty,\infty}^{-s+n\alpha\big(\frac1p-1\big),\alpha}$.

\noindent{\bf Case 4.}\;$\boldsymbol{0<p<1,1\leqslant q<\infty}$. For any
$f\in\big(M_{p,q}^{s,\alpha}\big)^*$ and every $k\in\mathbb{Z}^n$,
there exists some $x_k\in\mathbb{R}^n$ satisfying
$$
\|\square_k^{\alpha}f\|_{\infty}\sim|\mathscr{F}^{-1}\eta_k^{\alpha}\mathscr{F}f(x_k)|.
$$
Let $\{a_k\}$ be an arbitrary $\ell_{s-n\alpha(1/p-1)}^q$ sequence,
and we construct another sequence namely $\{\widetilde{a}_k\}$, such
that $|\widetilde{a}_k|=|a_k|$, and the argument of
$\widetilde{a}_k$ is the opposite number of the principal argument
of $\mathscr{F}^{-1}\eta_k^{\alpha}\mathscr{F}f(x_k)$. From
\eqref{eta-2},\eqref{eta-4}, we get
$\|\mathscr{F}^{-1}\eta_k^{\alpha}(\cdot-x_k)\|_p\lesssim\langle
k\rangle^{-\frac{n\alpha}{1-\alpha}\big(\frac1p-1\big)}$. Therefore,
we have
\begin{equation*}
\begin{split}
\sum_{k\in\mathbb{Z}^n}|a_k|\|\square_k^{\alpha}f\|_{\infty} &\sim
\left<f,\sum_k\widetilde{a}_k\mathscr{F}^{-1}\eta_k^{\alpha}(\cdot-x_k)\right>
\\
&\leqslant\|f\|_{\big(M_{p,q}^{s,\alpha}\big)^*}\left\|\sum_k\widetilde{a}_k\mathscr{F}^{-1}\eta_k^{\alpha}(\cdot-x_k)\right\|_{M_{p,q}^{s,\alpha}}
\\
&\lesssim\|f\|_{\big(M_{p,q}^{s,\alpha}\big)^*}\|\{a_k\}_{k\in\mathbb{Z}^n}\|_{\ell_{s-n\alpha(\frac1p-1)}^q}.
\end{split}
\end{equation*}
This implies $\big(M_{p,q}^{s,\alpha}\big)^*\subset
M_{\infty,q^*}^{-s+n\alpha\big(\frac1p-1\big),\alpha}$.
\end{proof}

\subsection{Complex interpolation}

The complex interpolation for  Besov spaces has a beautiful  theory; cf. \cite{triebel}.  We can imitate the counterpart for the Besov space to construct the
complex interpolation for  $\alpha$-modulation spaces. It
will be repeatedly used in the following argument.
Since there is little essential modification in the statement, we
only provide the outline of the proof.

We start with some abstract theory about complex interpolation on
quasi-Banach spaces. Let $S=\{z:0<\mbox{Re} z<1\}$ be a strip in the
complex plane. Its closure $\{z:0\leqslant\mbox{Re} z\leqslant1\}$
is denoted by $\overline{S}$. We say that $f(z)$ is an
$\mathscr{S}'(\mathbb{R}^n)$-analytic function in $S$ if the
following properties are satisfied: \\
(i) for every fixed $z\in\overline{S}$,
$f(z)\in\mathscr{S}'(\mathbb{R}^n)$; \\
(ii) for any $\varphi\in\mathscr{S}(\mathbb{R}^n)$ with compact
support, $\mathscr{F}^{-1}\varphi\mathscr{F}f(x,z)$ is a uniformly
continuous and bounded function in $\mathbb{R}^n\times\overline{S}$;
\\
(iii) for any $\varphi\in\mathscr{S}(\mathbb{R}^n)$ with compact
support, $\mathscr{F}^{-1}\varphi\mathscr{F}f(x,z)$ is an analytic
function in $S$ for every fixed $x\in\mathbb{R}^n$. \\
We denote the set of all $\mathscr{S}'(\mathbb{R}^n)$-analytic
functions in $S$ by $\mathbf{A}(\mathscr{S'}(\mathbb{R}^n))$. The
idea we used here is due to   Calder\'{o}n \cite{Cal}, Calder\'{o}n and Torchinsky
\cite{CalTor,CalTor2} and Triebel \cite{triebel}.

\begin{defin}\label{def-complex}
Let $A_0$ and $A_1$ be quasi-Banach spaces, and $0<\theta<1$. We
define
\begin{equation}
\begin{split}
\mathbf{F}(A_0,A_1)=&\Big\{\varphi(z)\in\mathbf{A}(\mathscr{S}'(\mathbb{R}^n)):\varphi(\ell+\mbox{i}t)\in
A_\ell,\ell=0,1,\forall t\in\mathbb{R},
\\
&\|\varphi(z)\|_{\mathbf{F}(A_0,A_1)}\stackrel{\triangle}{=}\max_{\ell=0,1}\sup_{t\in\mathbb{R}}\|\varphi(\ell+\mbox{i}t)\|_{A_{\ell}}\Big\};
\end{split}
\end{equation}
and
\begin{equation}\label{inter-space}
\begin{split}
(A_0,A_1)_{\theta}=&\Big\{f\in\mathscr{S}'(\mathbb{R}^n):\exists
\varphi(z)\in\mathbf{F}(A_0,A_1)\; \mbox{such
that}\;f=\varphi(\theta),
\\
&\|f\|_{(A_0,A_1)_{\theta}}\stackrel{\triangle}{=}\inf_{\varphi}\|\varphi(z)\|_{\mathbf{F}(A_0,A_1)}\Big\},
\end{split}
\end{equation}
where the infimum is taken over all
$\varphi(z)\in\mathbf{F}(A_0,A_1)$ such that $\varphi(\theta)=f$.
\end{defin}

\noindent The following two propositions are essentially known in
\cite{triebel} and the references therein.
\begin{prop}
Suppose all notations have the same meaning as in Definition
\ref{def-complex}, then we have
$$
\big((A_0,A_1)_{\theta},\|\cdot\|_{(A_0,A_1)_{\theta}}\big)
$$
is a quasi-Banach space.
\end{prop}

\begin{prop}\label{calcul-interp-norm}
Suppose all notations have the same meaning as in Definition
\ref{def-complex}, then we have
\begin{equation}
\|f\|_{(A_0,A_1)_{\theta}}=\inf_{\varphi}\left(\sup_{t\in\mathbb{R}}\|\varphi(\mbox{i}t)\|_{A_0}^{1-\theta}\sup_{t\in\mathbb{R}}\|\varphi(1+\mbox{i}t)\|_{A_1}^{\theta}\right),
\end{equation}
where the infimum is taken over all
$\varphi(z)\in\mathbf{F}(A_0,A_1)$ such that $\varphi(\theta)=f$.
\end{prop}

We point out the interpolation functor referred in
\eqref{inter-space} is an exact interpolation functor of exponent
$\theta$. For our purpose, we will use the following multi-linear case.

\begin{prop}\label{boundness}
Let $T$ be a continuous multi-linear operator from $A_0^{(1)}\times
A_0^{(2)}\times\cdots\times A_0^{(m)}$ to $B_0$ and from
$A_1^{(1)}\times A_1^{(2)}\times\cdots\times A_1^{(m)}$ to $B_1$,
satisfying
\begin{equation*}
\begin{aligned}
\big\|T\big(f^{(1)},f^{(2)},\cdots,f^{(m)}\big)\big\|_{B_0}
\leqslant C_0\prod_{j=1}^m\big\|f^{(j)}\big\|_{A_0^{(j)}}; \\
\big\|T\big(f^{(1)},f^{(2)},\cdots,f^{(m)}\big)\big\|_{B_1}
\leqslant C_1\prod_{j=1}^m\big\|f^{(1)}\big\|_{A_1^{(1)}},
\end{aligned}
\qquad f^{(j)}\in A_0^{(j)}\cap A_1^{(j)}.
\end{equation*}
Then $T$ is continuous from
$\big(A_0^{(1)},A_1^{(1)}\big)_{\theta}\times\big(A_0^{(2)},A_1^{(2)}\big)_{\theta}\times\cdots\times\big(A_0^{(m)},A_1^{(m)}\big)_{\theta}$
to $(B_0,B_1)_{\theta}$ with norm at most
$C_0^{1-\theta}C_1^{\theta}$, provided $0\leqslant\theta\leqslant1$.
\end{prop}

\begin{proof}
From Proposition \ref{calcul-interp-norm}, we know there exist $m$
sequences $\{\varphi_k^{(j)}(z)\}_{k\in\mathbb{N}}, j=1,\cdots,m$
satisfying
\begin{equation}\label{interp-functor-1}
\lim_{k\to\infty}\sup_t\|\varphi_k^{(j)}(\mbox{i}t)\|_{A_0}^{1-\theta}\sup_t\|\varphi_k^{(j)}(1+\mbox{i}t)\|_{A_1}^{\theta}=\|f^{(j)}\|_{(A_0,A_1)_{\theta}}.
\end{equation}
We put
$\psi_k^{(j)}(z)=C_0^{\frac{z-1}{m}}C_1^{-\frac{z}{m}}T(\varphi_k^{(j)}(z))$.
It is easy to see that $\psi_k^{(j)}(z)\in F(B_0,B_1)$ with
$\psi_k^{(j)}(\theta)=C_0^{\theta-1}C_1^{-\theta}Tf^{(j)}$, and
\begin{equation*}
\|\psi_k^{(j)}(\ell+\mbox{i}t)\|_{B_{\ell}}\leqslant
C_{\ell}^{-\frac1m}\|T\varphi_k^{(j)}(\ell+\mbox{i}t)\|_{B_{\ell}}\leqslant\|\varphi_k^{(j)}(\ell+\mbox{i}t)\|_{A_{\ell}},
\quad\ell=0,1.
\end{equation*}
Thus, combining Proposition \ref{calcul-interp-norm}, we have
\begin{equation}\label{interp-functor-2}
\begin{split}
\big\|T\big(f^{(1)},\cdots,f^{(m)}\big)\big\|_{(B_0,B_1)_{\theta}}&=C_0^{1-\theta}C_1^{\theta}\prod_{j=1}^m\|\psi_k^{(j)}(\theta)\|_{(B_0,B_1)_{\theta}}
\\
&\leqslant
C_0^{1-\theta}C_1^{\theta}\prod_{j=1}^m\left(\sup_t\|\psi_k^{(j)}(\mbox{i}t)\|_{B_0}^{1-\theta}\sup_t\|\psi_k^{(j)}(1+\mbox{i}t)\|_{B_1}^{\theta}\right)
\\
&\leqslant
C_0^{1-\theta}C_1^{\theta}\prod_{j=1}^m\left(\sup_t\|\varphi_k^{(j)}(\mbox{i}t)\|_{A_0}^{1-\theta}\sup_t\|\varphi_k^{(j)}(1+\mbox{i}t)\|_{B_1}^{\theta}\right)
\end{split}
\end{equation}
The conclusion follows from
\eqref{interp-functor-2},\eqref{interp-functor-1}.
\end{proof}

\begin{thm}\label{interpolation-M}
Suppose $0<\theta<1$ and
\begin{equation}
s=(1-\theta)s_0+\theta s_1, \quad
\frac1p=\frac{1-\theta}{p_0}+\frac{\theta}{p_1}, \quad
\frac1q=\frac{1-\theta}{q_0}+\frac{\theta}{q_1},
\end{equation}
then we have
\begin{equation}
\big(M_{p_0,q_0}^{s_0,\alpha},M_{p_1,q_1}^{s_1,\alpha}\big)_{\theta}=M_{p,q}^{s,\alpha}.
\end{equation}
\end{thm}

\begin{proof} [Sketch of Proof.]
For $z\in\overline{S}$, we write
$$
s(z)=(1-z)s_0+zs_1, \quad
\frac{1}{p(z)}=\frac{1-z}{p_0}+\frac{z}{p_1}, \quad
\frac{1}{q(z)}=\frac{1-z}{q_0}+\frac{z}{q_1}.
$$
For any $f\in M_{p,q}^{s,\alpha}$, we set
\begin{equation*}
\varphi(x,z)=\sum_{k\in\mathbb{Z}^n}\langle
k\rangle^{\frac{1}{1-\alpha}\big[\frac{sq}{q(z)}-s(z)\big]}\|\square_k^{\alpha}f\|_p^{\frac{q}{q(z)}-\frac{p}{p(z)}}(\square_k^{\alpha}f)^{\frac{p}{p(z)}}(x).
\end{equation*}
Obviously, $\varphi(z)\in\mathbf{A}(\mathscr{S}'(\mathbb{R}^n))$ and
$\varphi(\theta)=f$. Direct calculation shows
$$
\|\varphi(\ell+\mbox{i}t)\|_{M_{p_{\ell},q_{\ell}}^{s_{\ell},\alpha}}\lesssim\|f\|_{M_{p,q}^{s,\alpha}},
\quad \ell=0,1.
$$
This proves that
$
M_{p,q}^{s,\alpha}\subset\big(M_{p_0,q_0}^{s_0,\alpha},M_{p_1,q_1}^{s_1,\alpha}\big)_{\theta}.
$

Conversely, for any
$f\in\big(M_{p_0,q_0}^{s_0,\alpha},M_{p_1,q_1}^{s_1,\alpha}\big)_{\theta}$,
if $\varphi\in\mathbf{A}(\mathscr{S}'(\mathbb{R}^n))$ such that
$\varphi(\theta)=f$, for some $\theta\in(0,1)$, we can find two
positive functions $\mu_0(\theta,t)$ and $\mu_1(\theta,t)$ in
$(0,1)\times\mathbb{R}$ satisfying
\begin{equation*}
\|\square_k^{\alpha}f\|_p\leqslant\left(\frac{1}{1-\theta}\int_{\mathbb{R}}\|\square_k^{\alpha}\varphi(\mbox{i}t)\|_{p_0}^r\mu_0(\theta,t)dt\right)^{\frac{1-\theta}{r}}
\left(\frac{1}{\theta}\int_{\mathbb{R}}\|\square_k^{\alpha}\varphi(1+\mbox{i}t)\|_{p_1}^r\mu_1(\theta,t)dt\right)^{\frac{\theta}{r}},
\end{equation*}
with
$\frac{1}{1-\theta}\int_{\mathbb{R}}\mu_0(\theta,t)dt=\frac{1}{\theta}\int_{\mathbb{R}}\mu_1(\theta,t)dt=1$.
Taking the $\ell_{s,\alpha}^q$ norm of both sides leads to
\begin{align}
\|\{\square_k^{\alpha}f\}_{k\in\mathbb{Z}^n}\|_{\ell_{s,\alpha}^q} & \leqslant\left\|\frac{1}{1-\theta}\int_{\mathbb{R}}\langle
k\rangle^{\frac{s_0r}{1-\alpha}}\|\square_k^{\alpha}\varphi(\mbox{i}t)\|_{p_0}^r\mu_0(\theta,t)dt\right\|_{\ell^{\frac{q_0}{r}}}^{\frac{1-\theta}{r}}
\nonumber \\
& \quad \quad \times\left\|\frac{1}{\theta}\int_{\mathbb{R}}\langle
k\rangle^{\frac{s_1r}{1-\alpha}}\|\square_k^{\alpha}\varphi(1+\mbox{i}t)\|_{p_1}^r\mu_1(\theta,t)dt\right\|_{\ell^{\frac{q_1}{r}}}^{\frac{\theta}{r}}.
\end{align}
Then, Minkowski's inequality implies that
\begin{equation*}
\|f\|_{M_{p,q}^{s,\alpha}}\lesssim\sup_t\|\varphi(\mbox{i}t)\|_{M_{p_0,q_0}^{s_0,\alpha}}\sup_t\|\varphi(1+\mbox{i}t)\|_{M_{p_1,q_1}^{s_1,\alpha}}.
\end{equation*}
This proves $
\big(M_{p_0,q_0}^{s_0,\alpha},M_{p_1,q_1}^{s_1,\alpha}\big)_{\theta}\subset
M_{p,q}^{s,\alpha}. $
\end{proof}

~~~~

\noindent The following is a natural consequence of Proposition
\ref{boundness} and Theorem \ref{interpolation-M}, and is frequently
used later on.
\begin{cor}\label{interpolation}
Suppose $T$ is a continuous multi-linear mapping from
$M_{p_0^{(1)},q_0^{(1)}}^{s_0^{(1)},\alpha}\times\cdots\times
M_{p_0^{(m)},q_0^{(m)}}^{s_0^{(m)},\alpha}$ to
$M_{p_0,q_0}^{s_0,\alpha}$ with norm $M_0$, and is also continuous,
multi-linear from
$M_{p_1^{(1)},q_1^{(1)}}^{s_1^{(1)},\alpha}\times\cdots\times
M_{p_1^{(m)},q_1^{(m)}}^{s_1^{(m)},\alpha}$ to
$M_{p_1,q_1}^{s_1,\alpha}$ with norm $M_1$. Then $T$ is continuous
and multi-linear from
$M_{p^{(1)},q^{(1)}}^{s^{(1)},\alpha}\times\cdots\times
M_{p^{(m)},q^{(m)}}^{s^{(m)},\alpha}$ to $M_{p,q}^{s,\alpha}$ with
norm at most $M_0^{1-\theta}M_1^{\theta}$, provided
$0\leqslant\theta\leqslant1$, and
\begin{equation*}
s^{(j)}=(1-\theta)s_0^{(j)}+\theta s_1^{(j)}, \quad
\frac{1}{p^{(j)}}=\frac{1-\theta}{p_0^{(j)}}+\frac{\theta}{p_1^{(j)}},
\quad
\frac{1}{q^{(j)}}=\frac{1-\theta}{q_0^{(j)}}+\frac{\theta}{q_1^{(j)}},
\quad j=1\cdots m.
\end{equation*}
\end{cor}

\section{Scaling property}

For Besov space, it is well known that
\begin{equation}\label{dilation-besov}
\|f_{\lambda}\|_{B_{p,q}^s}\lesssim\lambda^{-\frac{n}{p}}(1\vee\lambda^s)\|f\|_{B_{p,q}^s}.
\end{equation}
For modulation spaces with $s=0$ and $1\le p,q \le \infty$, the sharp
dilation property was obtained in \cite{ST} and they showed
\begin{align}
& \|f_{\lambda}\|_{M_{p,q}^0}\lesssim \lambda^{-\frac{n}{p}} \lambda^{0 \vee n\big(\frac1q-\frac1p\big) \vee n\big(\frac1p+\frac1q-1\big)} \|f\|_{M_{p,q}^0}, \quad \lambda >1;
\label{dimo} \\
& \|f_{\lambda}\|_{M_{p,q}^0}\lesssim \lambda^{-\frac{n}{p}} \lambda^{- \left[ 0 \vee n\big(\frac1p - \frac1q\big) \vee n\big(1-\frac1p-\frac1q \big) \right] } \|f\|_{M_{p,q}^0}, \
\quad \lambda <1. \label{dimo1}
\end{align}

In this section, we study the scaling property of $\alpha$-modulation spaces. For $0<p,q\leqslant\infty$ and
$(\alpha_1,\alpha_2)\in[0,1]\times[0,1]$, we denote
\begin{equation}\label{notation-R}
R (p,q;\alpha_1,\alpha_2)= 0\vee
\left[n(\alpha_1-\alpha_2)\big(\tfrac1q-\tfrac1p\big)\right]\vee
\left[n(\alpha_1-\alpha_2)\big(\tfrac1p+\tfrac1q-1\big)\right],
\end{equation}
which will be frequently used in this and the next sections. Then,
we divide $\mathbb{R}_+^2$ into 3 sub-domains in two ways (see Fig. \ref{figure-scaling}). One way
is, $\mathbb{R}_+^2=\mbox{S}_1\cup\mbox{S}_2\cup\mbox{S}_3$ with
\begin{equation*}
\begin{split}
\mbox{S}_1
&=\left\{\big(\tfrac1p,\tfrac1q\big)\in\mathbb{R}_+^2:\tfrac1q\geqslant\tfrac1p,\tfrac1p\leqslant\tfrac12\right\};
\\
\mbox{S}_2
&=\left\{\big(\tfrac1p,\tfrac1q\big)\in\mathbb{R}_+^2:\tfrac1p+\tfrac1q\geqslant1,\tfrac1p>\tfrac12\right\};
\\
\mbox{S}_3 &=\mathbb{R}_+^2\backslash\{\mbox{S}_1\cup\mbox{S}_2\},
\end{split}
\end{equation*}
Another way is,
$\mathbb{R}_+^2=\mbox{T}_1\cup\mbox{T}_2\cup\mbox{T}_3$ with
\begin{equation*}
\begin{split}
\mbox{T}_1
&=\left\{\big(\tfrac1p,\tfrac1q\big)\in\mathbb{R}_+^2:\tfrac1p\geqslant\tfrac1q,\tfrac1p>\tfrac12\right\}; \\
\mbox{T}_2
&=\left\{\big(\tfrac1p,\tfrac1q\big)\in\mathbb{R}_+^2:\tfrac1p+\tfrac1q\leqslant1,\tfrac1p\leqslant\tfrac12\right\};
\\
\mbox{T}_3 &=\mathbb{R}_+^2\backslash\{\mbox{T}_1\cup\mbox{T}_2\}.
\end{split}
\end{equation*}
\begin{figure}
\begin{center}
\hspace*{-2.5cm}\includegraphics[scale=0.94]{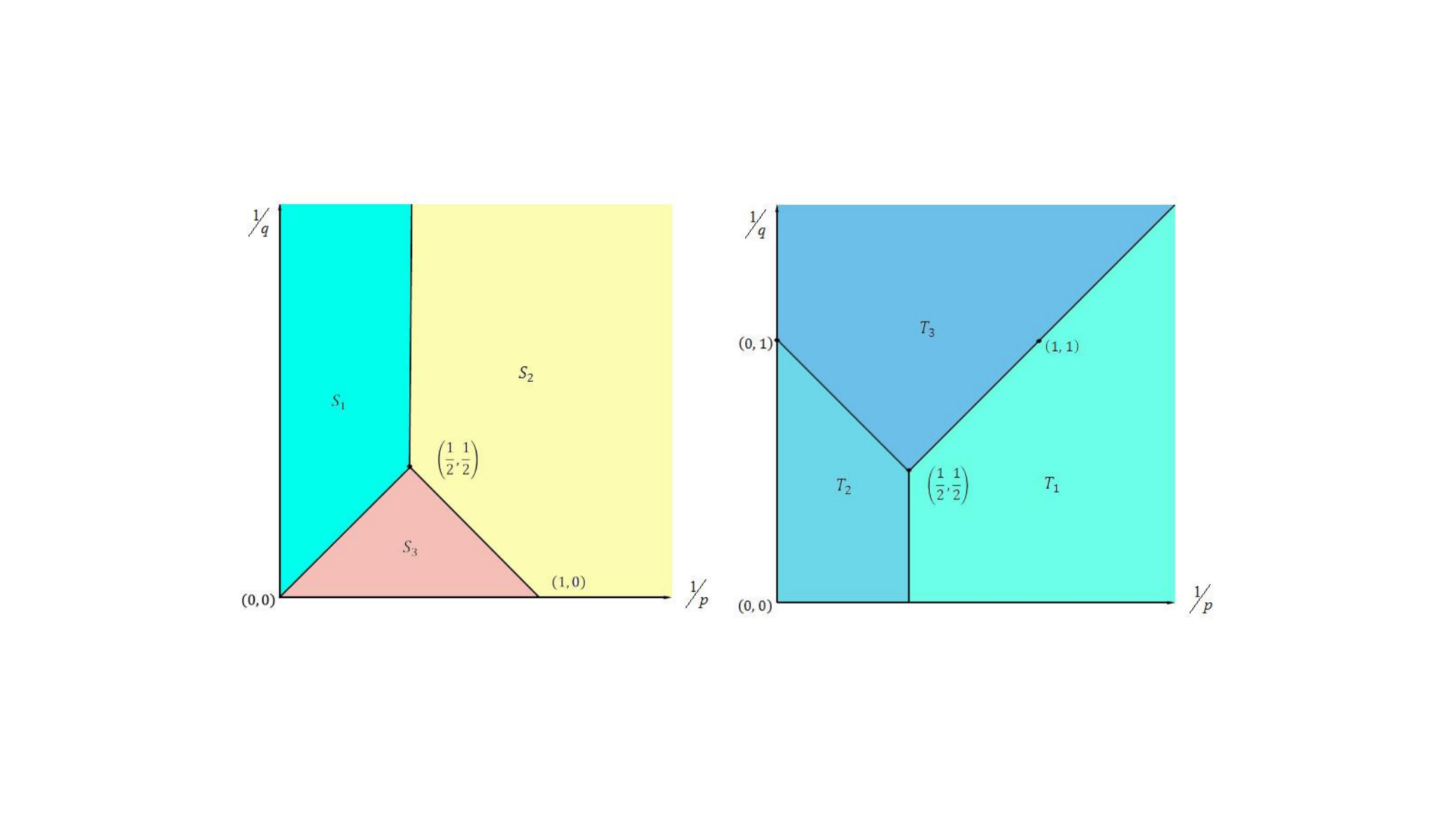}

\vspace*{-3cm}\caption{\small Distribution of $s_c$. The left-hand side figure is for $\lambda>1$, the
right-hand side figure is for $\lambda\leqslant1$. }\label{figure-scaling}
\end{center}
\end{figure}
 If
$\alpha_1\geqslant\alpha_2$, then
\begin{equation}
R (p,q;\alpha_1,\alpha_2)=
\begin{cases}
n(\alpha_1-\alpha_2)\big(\frac1q-\frac1p\big), \quad &
\big(\frac1p,\frac1q\big)\in\mbox{S}_1;
\\
n(\alpha_1-\alpha_2)\big(\frac1p+\frac1q-1\big), \quad &
\big(\frac1p,\frac1q\big)\in\mbox{S}_2;
\\
0, \quad & \big(\frac1p,\frac1q\big)\in\mbox{S}_3.
\end{cases}
\end{equation}
If $\alpha_1<\alpha_2$, then
\begin{equation}
R (p,q;\alpha_1,\alpha_2)=
\begin{cases}
0, \quad & \big(\frac1p,\frac1q\big)\in\mbox{T}_3;
\\
n(\alpha_1-\alpha_2)\big(\frac1p+\frac1q-1\big), \quad &
\big(\frac1p,\frac1q\big)\in\mbox{T}_2; \\
n(\alpha_1-\alpha_2)\big(\frac1q-\frac1p\big), \quad &
\big(\frac1p,\frac1q\big)\in\mbox{T}_1.
\end{cases}
\end{equation}

Before describing the dilation property of the $\alpha$-modulation
spaces, we introduce some critical powers. Let us write $s_p =n(1/(1\wedge p) -1)$ and
\begin{equation}
s_c=
\begin{cases}
 R (p,q;1,\alpha), & \quad
\lambda>1,
\\
-R (p,q;\alpha,1), & \quad \lambda\leqslant 1.
\end{cases}
\end{equation}
\begin{thm}\label{dilationP}
Let $0\leqslant\alpha<1,\lambda>0$ and $s + s_c \neq 0$. Then
\begin{equation}\label{scaling}
\|f_{\lambda}\|_{M_{p,q}^{s,\alpha}}\lesssim\lambda^{-\frac{n}{p}} \left[(1\vee\lambda)^{s_p}\vee\lambda^{s+s_c}\right]\|f\|_{M_{p,q}^{s,\alpha}}.
\end{equation}
holds for all $f\in M_{p,q}^{s,\alpha}$. Conversely, if
\begin{equation*}
\|f_{\lambda}\|_{M_{p,q}^{s,\alpha}}\lesssim\lambda^{-\frac{n}{p}} F(\lambda)\|f\|_{M_{p,q}^{s,\alpha}}
\end{equation*}
holds for some $F: \ (0,\infty)\to (0,\infty)$ and all $f\in M_{p,q}^{s,\alpha}$, then $F(\lambda) \gtrsim (1\vee\lambda)^{s_p}\vee\lambda^{s+s_c}$.
\end{thm}
\begin{proof}({\bf Sufficiency})
We denote by $\square_{k,1/\lambda}^{\alpha}$ the
pseudo-differential operator with symbol
$(\eta_k^{\alpha})_{\lambda}$. For every $l\in\mathbb{Z}^n$ and
$\lambda>0$, we introduce
\begin{equation}
\Lambda(l,\lambda)=\{k\in\mathbb{Z}^n:\square_{k,1/\lambda}^{\alpha}\square_l^{\alpha}f\neq 0\}.
\end{equation}
For any $k\in\Lambda(l,\lambda)$, it follows from \eqref{eta-2} that $k,l$ and $\lambda$ satisfy
\begin{subequations}\label{000}
\begin{align}
\lambda\langle l\rangle^{\frac{\alpha}{1-\alpha}}(l_j-C)<\langle
k\rangle^{\frac{\alpha}{1-\alpha}}(k_j+C); \label{000-1} \\
\lambda\langle l\rangle^{\frac{\alpha}{1-\alpha}}(l_j+C)>\langle
k\rangle^{\frac{\alpha}{1-\alpha}}(k_j-C)  \label{000-2}
\end{align}
\end{subequations}
with $j=1,2,\cdots,n$. In view of \eqref{000}, one sees that $k\in\Lambda(l,\lambda)$ is equivalent to  $l\in\Lambda(k, 1/\lambda)$.  Moreover,  if \eqref{000} holds, then
\begin{align}\label{relation-k-l-small}
  \langle
l\rangle\lesssim1\vee\lambda^{-(1-\alpha)}
\ \ \mbox{if and only if } \ \
\langle k\rangle\lesssim1\vee\lambda^{1-\alpha}.
\end{align}
If $\langle l\rangle\gg1\vee\lambda^{-(1-\alpha)}$, without
loss of generality, we may assume $l$ belongs to some $\mathbb{R}^n_j$,
when $l_j>0$, from \eqref{000-1}; whereas when $l_j<0$, from
\eqref{000-2}$\times(-1)$, we see $\langle
k\rangle^{\frac{1}{1-\alpha}}\gtrsim\lambda\langle
l\rangle^{\frac{1}{1-\alpha}}$. Conversely, for
$k\in\Lambda(l,\lambda)\cap\mathbb{R}^n_j$, also from \eqref{000}, we have
$\langle k\rangle^{\frac{1}{1-\alpha}}\lesssim\lambda\langle
l\rangle^{\frac{1}{1-\alpha}}$. Thus we have
\begin{equation}\label{005}
\langle k\rangle\sim\lambda^{1-\alpha}\langle l\rangle.
\end{equation}
 Since the volumes  of ${\rm supp} (\eta_k^{\alpha})_{\lambda}$ and ${\rm supp} \eta_{l}^{\alpha}$ are $ O(\lambda^{-n} \langle
k\rangle^{\frac{n \alpha}{1-\alpha}} )$ and $ O( \langle
l \rangle^{\frac{n \alpha}{1-\alpha}} )$, respectively, we see that
\begin{equation}\label{003}
\#\Lambda(l,\lambda)\sim1\vee\lambda^{n(1-\alpha)}.
\end{equation}

When $q=1$, from Lemma \ref{lemma-sequence}, we have
\begin{equation}\label{00000}
\begin{split}
\left\|\sum_{k\in\Gamma}\square_k^{\alpha}f_{\lambda}\right\|_{M_{p,1}^{s,\alpha}}
&=\sum_{k\in\Gamma}\langle
k\rangle^{\frac{s}{1-\alpha}}\|\square_k^{\alpha}f_{\lambda}\|_p
\\
&=\sum_{k\in\Gamma}\langle
k\rangle^{\frac{s}{1-\alpha}}\|(\square_{k,1/\lambda}^{\alpha}f)(\lambda\cdot)\|_p
\\
&=\lambda^{-\frac{n}{p}}\sum_{k\in\Gamma}\langle
k\rangle^{\frac{s}{1-\alpha}}\|\square_{k,1/\lambda}^{\alpha}f\|_p
\\
&\leqslant\lambda^{-\frac{n}{p}}\sum_{k\in\Gamma}\langle
k\rangle^{\frac{s}{1-\alpha}}\sum_{l\in\Lambda(k,1/\lambda)}\|\square_{k,1/\lambda}^{\alpha}\square_l^{\alpha}f\|_p.
\end{split}
\end{equation}

\begin{figure}
\begin{center}
\includegraphics[scale=.55]{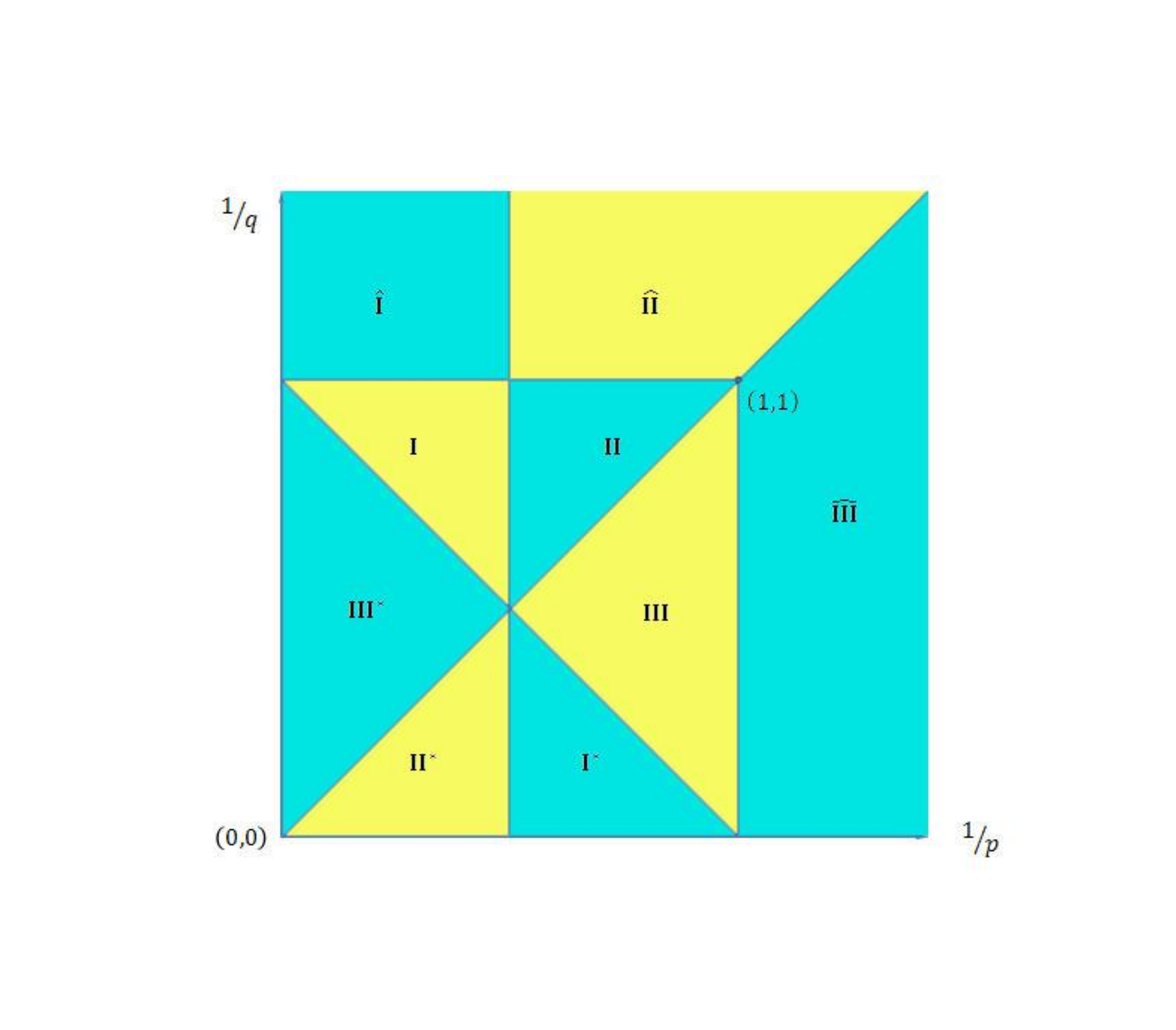}
\vspace*{-2cm}\caption{ 9 regions for the proof of Theorem \ref{dilationP}}\label{figure-9-regions}
\end{center}
\end{figure}

\noindent{\bf Case 1.} $\boldsymbol{\lambda\leqslant1,
\big(\frac1p,\frac1q\big)\in\mbox{I}\cup\mbox{II}.}$
For
$p=1,\infty$, we apply the same technique as that appeared in
Proposition \ref{equivalent-norm} to remove
$\square_{k,1/\lambda}^{\alpha}$ in \eqref{00000}. When
$|k|\lesssim1$, from \eqref{relation-k-l-small} we see that for
$l\in\Lambda(k,1/\lambda)$, there is $1\lesssim\langle
l\rangle^{\frac{1}{1-\alpha}}\lesssim\frac{1}{\lambda}$, which leads
to
\begin{equation}\label{low-f-lambda-0}
\begin{split}
\left\|\sum_{|k|\lesssim1}\square_k^{\alpha}f_{\lambda}\right\|_{M_{p,1}^{s,\alpha}}
&\lesssim\lambda^{-\frac{n}{p}}\sum_{\langle l\rangle \lesssim \lambda^{-(1-\alpha)}}\|\square_l^{\alpha}f\|_p
\lesssim\lambda^{-\frac{n}{p}}\sum_{l\in\mathbb{Z}^n}(1\vee\lambda^s)\langle l\rangle^{\frac{s}{1-\alpha}}\|\square_l^{\alpha}f\|_p \\
&\lesssim\lambda^{-\frac{n}{p}}(1\vee\lambda^s)\|f\|_{M_{p,1}^{s,\alpha}}.
\end{split}
\end{equation}
By Plancherel's identity,
\begin{equation}\label{case1.1-4}
\begin{split}
\left\|\sum_{|k|\lesssim1}\square_k^{\alpha}f_{\lambda}\right\|_{M_{2,2}^{s,\alpha}}&\lesssim\lambda^{-\frac{n}{2}}\left(\sum_{|k|\lesssim1}\langle
k\rangle^{\frac{2s}{1-\alpha}}\left\|\sum_{l\in\Lambda(k,1/\lambda)}\square_{k,1/\lambda}^{\alpha}\square_l^{\alpha}f\right\|_2^2\right)^{\frac12}
\\
&\lesssim\lambda^{-\frac{n}{2}}\left(\sum_{\ \langle l\rangle \lesssim \lambda^{-(1-\alpha)}}\|\square_l^{\alpha}f\|_2^2\right)^{\frac12}
\lesssim\lambda^{-\frac{n}{2}}(1\vee\lambda^s)\|f\|_{M_{2,2}^{s,\alpha}}.
\end{split}
\end{equation}
When $|k|\gg1$, from \eqref{005}-\eqref{00000}, we see that
\begin{equation}\label{2-1-111}
\begin{split}
\left\|\sum_{|k|\gg1}\square_k^{\alpha}f_{\lambda}\right\|_{M_{p,1}^{s,\alpha}}&\lesssim\lambda^{-\frac{n}{p}+s}\sum_{|k|\gg1}\sum_{l\in\Lambda(k,1/\lambda)}\langle
l\rangle^{\frac{s}{1-\alpha}}\|\square_l^{\alpha}f\|_p
\\
&\lesssim\lambda^{-\frac{n}{p}+s}\sum_l\langle
l\rangle^{\frac{s}{1-\alpha}}\sum_{k\in\Lambda(l,\lambda)}\|\square_l^{\alpha}f\|_p
\lesssim\lambda^{-\frac{n}{p}+s}\|f\|_{M_{p,1}^{s,\alpha}}.
\end{split}
\end{equation}
In view  of Plancherel's formula,
\begin{equation}\label{2-1-222}
\begin{split}
\left\|\sum_{|k|\gg1}\square_k^{\alpha}f_{\lambda}\right\|_{M_{2,2}^{s,\alpha}}&\leqslant\lambda^{-\frac{n}{2}+s}\left(\sum_{|k|\gg1}\sum_{l\in\Lambda(k,1/\lambda)}\langle
l\rangle^{\frac{2s}{1-\alpha}}\|\square_{k,1/\lambda}^{\alpha}\square_l^{\alpha}f\|_2^2\right)^{\frac12}
\\
&\lesssim\lambda^{-\frac{n}{2}+s}\left(\sum_{l}\langle
l\rangle^{\frac{2s}{1-\alpha}}\sum_{k\in\Lambda(l,\lambda)}\|\square_{k,1/\lambda}^{\alpha}\square_l^{\alpha}f\|_2^2\right)^{\frac12}
\lesssim\lambda^{-\frac{n}{2}+s}\|f\|_{M_{2,2}^{s,\alpha}}.
\end{split}
\end{equation}
Combining \eqref{low-f-lambda-0}-\eqref{2-1-222}, we use complex
interpolation to get
\begin{equation}\label{case111111111111111111111111111111111}
\|f_{\lambda}\|_{M_{p,q}^{s,\alpha}}\lesssim\lambda^{-\frac{n}{p}}(1\vee\lambda^s)\|f\|_{M_{p,q}^{s,\alpha}}.
\end{equation}

\noindent{\bf Case 2.} $\boldsymbol{\lambda>1,
\big(\frac1p,\frac1q\big)\in\mbox{I}\cup\mbox{III}^*.}$  Through the point
 $\big(\frac{1}{p},\frac1q\big)$, one can draw the parallel line to the
$\frac1q$-axis. We assume there exists some
$(\theta,\eta)\in[0,1]\times[0,1]$, such that the parallel line cuts the
line segment connecting $(0,1),\big(\frac12,1\big)$ and the line
segment connecting $(0,0),\big(\frac12,\frac12\big)$ at
$\big(\frac{\theta}{2},1\big)$ and
$\big(\frac{\theta}{2},\frac{\theta}{2}\big)$, respectively. Assume that
$$
\frac1p=\frac{\theta}{2}, \quad
\frac1q=1-\left(1-\frac{\theta}{2}\right)\eta.
$$
When $|k|\lesssim\lambda^{1-\alpha}$, from
\eqref{relation-k-l-small} and \eqref{00000},  we have
\begin{equation}\label{case1.2-1}
\begin{split}
\left\|\sum_{|k|\lesssim\lambda^{1-\alpha}}\square_k^{\alpha}f_{\lambda}\right\|_{M_{\infty,1}^{s,\alpha}}
&\leqslant\sum_{|l|\lesssim1}\|\square_l^{\alpha}f\|_{\infty}\sum_{k\in\Lambda(l,\lambda)}\langle
k\rangle^{\frac{s}{1-\alpha}}
\lesssim\Big(1\vee\lambda^{s+n(1-\alpha)}\Big)\|f\|_{M_{\infty,1}^{s,\alpha}}.
\end{split}
\end{equation}
By the Schwartz inequality and the Plancherel identity,
\begin{equation}\label{case1.2-2}
\begin{split}
\left\|\sum_{|k|\lesssim\lambda^{1-\alpha}}\square_k^{\alpha}f_{\lambda}\right\|_{M_{2,1}^{s,\alpha}}
&\leqslant\lambda^{-\frac{n}{2}}\sum_{|l|\lesssim1}\left(\sum_{k\in\Lambda(l,\lambda)}\|\square_{k,1/\lambda}^{\alpha}\square_l^{\alpha}f\|_2^2\right)^{\frac12}\left(\sum_{k\in\Lambda(l,\lambda)}\langle
k\rangle^{\frac{2s}{1-\alpha}}\right)^{\frac12} \\
&\lesssim\lambda^{-\frac{n}{2}}\Big(1\vee\lambda^{s+\frac{n}{2}(1-\alpha)}\Big)\|f\|_{M_{2,1}^{s,\alpha}}.
\end{split}
\end{equation}
\noindent From \eqref{relation-k-l-small}, we have
\begin{equation}\label{dsfsfdfsf-11}
\begin{split}
\left\|\sum_{|k|\lesssim\lambda^{1-\alpha}}\square_k^{\alpha}f_{\lambda}\right\|_{M_{\infty,\infty}^{s,\alpha}}&\lesssim\sup_{|k|\lesssim\lambda^{1-\alpha}}\langle
k\rangle^{\frac{s}{1-\alpha}}\|\square_k^{\alpha}f_{\lambda}\|_{\infty}
\\
&\lesssim(1\vee\lambda^s)\sup_{|k|\lesssim\lambda^{1-\alpha}}\sum_{l\in\Lambda(k,1/\lambda)}\|\square_{k,1/\lambda}^{\alpha}\square_l^{\alpha}f\|_{\infty}
\\
&\lesssim(1\vee\lambda^s)\sup_{|l|\lesssim1}\|\square_l^{\alpha}f\|_{\infty}\lesssim(1\vee\lambda^s)\|f\|_{M_{\infty,\infty}^{s,\alpha}}.
\end{split}
\end{equation}
In view of Plancherel's equality,
\begin{equation}\label{dsfsfdfsf}
\begin{split}
\left\|\sum_{|k|\lesssim\lambda^{1-\alpha}}\square_k^{\alpha}f_{\lambda}\right\|_{M_{2,2}^{s,\alpha}}
&\lesssim\left(\sum_{|k|\lesssim\lambda^{1-\alpha}}\langle
k\rangle^{\frac{2s}{1-\alpha}}\|\square_k^{\alpha}f_{\lambda}\|_2^2\right)^{\frac12}
\\
&\lesssim\lambda^{-\frac{n}{2}}\left(\sum_{|l|\lesssim1} (1\vee\lambda^s)^2 \sum_{k\in\Lambda(l,\lambda)}\|\square_{k,1/\lambda}^{\alpha}\square_l^{\alpha}f\|_2^2\right)^{\frac12}
\\
&\lesssim\lambda^{-\frac{n}{2}}(1\vee\lambda^s)\left(\sum_{|l|\lesssim1}\|\square_l^{\alpha}f\|_2^2\right)^{\frac12}
\lesssim\lambda^{-\frac{n}{2}}(1\vee\lambda^s)\|f\|_{M_{2,2}^{s,\alpha}}.
\end{split}
\end{equation}
When $|k|\gg\lambda^{1-\alpha}$, from \eqref{005}-\eqref{00000}, we
have
\begin{equation}\label{2-2-111}
\begin{split}
\left\|\sum_{|k|\gg\lambda^{1-\alpha}}\square_k^{\alpha}f_{\lambda}\right\|_{M_{\infty,1}^{s,\alpha}}&\leqslant\lambda^s\sum_{|l|\gg1}\langle
l\rangle^{\frac{s}{1-\alpha}}\sum_{k\in\Lambda(l,\lambda)}\|\square_l^{\alpha}f\|_{\infty}
\lesssim\lambda^{s+n(1-\alpha)}\|f\|_{M_{\infty,1}^{s,\alpha}}.
\end{split}
\end{equation}
By Jensen's inequality,
\begin{equation}\label{9907}
\begin{split}
\left\|\sum_{|k|\gg\lambda^{1-\alpha}}\square_k^{\alpha}f_{\lambda}\right\|_{M_{2,1}^{s,\alpha}}&\lesssim\lambda^{-\frac{n}{2}+s}\sum_{|l|\gg1}\langle
l\rangle^{\frac{s}{1-\alpha}}\sum_{k\in\Lambda(l,\lambda)}\|\square_{k,1/\lambda}^{\alpha}\square_l^{\alpha}f\|_2
\\
&\lesssim\lambda^{\frac{n}{2}+s}\sum_{|l|\gg1}\langle
l\rangle^{\frac{s}{1-\alpha}} [\#\Lambda(l,\lambda)]^{\frac12}\left(\sum_{k\in\Lambda (l,\lambda)}\|\square_{k,1/\lambda}^{\alpha}\square_{l}^{\alpha}f\|_2^2\right)^{\frac12}
\\
&\lesssim\lambda^{-\frac{n}{2}+s+\frac{n}{2}(1-\alpha)}\|f\|_{M_{2,1}^{s,\alpha}}.
\end{split}
\end{equation}
From \eqref{005},\eqref{003}, we have
\begin{equation}\label{2-2-3}
\begin{split}
\left\|\sum_{|k|\gg\lambda^{1-\alpha}}\square_k^{\alpha}f_{\lambda}\right\|_{M_{\infty,\infty}^{s,\alpha}}&\lesssim\sup_k\langle
k\rangle^{\frac{s}{1-\alpha}}\|\square_{k,1/\lambda}^{\alpha}f\|_{\infty}
\\
&\lesssim\lambda^s\sup_k\sum_{l\in\Lambda(k,1/\lambda)}\langle
l\rangle^{\frac{s}{1-\alpha}}\|\square_{k,1/\lambda}^{\alpha}\square_l^{\alpha}f\|_{\infty}
\lesssim\lambda^s\|f\|_{M_{\infty,\infty}^{s,\alpha}}.
\end{split}
\end{equation}
Similar to \eqref{2-1-222}, one has that
\begin{equation}\label{2-2-4}
\left\|\sum_{|k|\gg\lambda^{1-\alpha}}\square_k^{\alpha}f_{\lambda}\right\|_{M_{2,2}^{s,\alpha}}\lesssim\lambda^{-\frac{n}{2}+s}\|f\|_{M_{2,2}^{s,\alpha}}.
\end{equation}
Since
$n(1-\alpha)\big(1-\frac{\theta}{2}\big)=(1-\theta)n(1-\alpha)+\theta\frac{n}{2}(1-\alpha)$,
combining
\eqref{case1.2-1},\eqref{2-2-111},\eqref{case1.2-2},\eqref{9907},
complex interpolation yields
\begin{equation}\label{second-case-1}
\|f_{\lambda}\|_{M_{\frac{2}{\theta},1}^{s,\alpha}}\lesssim\lambda^{-\frac{\theta}{2}n} \left(1\vee\lambda^{s+n(1-\alpha)\big(1-\frac{\theta}{2}\big)}\right)\|f\|_{M_{\frac{2}{\theta},1}^{s,\alpha}}.
\end{equation}
Combining
\eqref{dsfsfdfsf-11},\eqref{2-2-3},\eqref{dsfsfdfsf},\eqref{2-2-4},
complex interpolation yields
\begin{equation}\label{second-case-2}
\|f_{\lambda}\|_{M_{\frac{2}{\theta},\frac{2}{\theta}}^{s,\alpha}}\lesssim\lambda^{-\frac{\theta}{2}n}(1\vee\lambda^s)\|f\|_{M_{\frac{2}{\theta},\frac{2}{\theta}}^{s,\alpha}}.
\end{equation}
Interpolating \eqref{second-case-1} and \eqref{second-case-2}, we have
\begin{equation}\label{case222222222222222222222222222222222}
\|f_{\lambda}\|_{M_{p,q}^{s,\alpha}}\lesssim\lambda^{-\frac{n}{p}}\left(1\vee\lambda^{s+n(1-\alpha)\big(\frac1q-\frac1p\big)}\right)\|f\|_{M_{p,q}^{s,\alpha}}.
\end{equation}

\noindent{\bf Case
3.} $\boldsymbol{\lambda>1,\big(\frac1p,\frac1q\big)\in\mbox{II}\cup\mbox{III}.}$
Through the point $\big(\frac{1}{p},\frac1q\big)$ one can make the parallel
line to the $\frac1q$-axis. We assume there exists some
$(\theta,\eta)\in[0,1]\times[0,1]$, such that the parallel line cuts the
line segment connecting $(1,1),\big(\frac12,1\big)$ and the line
segment connecting $(1,0),\big(\frac12,\frac12\big)$ at
$\big(1-\frac{\theta}{2},1\big)$ and
$\big(1-\frac{\theta}{2},\frac{\theta}{2}\big)$, respectively. We can assume that
$$
\frac1p=1-\frac{\theta}{2}, \quad
\frac1q=1-\left(1-\frac{\theta}{2}\right)\eta.
$$
When $|k|\lesssim\lambda^{1-\alpha}$, similarly to \eqref{case1.2-1}
and \eqref{dsfsfdfsf-11}, we have
\begin{equation}\label{case3-1111111111}
\left\|\sum_{|k|\lesssim\lambda^{1-\alpha}}\square_k^{\alpha}f_{\lambda}\right\|_{M_{1,1}^{s,\alpha}}\lesssim
\lambda^{-n}\left(1\vee\lambda^{s+n(1-\alpha)}\right)\|f\|_{M_{1,1}^{s,\alpha}};
\end{equation}
\begin{equation}\label{case3-1111111112}
\left\|\sum_{|k|\lesssim\lambda^{1-\alpha}}\square_k^{\alpha}f_{\lambda}\right\|_{M_{1,\infty}^{s,\alpha}}\lesssim
\lambda^{-n}\left(1\vee\lambda^s\right)\|f\|_{M_{1,\infty}^{s,\alpha}}.
\end{equation}
When $|k|\gg\lambda^{1-\alpha}$, similarly to \eqref{2-2-111} and
\eqref{2-2-3}, we have
\begin{equation}\label{case3-1111111113}
\left\|\sum_{|k|\gg\lambda^{1-\alpha}}\square_k^{\alpha}f_{\lambda}\right\|_{M_{1,1}^{s,\alpha}}\lesssim\lambda^{-n+s+n(1-\alpha)}\|f\|_{M_{1,1}^{s,\alpha}};
\end{equation}
\begin{equation}\label{case3-1111111114}
\left\|\sum_{|k|\gg\lambda^{1-\alpha}}\square_k^{\alpha}f_{\lambda}\right\|_{M_{1,\infty}^{s,\alpha}}\lesssim\lambda^{-n+s}\|f\|_{M_{1,\infty}^{s,\alpha}}.
\end{equation}
Combining
\eqref{case3-1111111111},\eqref{case3-1111111113},\eqref{case1.2-2},\eqref{9907},
complex interpolation yields
\begin{equation}\label{case3-1111111115}
\|f_{\lambda}\|_{M_{\frac{2}{2-\theta},1}^{s,\alpha}}\lesssim\lambda^{-\frac{n}{2} \big(1-\frac{\theta}{2}\big)}\left(1\vee\lambda^{s+n(1-\alpha)\big(1-\frac{\theta}{2}\big)}\right)\|f\|_{M_{\frac{2}{2-\theta},1}^{s,\alpha}}.
\end{equation}
Combining
\eqref{case3-1111111112},\eqref{case3-1111111114},\eqref{dsfsfdfsf},\eqref{2-2-4},
complex interpolation yields
\begin{equation}\label{case3-1111111116}
\|f_{\lambda}\|_{M_{\frac{2}{2-\theta},\frac{2}{\theta}}^{s,\alpha}}\lesssim\lambda^{-\frac{n}{2}\big(1-\frac{\theta}{2}\big)}\left(1\vee\lambda^s\right)\|f\|_{M_{\frac{2}{2-\theta},\frac{2}{\theta}}^{s,\alpha}}.
\end{equation}
Interpolating  \eqref{case3-1111111115} and \eqref{case3-1111111116}, we
have
\begin{equation}\label{case333333333333333333333333333333333}
\|f_{\lambda}\|_{M_{p,q}^{s,\alpha}}\lesssim\lambda^{-\frac{n}{p}}\left(1\vee\lambda^{s+n(1-\alpha)\big(\frac1p+\frac1q-1\big)}\right) \|f\|_{M_{p,q}^{s,\alpha}}.
\end{equation}

\noindent{\bf Case
4.} $\boldsymbol{\lambda\leqslant1,\big(\frac1p,\frac1q\big)\in\{\mbox{I}^*\cup\mbox{II}^*\cup\mbox{III}\cup\mbox{III}^*\}\backslash \{(0,1] \times[0,1]\}};  \, {\bf
or}\,\boldsymbol{\lambda>1,\big(\frac1p,\frac1q\big)}$
$\boldsymbol{\in\mbox{I}^*\cup\mbox{II}^*\backslash\{(1,0)\}.}$ We
observe that, for any $\big(\frac{1}{p},\frac{1}{q}\big)\in
 \mbox{I}^{\ast} \cup \mbox{II}^{\ast}$, we have
$\big(\frac{1}{p^*},\frac{1}{q^*}\big) \in
\mbox{I} \cup \mbox{II}$.  By duality,
\begin{equation}\label{018}
\begin{split}
|\langle f_{\lambda},g \rangle|=\frac{1}{\lambda^n}|\langle f,g_{\frac{1}{\lambda}} \rangle|
\leqslant\frac{1}{\lambda^n}\|f\|_{M_{p,q}^{s,\alpha}}\big\|g_{\frac{1}{\lambda}}\big\|_{M_{p^{\ast},q^{\ast}}^{-s,\alpha}}.
\end{split}
\end{equation}
If we denote $\|f_{\lambda}\|_{M_{p,q}^{s,\alpha}}\lesssim
F(s,\lambda;p,q)\|f\|_{M_{p,q}^{s,\alpha}}$, from the previous
several cases, we know that
\begin{equation}\label{019}
\begin{split}
\big\|g_{\frac{1}{\lambda}}\big\|_{M_{p^{\ast},q^{\ast}}^{-s,\alpha}}\lesssim
F\left(-s, \lambda^{-1}; p^*,q^*\right)\|g\|_{M_{p^{\ast},q^{\ast}}^{-s,\alpha}}.
\end{split}
\end{equation}
By the principle of duality, it follows from \eqref{018} and \eqref{019} that
\begin{equation}\label{020}
\begin{split}
\|f_{\lambda}\|_{M_{p,q}^{s,\alpha}}
&\lesssim\lambda^{-n}F\left(-s, \lambda^{-1};p^*,q^*\right)\|f\|_{M_{p,q}^{s,\alpha}}.
\end{split}
\end{equation}
For $\lambda\leqslant1$ with
$\big(\frac1p,\frac1q\big)\in\{\mbox{I}^*\cup\mbox{III}\}\backslash
\{1\}\times[0,1]$, from Case 2, \eqref{020} gives
\begin{equation}\label{3-003}
\|f_{\lambda}\|_{M_{p,q}^{s,\alpha}}\lesssim\lambda^{-\frac{n}{p}}\left(1\vee\lambda^{s-n(1-\alpha)\big(\frac1p-\frac1q\big)}\right)\|f\|_{M_{p,q}^{s,\alpha}}.
\end{equation}
For $\lambda\leqslant1$ with
$\big(\frac1p,\frac1q\big)\in\{\mbox{II}^*\cup\mbox{III}^*\}\backslash\{(0,1)\}$,
from Case 3, \eqref{020} gives
\begin{equation}\label{3-004}
\|f_{\lambda}\|_{M_{p,q}^{s,\alpha}}\lesssim\lambda^{-\frac{n}{p}}\left(1\vee\lambda^{s-n(1-\alpha)\big(1-\frac1q-\frac1p\big)}\right)\|f\|_{M_{p,q}^{s,\alpha}}.
\end{equation}
For $\lambda>1$ with
$\big(\frac1p,\frac1q\big)\in\mbox{I}^*\cup\mbox{II}^*\backslash\{(1,0)\}$,
from Case 1, \eqref{020} gives
\begin{equation}\label{3-005}
\|f_{\lambda}\|_{M_{p,q}^{s,\alpha}}\lesssim\lambda^{-\frac{n}{p}}(1\vee\lambda^s)\|f\|_{M_{p,q}^{s,\alpha}}.
\end{equation}

\noindent{\bf Case
5.} $\boldsymbol{\big(\frac1p,\frac1q\big)\in\{\widehat{\mbox{I}}\cup\widehat{\mbox{II}}\}\backslash(1,\infty)\times(1,\infty).}$
Since $\ell^q\subset\ell^1$, we know
\begin{equation}\label{jijijijijijijijiji}
\begin{split}
\|\square_k^{\alpha}f_{\lambda}\|_p
\lesssim\lambda^{-\frac{n}{p}}\sum_{l\in\Lambda(k,1/\lambda)}\|\square_{k,1/\lambda}^{\alpha}\square_l^{\alpha}f\|_p
\lesssim\lambda^{-\frac{n}{p}}\left(\sum_{l\in\Lambda(k,1/\lambda)}\|\square_{k,1/\lambda}^{\alpha}\square_l^{\alpha}f\|_p^q\right)^{\frac1q}.
\end{split}
\end{equation}
From \eqref{jijijijijijijijiji},\eqref{relation-k-l-small}, when
$\lambda\leqslant1$ and $|k|\lesssim1$, we conclude that
\begin{equation}\label{scaling-case4-1}
\begin{split}
\left\|\sum_{|k|\lesssim1}\square_k^{\alpha}f_{\lambda}\right\|_{M_{\infty,q}^{s,\alpha}}
&\lesssim\left(\sum_{|k|\lesssim1}\langle
k\rangle^{\frac{sq}{1-\alpha}}\|\square_k^{\alpha}f_{\lambda}\|_{\infty}^q\right)^{\frac1q}
\\
&\lesssim\left(\sum_{|k|\lesssim1}\langle
k\rangle^{\frac{sq}{1-\alpha}}\sum_{l\in\Lambda(k,1/\lambda)}\|\square_{k,1/\lambda}^{\alpha}\square_l^{\alpha}f\|_{\infty}^q\right)^{\frac1q}
\\
&\lesssim\left(\sum_{\langle l\rangle \lesssim \lambda^{-(1-\alpha)}}\|\square_l^{\alpha}f\|_{\infty}^q\right)^{\frac1q}
\lesssim(1\vee\lambda^s)\|f\|_{M_{\infty,q}^{s,\alpha}};
\end{split}
\end{equation}
and when $\lambda>1$ and $|k|\lesssim\lambda^{1-\alpha}$,
\begin{equation}\label{scaling-case4-2}
\begin{split}
\left\|\sum_{|k|\lesssim\lambda^{1-\alpha}}\square_k^{\alpha}f_{\lambda}\right\|_{M_{\infty,q}^{s,\alpha}}
&\lesssim\left(\sum_{|k|\lesssim\lambda^{1-\alpha}}\langle
k\rangle^{\frac{sq}{1-\alpha}}\|\square_k^{\alpha}f_{\lambda}\|_{\infty}^q\right)^{\frac1q}
\\
&\lesssim\left(\sum_{|k|\lesssim\lambda^{1-\alpha}}\langle
k\rangle^{\frac{sq}{1-\alpha}}\sum_{l\in\Lambda(k,1/\lambda)}\|\square_l^{\alpha}f\|_{\infty}^q\right)^{\frac1q}
\\
&\lesssim\left(\sum_{|l|\lesssim1}\|\square_l^{\alpha}f\|_{\infty}^q\sum_{|k|\lesssim\lambda^{1-\alpha}}\langle
k\rangle^{\frac{sq}{1-\alpha}}\right)^{\frac1q} \\
&\lesssim\left(1\vee\lambda^{s+n(1-\alpha)\frac1q}\right)\|f\|_{M_{\infty,q}^{s,\alpha}}.
\end{split}
\end{equation}
When $|k|\gg1\vee\lambda^{1-\alpha}$, from
\eqref{jijijijijijijijiji}, \eqref{005}, \eqref{003}, we have
\begin{equation}\label{scaling-case4-3}
\begin{split}
\left\|\sum_{|k|\gg1\vee\lambda^{1-\alpha}}\square_k^{\alpha}f_{\lambda}\right\|_{M_{\infty,q}^{s,\alpha}}
&\lesssim\left(\sum_{|k|\gg1\vee\lambda^{1-\alpha}}\langle
k\rangle^{\frac{sq}{1-\alpha}}\sum_{l\in\Lambda(k,1/\lambda)}\|\square_l^{\alpha}f\|_{\infty}^q\right)^{\frac1q}
\\
&\lesssim\lambda^s\left(\sum_{|l|\gg1}\langle
l\rangle^{\frac{sq}{1-\alpha}}\sum_{k\in\Lambda(l,\lambda)}\|\square_l^{\alpha}f\|_{\infty}\right)^{\frac1q}
\\
&\lesssim\lambda^s(1\vee\lambda)^{n(1-\alpha)\frac1q}\|f\|_{M_{\infty,q}^{s,\alpha}}.
\end{split}
\end{equation}
Therefore, combining
\eqref{scaling-case4-1}-\eqref{scaling-case4-3}, we get
\begin{equation}\label{scaling-case4-m-infty-q}
\|f_{\lambda}\|_{M_{\infty,q}^{s,\alpha}}\lesssim1\vee\lambda^s\vee\lambda^{s+n(1-\alpha)\frac1q}\|f\|_{M_{\infty,q}^{s,\alpha}}.
\end{equation}
The same for $M_{1,q}^{s,\alpha}$. Corresponding to
\eqref{scaling-case4-m-infty-q}, we get
\begin{equation}\label{scaling-case4-m-1-q}
\|f_{\lambda}\|_{M_{1,q}^{s,\alpha}}\lesssim\lambda^{-n}\left(1\vee\lambda^s\vee\lambda^{s+n(1-\alpha)\frac1q}\right)\|f\|_{M_{1,q}^{s,\alpha}}.
\end{equation}
Whereas when $p=2, \lambda\leqslant1$ and $|k|\lesssim1$, from
\eqref{jijijijijijijijiji},\eqref{relation-k-l-small}, we have
\begin{equation}\label{scaling-case4-4}
\begin{split}
\left\|\sum_{|k|\lesssim1}\square_k^{\alpha}f_{\lambda}\right\|_{M_{2,q}^{s,\alpha}}
&\lesssim\lambda^{-\frac{n}{2}}\left(\sum_{|k|\lesssim1}\langle
k\rangle^{\frac{sq}{1-\alpha}}\sum_{l\in\Lambda(k,1/\lambda)}\|\square_{k,1/\lambda}^{\alpha}\square_l^{\lambda}f\|_2^q\right)^{\frac1q}
\\
&\lesssim\lambda^{-\frac{n}{2}}\left(\sum_{\langle l\rangle \lesssim \lambda^{-(1-\alpha)}}\|\square_l^{\alpha}f\|_2^q\right)^{\frac1q}
\lesssim\lambda^{-\frac{n}{2}}(1\vee\lambda^s)\|f\|_{M_{2,q}^{s,\alpha}}.
\end{split}
\end{equation}
When $\lambda>1, |k|\lesssim\lambda^{1-\alpha}$, by
\eqref{jijijijijijijijiji}, \eqref{relation-k-l-small} and
 H\"{o}lder's inequality, we have
\begin{equation}\label{scaling-case4-5}
\begin{split}
\left\|\sum_{|k|\lesssim\lambda^{1-\alpha}}\square_k^{\alpha}f_{\lambda}\right\|_{M_{2,q}^{s,\alpha}}
&\lesssim\lambda^{-\frac{n}{2}}\left(\sum_{|l|\lesssim1}\sum_{k\in\Lambda(l,\lambda)}\langle
k\rangle^{\frac{sq}{1-\alpha}}\|\square_{k,1/\lambda}^{\alpha}\square_l^{\alpha}f\|_2^q\right)^{\frac1q}
\\
&\lesssim\lambda^{-\frac{n}{2}}\left[\sum_{|l|\lesssim1}\left(\sum_{k\in\Lambda(l,\lambda)}\|\square_{k,1/\lambda}^{\alpha}\square_l^{\alpha}f\|_2^2\right)^{\frac{q}{2}}
\left(\sum_{|k|\lesssim\lambda^{1-\alpha}}\langle
k\rangle^{\frac{sq}{1-\alpha}\cdot\frac{2}{2-q}}\right)^{\frac{2-q}{2}}\right]^{\frac1q}
\\
&\lesssim\lambda^{-\frac{n}{2}}\left(1\vee\lambda^{s+n(1-\alpha)\big(\frac1q-\frac12\big)}\right)\|f\|_{M_{2,q}^{s,\alpha}}.
\end{split}
\end{equation}
When $|k|\gg1\vee\lambda^{1-\alpha}$, in view of
\eqref{jijijijijijijijiji},\eqref{005},\eqref{003}
and H\"{o}lder's inequality, we have
\begin{equation}\label{scaling-case4-6}
\begin{split}
\left\|\sum_{|k|\gg1\vee\lambda^{1-\alpha}}\square_k^{\alpha}f_{\lambda}\right\|_{M_{2,q}^{s,\alpha}}
&\lesssim\lambda^{-\frac{n}{2}+s}\left(\sum_{|l|\gg1}\langle
l\rangle^{\frac{sq}{1-\alpha}}\sum_{k\in\Lambda(l,\lambda)}\|\square_l^{\alpha}\square_{k,1/\lambda}^{\alpha}f\|_2^q\right)^{\frac1q}
\\
&\lesssim\lambda^{-\frac{n}{2}+s}\left[\sum_{|l|\gg1}\langle
l\rangle^{\frac{sq}{1-\alpha}}[\#\Lambda(l,\lambda)]^{1-\frac{q}{2}}\left(\sum_{k\in\Lambda(l,\lambda)}\|\square_l^{\alpha}\square_{k,1/\lambda}^{\alpha}f\|_2^2\right)^{\frac{q}{2}}\right]^{\frac1q}
\\
&\lesssim\lambda^{-\frac{n}{2}+s}(1\vee\lambda)^{n(1-\alpha)\big(\frac1q-\frac12\big)}\|f\|_{M_{2,q}^{s,\alpha}}.
\end{split}
\end{equation}
Therefore, combining
\eqref{scaling-case4-4}-\eqref{scaling-case4-6}, we get
\begin{equation}\label{scaling-case4-m-2-q}
\|f_{\lambda}\|_{M_{2,q}^{s,\alpha}}\lesssim\lambda^{-\frac{n}{2}}\left(1\vee\lambda^s\vee\lambda^{s+n(1-\alpha)\big(\frac1q-\frac12\big)}\right)\|f\|_{M_{2,q}^{s,\alpha}}.
\end{equation}
For $\big(\frac1p,\frac1q\big)\in\widehat{\mbox{I}}$,
complex interpolation between \eqref{scaling-case4-m-infty-q} and
\eqref{scaling-case4-m-2-q} yields
\begin{equation}\label{scaling-case4-7}
\|f_{\lambda}\|_{M_{p,q}^{s,\alpha}}\lesssim\lambda^{-\frac{n}{p}}\left(1\vee\lambda^s\vee\lambda^{s+n(1-\alpha) \big(\frac1q-\frac1p\big)}\right)\|f\|_{M_{p,q}^{s,\alpha}};
\end{equation}
while for $\big(\frac1p,\frac1q\big)\in\widehat{\mbox{II}}\backslash(1,\infty)\times(1,\infty)$, complex interpolation between \eqref{scaling-case4-m-1-q} and
\eqref{scaling-case4-m-2-q}  yields
\begin{equation}\label{scaling-case4-8}
\|f_{\lambda}\|_{M_{p,q}^{s,\alpha}}\lesssim\lambda^{-\frac{n}{p}}\left(1\vee\lambda^s\vee\lambda^{s+n(1-\alpha)\big(\frac1p+\frac1q-1\big)}\right)\|f\|_{M_{p,q}^{s,\alpha}}.
\end{equation}

\noindent{\bf Case
6.}\, $\boldsymbol{\big(\frac1p,\frac1q\big)\in\widehat{\mbox{III}}.}$
Since $l^p\subset l^1$, we know
\begin{equation}\label{c7-01}
\begin{split}
\|\square_k^{\alpha}f_{\lambda}\|_p
=\lambda^{-\frac{n}{p}}\|\square_{k,1/\lambda}^{\alpha}f\|_p
\leqslant\lambda^{-\frac{n}{p}}\left(\sum_{l\in\Lambda(k,1/\lambda)}\|\square_{k,1/\lambda}^{\alpha}\square_l^{\alpha}f\|_p^p\right)^{\frac1p}.
\end{split}
\end{equation}
When $\lambda\leqslant1$ and $|k|\lesssim1$, from
\eqref{eta-2},\eqref{005}, we see that
\begin{equation*}
\mbox{diam}\;\mbox{supp}\mathscr{F}[(\mathscr{F}^{-1}(\eta_k^{\alpha})_{\lambda})(x-\cdot)\square_l^{\alpha}f(\cdot)]\lesssim1/\lambda,
\end{equation*}
By Proposition \ref{convolution}, imitating the processes  as in
\eqref{equivalent-ffdfddfdfdfdfd}-\eqref{algebra-p<1-2}, we get
\begin{equation}\label{5-2222}
\|\square_{k,1/\lambda}^{\alpha}\square_l^{\alpha}f\|_p\lesssim\left(1/\lambda\right)^{n\big(\frac1p-1\big)}\left(\lambda\big/\langle
k\rangle^{\frac{\alpha}{1-\alpha}}\right)^{n\big(\frac1p-1\big)}\|\square_l^{\alpha}f\|_p\lesssim\|\square_l^{\alpha}f\|_p.
\end{equation}
From \eqref{c7-01}, \eqref{5-2222}, the embedding $\ell^1\subset\ell^{\frac{q}{p}}$, and H\"{o}lder's
inequality,  we have
\begin{equation}\label{5-3}
\begin{split}
\left\|\sum_{|k|\lesssim1}\square_k^{\alpha}f_{\lambda}\right\|_{M_{p,q}^{s,\alpha}}
&\lesssim\lambda^{-\frac{n}{p}}\left[\sum_{|k|\lesssim1}\langle k\rangle^{\frac{sq}{1-\alpha}}\left(\sum_{l\in\Lambda(k,1/\lambda)}\|\square_{k,1/\lambda}^{\alpha}\square_l^{\alpha}f\|_p^p\right)^{\frac{q}{p}}\right]^{\frac1q}
\\
&\lesssim\lambda^{-\frac{n}{p}}\left(\sum_{|k|\lesssim1}\langle
k\rangle^{\frac{sp}{1-\alpha}}\sum_{l\in\Lambda(k,1/\lambda)}\|\square_{k,1/\lambda}^{\alpha}\square_l^{\alpha}f\|_p^p\right)^{\frac1p}
\\
&\lesssim\lambda^{-\frac{n}{p}}\left[\sum_{\langle
l\rangle\lesssim\lambda^{-(1-\alpha)}}\left(\sum_{|k|\lesssim1}\langle
k\rangle^{\frac{s}{1-\alpha}\cdot\frac{pq}{q-p}}\right)^{\frac{q-p}{q}}\left(\sum_{|k|\lesssim1}\|\square_l^{\alpha}f\|_p^q\right)^{\frac{p}{q}}\right]^{\frac1p}
\\
&\lesssim\lambda^{-\frac{n}{p}}(1\vee\lambda^s)\|f\|_{M_{p,q}^{s,\alpha}}.
\end{split}
\end{equation}
When $\lambda\leqslant1$ and $|k|\gg1$,  from
\eqref{eta-2}, \eqref{005},  we see that
\begin{equation*}
\mbox{diam}\;\mbox{supp}\mathscr{F}[(\mathscr{F}^{-1}(\eta_k^{\alpha})_{\lambda})(x-\cdot)\square_l^{\alpha}f(\cdot)]\lesssim\langle
k\rangle^{\frac{\alpha}{1-\alpha}}/\lambda.
\end{equation*}
Similarly to \eqref{5-2222},  we get
\begin{equation}\label{5-4}
\|\square_{k, 1/\lambda}^{\alpha}\square_l^{\alpha}f\|_p\lesssim\left(\langle
k\rangle^{\frac{\alpha}{1-\alpha}}/\lambda\right)^{n\big(\frac1p-1\big)}\left(\lambda\big/\langle
k\rangle^{\frac{\alpha}{1-\alpha}}\right)^{n\big(\frac1p-1\big)}\|\square_l^{\alpha}f\|_p\leqslant\|\square_l^{\alpha}f\|_p.
\end{equation}
From \eqref{c7-01}, \eqref{5-4}, \eqref{005}, \eqref{003},    we have
\begin{equation}\label{5-5}
\begin{split}
\left\|\sum_{|k|\gg1}\square_k^{\alpha}f_{\lambda}\right\|_{M_{p, q}^{s, \alpha}}&\lesssim\lambda^{-\frac{n}{p}}\left[\sum_{|k|\gg1}\langle
k\rangle^{\frac{sq}{1-\alpha}}\left(\sum_{l\in\Lambda(k, 1/\lambda)}\|\square_{k, 1/\lambda}^{\alpha}\square_l^{\alpha}f\|_p^p\right)^{\frac{q}{p}}\right]^{\frac1q}
\\
&\lesssim\lambda^{-\frac{n}{p}+s}\left(\sum_{|k|\gg1}[\# \Lambda(k, 1/\lambda)]^{\frac{q}{p}-1}\sum_{l\in\Lambda(k, 1/\lambda)}\langle
l\rangle^{\frac{sq}{1-\alpha}}\|\square_l^{\alpha}f\|_p^q\right)^{\frac1q}
\\
&\lesssim\lambda^{-\frac{n}{p}+s-n(1-\alpha)\big(\frac1p-\frac1q\big)}\left(\sum_l\langle
l\rangle^{\frac{sq}{1-\alpha}}\sum_{k\in\Lambda(l, \lambda)}\|\square_l^{\alpha}f\|_p^q\right)^{\frac1q}
\\
&\lesssim\lambda^{-\frac{n}{p}+s-n(1-\alpha)\big(\frac1p-\frac1q\big)}\|f\|_{M_{p, q}^{s, \alpha}}.
\end{split}
\end{equation}

\noindent For $\lambda>1$ and $|k|\lesssim\lambda^{1-\alpha}$,  from
\eqref{eta-2}, \eqref{005},  we know
\begin{equation*}
\mbox{diam}\;\mbox{supp}\mathscr{F}[(\mathscr{F}^{-1}(\eta_k^{\alpha})_{\lambda})(x-\cdot)\square_l^{\alpha}f(\cdot)]\lesssim1.
\end{equation*}
Similarly to \eqref{5-2222},  we get
\begin{equation}\label{scaling-case5-111}
\|\square_{k, 1/\lambda}^{\alpha}\square_l^{\alpha}f\|_p\lesssim\left(\lambda\big/\langle
k\rangle^{\frac{\alpha}{1-\alpha}}\right)^{n\big(\frac1p-1\big)}\|\square_l^{\alpha}f\|_p.
\end{equation}
From
\eqref{c7-01}, \eqref{scaling-case5-111}, \eqref{relation-k-l-small},
we have
\begin{equation}\label{scaling-case5-113}
\begin{split}
\left\|\sum_{|k|\lesssim\lambda^{1-\alpha}}\square_k^{\alpha}f_{\lambda}\right\|_{M_{p, q}^{s, \alpha}}
&\lesssim\lambda^{-\frac{n}{p}}\left[\sum_{|k|\lesssim\lambda^{1-\alpha}}\langle
k\rangle^{\frac{sq}{1-\alpha}}\left(\sum_{l\in\Lambda(k, 1/\lambda)}\|\square_{k, 1/\lambda}^{\alpha}\square_l^{\alpha}f\|_p^p\right)^{\frac{q}{p}}\right]^{\frac1q}
\\
&\lesssim\lambda^{-\frac{n}{p}+n\big(\frac1p-1\big)}\left[\sum_{|k|\lesssim\lambda^{1-\alpha}}\langle
k\rangle^{\left[\frac{s}{1-\alpha}-\frac{n\alpha}{1-\alpha}\big(\frac1p-1\big)\right]q}\left(\sum_{|l|\lesssim1}\|\square_l^{\alpha}f\|_p^p\right)^{\frac{q}{p}}\right]^{\frac1q}
\\
&\lesssim\lambda^{-\frac{n}{p}+n\big(\frac1p-1\big)}\left(1\vee\lambda^{s-n\alpha\big(\frac1p-1\big)+n(1-\alpha)\frac1q}\right)\|f\|_{M_{p, q}^{s, \alpha}}.
\end{split}
\end{equation}
When $\lambda>1$ and $|k|\gg\lambda^{1-\alpha}$,  similarly to
\eqref{5-4},  we get
\begin{equation}\label{scaling-case5-112}
\|\square_{k, 1/\lambda}^{\alpha}\square_l^{\alpha}f\|_p\lesssim\|\square_l^{\alpha}f\|_p.
\end{equation}
From \eqref{c7-01}, \eqref{scaling-case5-112}, \eqref{005},  and the
embedding $\ell^1\subset\ell^{\frac{q}{p}}$,  we have
\begin{equation}\label{scaling-case5-114}
\begin{split}
\left\|\sum_{|k|\gg\lambda^{1-\alpha}}\square_k^{\alpha}f_{\lambda}\right\|_{M_{p, q}^{s, \alpha}}
&\lesssim\lambda^{-\frac{n}{p}}\left[\sum_{|k|\gg\lambda^{1-\alpha}}\langle
k\rangle^{\frac{sq}{1-\alpha}}\left(\sum_{l\in\Lambda(k, 1/\lambda)}\|\square_{k, 1/\lambda}^{\alpha}\square_l^{\alpha}f\|_p^p\right)^{\frac{q}{p}}\right]^{\frac1q}
\\
&\lesssim\lambda^{-\frac{n}{p}+s}\left(\sum_{|k|\gg\lambda^{1-\alpha}}\sum_{l\in\Lambda(k, 1/\lambda)}\langle
l\rangle^{\frac{sq}{1-\alpha}}\|\square_l^{\alpha}f\|_p^q\right)^{\frac1q}
\\
&\lesssim\lambda^{-\frac{n}{p}+s}\left(\sum_{|l|\gg1}\langle
l\rangle^{\frac{sq}{1-\alpha}}\sum_{k\in\Lambda(l, \lambda)}\|\square_l^{\alpha}f\|_p^q\right)^{\frac1q}
\\
&\lesssim\lambda^{-\frac{n}{p}+s+n(1-\alpha)\frac1q}\|f\|_{M_{p, q}^{s, \alpha}}.
\end{split}
\end{equation}
We summarize the argument in this case as: if $\lambda\leqslant1$,
\eqref{5-3} and \eqref{5-5} give
\begin{equation}\label{case5-1111111111111111111111111111111}
\|f_{\lambda}\|_{M_{p, q}^{s, \alpha}}\lesssim\lambda^{-\frac{n}{p}}\left(1\vee\lambda^{s+n(1-\alpha)\big(\frac1q-\frac1p\big)}\right)\|f\|_{M_{p, q}^{s, \alpha}};
\end{equation}
else if $\lambda>1$,  \eqref{scaling-case5-113} and
\eqref{scaling-case5-114} give
\begin{equation}\label{case5-2222222222222222222222222222222}
\|f_{\lambda}\|_{M_{p, q}^{s, \alpha}}\lesssim\lambda^{-\frac{n}{p}+n\big(\frac1p-1\big)}\left(1\vee\lambda^{s-n\alpha\big(\frac1p-1\big)+n(1-\alpha)\frac1q}\right)\|f\|_{M_{p, q}^{s, \alpha}}.
\end{equation}

\noindent{\bf Case
7.} $\boldsymbol{\big(\frac1p, \frac1q\big)\in\widehat{\mbox{II}}\cap(1, \infty)\times(1,\infty).}$
It is a natural consequence of Cases 5 and 6 by complex
interpolation.

\eqref{scaling} in the case   $\lambda\geqslant1$ follows from \eqref{case222222222222222222222222222222222}, \eqref{scaling-case4-7}, \eqref{case333333333333333333333333333333333}, \eqref{scaling-case4-8}, \eqref{3-005}, \eqref{case5-2222222222222222222222222222222}.
\eqref{scaling} in the case   $\lambda<1$ follows from \eqref{case111111111111111111111111111111111}, \eqref{scaling-case4-7},
\eqref{scaling-case4-8}, \eqref{3-003}, \eqref{case5-2222222222222222222222222222222}, \eqref{3-004}.
\end{proof}

\begin{remk}
If $s=-s_c$,  we have the substitution for \eqref{scaling}:
\begin{equation}\label{logcase}
\|f_{\lambda}\|_{M_{p, q}^{s, \alpha}}\lesssim \lambda^{-\frac{n}{p} } F(\lambda) \|f\|_{M_{p, q}^{s, \alpha}},
\end{equation}
where
\begin{equation}
F(\lambda)=
\begin{cases}
(\ln \lambda)^{0\vee\big(\frac1q-\frac1p\big)\vee\big(\frac1q+\frac1p-1\big)},
& \quad
\lambda>1, \ p \ge 1;
\\
\lambda^{n\big(\frac1p-1\big)}(\ln \lambda)^{\frac1q}, & \quad
\lambda>1, \ p \le 1 ; \\
\left(\ln \frac{1}{\lambda}\right)^{0\vee\big(\frac1p-\frac1q\big)\vee\big(1-\frac1p-\frac1q\big)},
& \quad \lambda\leqslant1.
\end{cases}
\end{equation}
\end{remk}

\section{Embedding between $\alpha$-modulation and Besov spaces}

\begin{figure}
\begin{center}
\hspace*{-2.5cm}\includegraphics[scale=0.93]{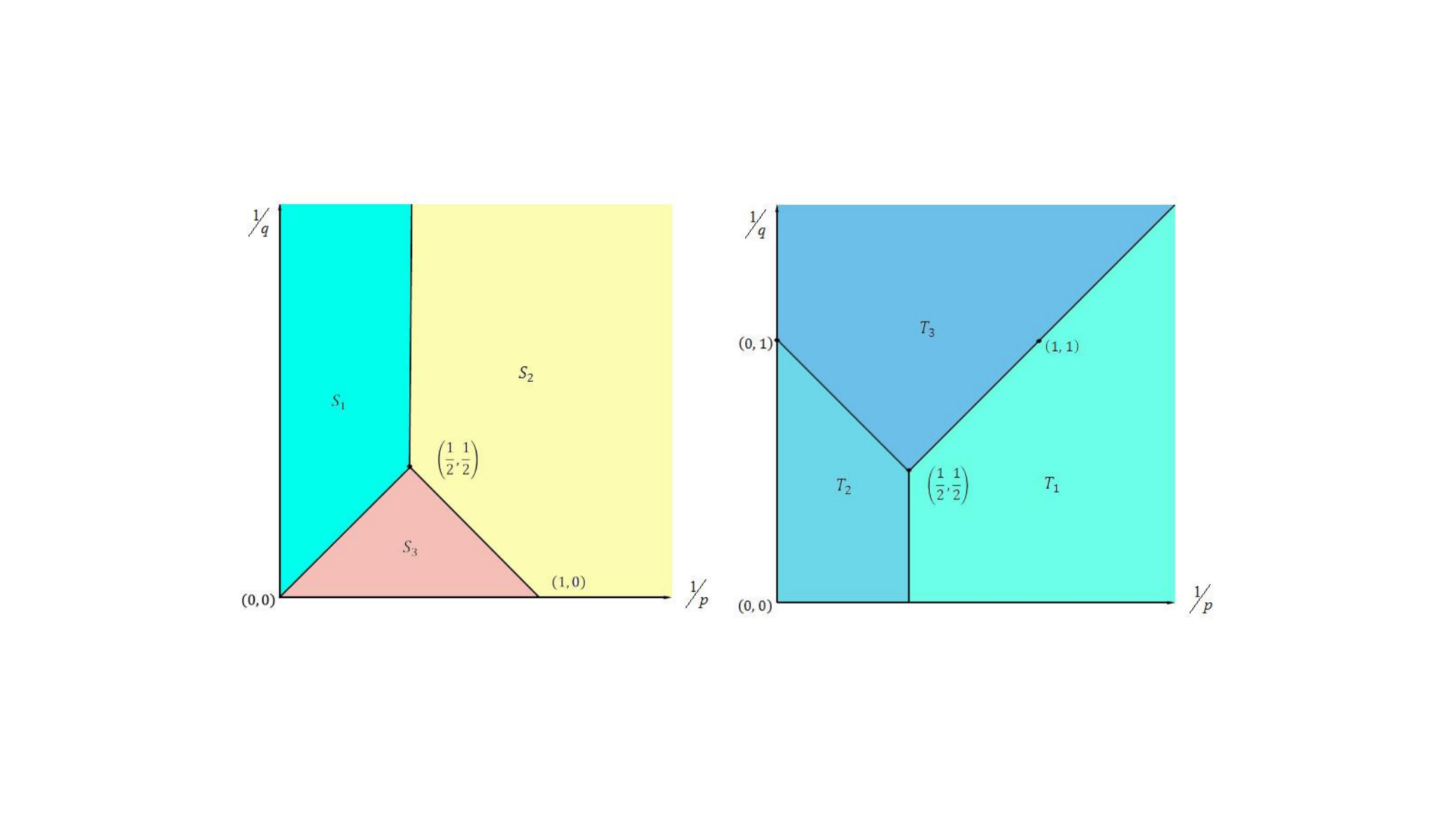}
\begin{minipage}{12cm}
\vspace*{-4cm}\caption{\small The left-hand side figure is for
$\alpha_1\geqslant\alpha_2$; while the right-hand side figure is for
$\alpha_1<\alpha_2$. }\label{figure-embedding}
\end{minipage}
\end{center}
\end{figure}

As $1\le p, q\le \infty$,  some sufficient conditions  for the inclusions between modulation and Besov spaces were obtained by Gr\"obner \cite{Gb},  then Toft \cite{T} improved Gr\"obner's sufficient conditions,  which were proven to be necessary by Sugimoto and Tomita \cite{ST}. Their results were generalized to the cases $0<p,  q \le \infty$ in \cite{WaHe07, WaHu07}.  Gr\"obner \cite{Gb} also considered the inclusions between $\alpha_1$-modulation and $\alpha_2$-modulation spaces for $1\le p, q\le \infty$ and his results are optimal in the cases $(1/p, 1/q)$ is located in the vertices of the square $[0,1]^2$. We will improve Gr\"obner's results in the cases $1\le p, q\le \infty$ and our results also cover the cases $0<p<1$  or $0<  q <1$.

\subsection{Embedding between $\alpha$-modulation spaces}

\begin{thm}\label{embedding}
Let $(\alpha_1,\alpha_2)\in[0,1)\times[0,1)$.  Then
\begin{equation}\label{embedding-dfdfdfdfdfd}
M_{p,q}^{s_1,\alpha_1}\subset M_{p,q}^{s_2,\alpha_2}
\end{equation}
holds if and only if $s_1\geqslant
s_2+\mbox{R}(p,q;\alpha_1,\alpha_2)$.
\end{thm}

\noindent {\bf Remark. \rm In the first versions of the paper, we obtain the sufficiency of Theorem \ref{embedding},  soon after Toft and Wahlberg  \cite{ToWa11} independently considered the embeddings between $\alpha$-modulation and Besov spaces in the cases $1\le p,q \le \infty$ and  they first showed the necessity of Theorem \ref{embedding} in the regions $(1/p,1/q) \in (S_2\cup S_3) \cap [0,1]^2$ (see Fig. 3). After their work we can finally show the necessity of Theorem \ref{embedding} in all cases.}

\begin{proof}({\bf Sufficiency})
For every $k\in\mathbb{Z}^n$,  we introduce
\begin{equation}\label{lambda-k-1}
\Lambda(k) = \{l\in\mathbb{Z}^n:
\square_l^{\alpha_2}\square_k^{\alpha_1}f\neq0\}.
\end{equation}
If $ \square_l^{\alpha_2}\square_k^{\alpha_1}f\neq0 $,   then $k$ and
$l$ satisfy
\begin{equation}\label{embed-0001}
\begin{split}
\langle l\rangle^{\frac{\alpha_2}{1-\alpha_2}}(l_j-C)&<\langle
k\rangle^{\frac{\alpha_1}{1-\alpha_1}}(k_j+C),  \\
\langle l\rangle^{\frac{\alpha_2}{1-\alpha_2}}(l_j+C)&>\langle
k\rangle^{\frac{\alpha_1}{1-\alpha_1}}(k_j-C)
\end{split}
\end{equation}
for all $j=1, 2, \cdots, n$. If $|k|\lesssim1$,  it is easy to see that
$|l|\lesssim1$.  If $|k|\gg1$,  analogous to \eqref{005},  we have
\begin{equation}\label{embed-0002}
\langle l\rangle\sim\langle
k\rangle^{\frac{1-\alpha_2}{1-\alpha_1}}.
\end{equation}
Assume that $p\geqslant1$, $q=1$ and  $s_2=0$, we have
\begin{equation}\label{embed-0003}
\|f\|_{M_{p, 1}^{0, \alpha_2}}\leqslant\sum_{k\in\mathbb{Z}^n}\sum_{l\in\Lambda(k)}\|\square_k^{\alpha_1}\square_l^{\alpha_2}f\|_p\lesssim\sum_k\#\Lambda(k)\|\square_k^{\alpha_1}f\|_p.
\end{equation}
We need an estimate of $\#\Lambda(k)$.
Similar to \eqref{003},  we have
\begin{equation}\label{embed-0004}
\#\Lambda(k)\sim1\vee\langle
k\rangle^{\frac{n(\alpha_1-\alpha_2)}{1-\alpha_1}}.
\end{equation}
If $p=2$,  inserting \eqref{embed-0004} into \eqref{embed-0003} and noticing \eqref{embed-0002},  in view of Jensen's
inequality we get
\begin{equation}\label{embed-0005}
M_{2, 1}^{s_2+0\vee\frac{n(\alpha_1-\alpha_2)}{2(1-\alpha_1)}, \alpha_1}\hookrightarrow
M_{2, 1}^{s_2, \alpha_2}.
\end{equation}
If $p=\infty$ or 1,  from
\eqref{embed-0003}, \eqref{embed-0004}, \eqref{embed-0002}, one can
  directly obtain that
\begin{eqnarray}
M_{\infty, 1}^{s_2+0\vee\frac{n(\alpha_1-\alpha_2)}{1-\alpha_1}, \alpha_1}&\hookrightarrow&
M_{\infty, 1}^{s_2, \alpha_2},  \label{embed-0006}\\
M_{1, 1}^{s_2+0\vee\frac{n(\alpha_1-\alpha_2)}{1-\alpha_1}, \alpha_1}&\hookrightarrow&
M_{1, 1}^{s_2, \alpha_2}. \label{embed-0007}
\end{eqnarray}

\noindent{\bf
Case 1:} $\boldsymbol{\big(\frac{1}{p}, \frac{1}{q}\big)\in\mbox{I}}.$
For any $\theta\in[0, 1]$,  $\big(\frac{\theta}{2}, 1\big)$ is at the
line connecting $\big(\frac12, 1\big)$ and $(0, 1)$. By complex
interpolation between \eqref{embed-0005} and
\eqref{embed-0006},  one has that
\begin{equation}\label{embed-0008}
\begin{split}
M_{\frac{2}{\theta}, 1}^{0\vee\big(1-\frac{\theta}{2}\big)\frac{n(\alpha_1-\alpha_2)}{1-\alpha_1}, \alpha_1}\hookrightarrow
M_{\frac{2}{\theta}, 1}^{0, \alpha_2}.
\end{split}
\end{equation}
Since $\big(\frac{\theta}{2}, 1-\frac{\theta}{2}\big)$ is at the
line segment by connecting  $\big(\frac12, \frac12\big)$ and $(0, 1)$,
A complex interpolation  combined with Proposition
\ref{alpha=2-equivalent-besov} and \eqref{embed-0006}  yields
\begin{equation}\label{embed-0009}
\begin{split}
M_{\frac{2}{\theta}, \frac{2}{2-\theta}}^{0\vee(1-\theta)\frac{n(\alpha_1-\alpha_2)}{1-\alpha_1}}\hookrightarrow
M_{\frac{2}{\theta}, \frac{2}{2-\theta}}^{0, \alpha_2}.
\end{split}
\end{equation}
For any $\big(\frac{1}{p}, \frac{1}{q}\big)\in\mbox{I}$,  we may
suppose that there exists some $(\theta, \eta)\in[0, 1]\times[0, 1]$,  such
that
$$
\frac{1}{p}=\frac{1}{p_1}=\frac{\theta}{2},  \qquad
\frac{1}{q}=1-\frac{\theta\eta}{2}.
$$
Therefore, a complex interpolation between \eqref{embed-0008} and \eqref{embed-0009} implies that
\begin{equation}\label{embed-00100}
\begin{split}
M_{p, q}^{s_2+\big(\frac{1}{q}-\frac{1}{p}\big)[0\vee
n(\alpha_1-\alpha_2)], \alpha_1}\hookrightarrow
M_{p, q}^{s_2, \alpha_2}.
\end{split}
\end{equation}

\noindent{\bf
Case 2:} $\boldsymbol{\big(\frac{1}{p}, \frac{1}{q}\big)\in\mbox{II}}.$
For any $\theta\in[0, 1]$,
$\big(1-\frac{\theta}{2}, 1-\frac{\theta}{2}\big)$ is at the line segment
connecting $\big(\frac12, \frac12\big)$ and $(1, 1)$.
From Proposition \ref{alpha=2-equivalent-besov} and
\eqref{embed-0007},  we see that
\begin{equation}\label{embed-0011}
M_{\frac{2}{2-\theta}, \frac{2}{2-\theta}}^{s_2+(1-\theta)[0\vee
n(\alpha_1-\alpha_2)], \alpha_1}\hookrightarrow
M_{\frac{2}{2-\theta}, \frac{2}{2-\theta}}^{s_2, \alpha_2}.
\end{equation}
Noticing that $\big(1-\frac{\theta}{2}, 1\big)$ is a point at  the line segment
connecting $\big(\frac12, 1\big)$ and $(0, 1)$,  from \eqref{embed-0005}
and \eqref{embed-0007},  we see that
\begin{equation}\label{embed-0012}
M_{\frac{2}{2-\theta}, 1}^{s_2+\big(1-\frac{\theta}{2}\big)[0\vee
n(\alpha_1-\alpha_2)], \alpha_1}\hookrightarrow
M^{s_2, \alpha_2}_{\frac{2}{2-\theta}, 1}.
\end{equation}
Noticing that for any $\big(\frac1p, \frac1q\big)\in\mbox{II}$,  there
exists some $(\theta, \eta)\in[0, 1]\times[0, 1]$ satisfying
$$
\frac{1}{p}=1-\frac{\theta}{2},  \qquad
\frac{1}{q}=1-\frac{\theta\eta}{2},
$$
on the basis of \eqref{embed-0011} and \eqref{embed-0012},  we conclude
that
\begin{equation}\label{embed-0013}
M_{p, q}^{s_2+\big(\frac{1}{p}+\frac{1}{q}-1\big)[0\vee
n(\alpha_1-\alpha_2)], \alpha_1}\hookrightarrow
M_{p, q}^{s_2, \alpha_2}.
\end{equation}
When $\alpha_1\leqslant\alpha_2$,  \eqref{embed-0013} coincides with
\eqref{embed-00100}.

\noindent{\bf
Case 3:} $\boldsymbol{\big(\frac{1}{p}, \frac{1}{q}\big)\in\mbox{I}^{\ast}\cup\mbox{II}^{\ast}}.$
When $\big(\frac{1}{p}, \frac{1}{q}\big)\in\mbox{I}^{\ast}$,
$\big(\frac{1}{p^*}, \frac{1}{q^*}\big)$ is in $\mbox{I}$. From
\eqref{embed-00100},  we know
\begin{equation*}
M_{p^{\ast}, q^{\ast}}^{-s_2, \alpha_2}\hookrightarrow
M_{p^{\ast}, q^{\ast}}^{-s_2-\big(\frac1p-\frac1q\big)[0\vee
n(\alpha_2-\alpha_1)], \alpha_1}.
\end{equation*}
The duality of $\alpha$-modulation space implies that
\begin{equation}\label{embed-0015}
M_{p, q}^{s_2+\big(\frac{1}{p}-\frac{1}{q}\big)[0\vee
n(\alpha_2-\alpha_1)], \alpha_1}\hookrightarrow
M_{p, q}^{s_2, \alpha_2}.
\end{equation}
When $\big(\frac{1}{p}, \frac{1}{q}\big)\in\mbox{II}^{\ast}$,
by \eqref{embed-0013} and duality one has that
\begin{equation}\label{embed-0016}
M_{p, q}^{s_2+\big(1-\frac{1}{p}-\frac{1}{q}\big)[0\vee
n(\alpha_2-\alpha_1)], \alpha_1}\hookrightarrow
M_{p, q}^{s_2, \alpha_2}.
\end{equation}
When $\alpha_1>\alpha_2$,  \eqref{embed-0016} coincides with
\eqref{embed-0015}.

\noindent{\bf
Case 4:} $\boldsymbol{\big(\frac{1}{p}, \frac{1}{q}\big)\in\mbox{III}\cup\mbox{III}^*}.$
We may assume that for any
$\big(\frac1p, \frac1q\big)\in\mbox{III}\;(\mbox{III}^*)$, there exists a
$\eta\in(0, 1)$ satisfying
$$
\frac{1}{q}=\eta\left(1-\frac{1}{p}\right)+\frac{1-\eta}{p}.
$$
Notice that $\big(\frac1p, \frac1p\big)$ and $\big(\frac1p, 1-\frac1p\big)$
are at the boundaries of $\mbox{II}$ and $\mbox{I}^*$ ($\mbox{II}^*$
and $\mbox{I}$),  respectively. If
$\big(\frac{1}{p}, \frac{1}{q}\big)\in\mbox{III}$, a complex
interpolation between \eqref{embed-0013} and
\eqref{embed-0015} yields
\begin{equation}\label{embed-0020}
M_{p, q}^{s_2+\left[n(\alpha_1-\alpha_2)\big(\frac{1}{p}+\frac{1}{q}-1\big)\vee
n(\alpha_2-\alpha_1)\big(\frac{1}{p}-\frac{1}{q}\big)\right], \alpha_1}\hookrightarrow
M_{p, q}^{s_2, \alpha_2};
\end{equation}
If $\big(\frac{1}{p}, \frac{1}{q}\big)\in\mbox{III}^*$,
a complex interpolation between \eqref{embed-00100} and
\eqref{embed-0016} yields
\begin{equation}\label{embed-0019}
M_{p, q}^{s_2+\left[n(\alpha_1-\alpha_2)\big(\frac{1}{q}-\frac{1}{p}\big)\vee
n(\alpha_2-\alpha_1)\big(1-\frac{1}{p}-\frac{1}{q}\big)\right], \alpha_1}\hookrightarrow
M_{p, q}^{s_2, \alpha_2}.
\end{equation}
When $\alpha_1>\alpha_2$,  \eqref{embed-0020} and \eqref{embed-0019}
coincide with \eqref{embed-0013} and \eqref{embed-00100},
respectively. When $\alpha_1\leqslant\alpha_2$,  \eqref{embed-0020}
and \eqref{embed-0019} coincide with \eqref{embed-0015} and
\eqref{embed-0016},  respectively.

\noindent{\bf
Case 5:} $\boldsymbol{\big(\frac1p, \frac1q\big)\in\widehat{\mbox{I}}.}$
Imitating the proof as in the counterpart of Theorem
\ref{dilationP},  we can easily get
$$
M_{\infty, q}^{s_2+0\vee\frac{n(\alpha_1-\alpha_2)}{1-\alpha_1}\frac1q, \alpha_1}\subset
M_{\infty, q}^{s_2, \alpha_2},  \quad
M_{2, q}^{s_2+0\vee\frac{n(\alpha_1-\alpha_2)}{1-\alpha_1}\big(\frac1q-\frac12\big), \alpha_1}\subset
M_{2, q}^{s_2, \alpha_2}.
$$
A complex interpolation yields
\begin{equation}\label{embedding-555}
M_{p, q}^{s_2+0\vee\frac{n(\alpha_1-\alpha_2)}{1-\alpha_1}\big(\frac1q-\frac1p\big), \alpha_1}\subset
M_{p, q}^{s_2, \alpha_2}.
\end{equation}
\eqref{embedding-555} coincides with \eqref{embed-00100}.

\noindent{\bf
Case 6:} $\boldsymbol{\big(\frac1p, \frac1q\big)\in\widehat{\mbox{III}}.}$
From \eqref{eta-2},  as well as \eqref{embed-0002},  we see that
\begin{equation*}
\mbox{diam}\;\mbox{supp}\mathscr{F}[\mathscr{F}^{-1}\eta_l^{\alpha_2}(x-\cdot)\square_k^{\alpha_1}f(\cdot)]\lesssim\langle
k\rangle^{\frac{\alpha_1\vee\alpha_2}{1-\alpha_1}},
\end{equation*}
In view of Proposition \ref{convolution},
\begin{equation}\label{cnmmmmmmmmmmm}
\|\square_l^{\alpha_2}\square_k^{\alpha_1}f\|_p\lesssim\langle
k\rangle^{\big(\frac1p-1\big)\frac{0\vee
n(\alpha_1-\alpha_2)}{1-\alpha_1}}\|\square_k^{\alpha_1}f\|_p.
\end{equation}
Inserting
\eqref{cnmmmmmmmmmmm}, \eqref{embed-0002}, \eqref{embed-0004},  from
the embedding $\ell^p\subset\ell^1$ and with the aid of Jensen's
inequality,  we have
\begin{equation}\label{embedding-60000}
\begin{split}
\|f\|_{M_{p, q}^{s_2, \alpha_2}}
&\lesssim\left[\sum_{l\in\mathbb{Z}^n}\langle
l\rangle^{\frac{s_2q}{1-\alpha_2}}\left(\sum_{k\in\Lambda(l)}\langle
k\rangle^{(1-p)\frac{0\vee
n(\alpha_1-\alpha_2)}{1-\alpha_1}}\|\square_k^{\alpha_1}f\|_p^p\right)^{\frac{q}{p}}\right]^{\frac1q}
\\
&\lesssim\left(\sum_{l}\langle
l\rangle^{\frac{s_2q}{1-\alpha_2}+\frac{0\vee
n(\alpha_2-\alpha_1)}{1-\alpha_2}\big(\frac{q}{p}-1\big)}\sum_{k\in\Lambda(l)}\langle
k\rangle^{nq\big(\frac1p-1\big)\frac{0\vee(\alpha_1-\alpha_2)}{1-\alpha_1}}\|\square_k^{\alpha_1}f\|_p^q\right)^{\frac1q}
\\
&\lesssim\left(\sum_k\langle
k\rangle^{\frac{s_2q}{1-\alpha_1}+\frac{0\vee
n(\alpha_2-\alpha_1)}{1-\alpha_1}\big(\frac{q}{p}-1\big)+nq\frac{0\vee(\alpha_1-\alpha_2)}{1-\alpha_1}\big(\frac1p+\frac1q-1\big)}\|\square_k^{\alpha_1}f\|_p^q\right)^{\frac1q}.
\end{split}
\end{equation}
When $\alpha_1\leqslant\alpha_2$,  \eqref{embedding-60000} gives
\begin{equation}\label{embedding-60000-1}
M_{p, q}^{s_2+n(\alpha_2-\alpha_1)\big(\frac1p-\frac1q\big), \alpha_1}\subset
M_{p, q}^{s_2, \alpha_2};
\end{equation}
whereas when $\alpha_1>\alpha_2$,  \eqref{embedding-60000} gives
\begin{equation}\label{embedding-60000-2}
M_{p, q}^{s_2+n(\alpha_1-\alpha_2)\big(\frac1p+\frac1q-1\big), \alpha_1}\subset
M_{p, q}^{s_2, \alpha_2}.
\end{equation}
\eqref{embedding-60000-1} and \eqref{embedding-60000-2} coincide
with \eqref{embed-0015} and \eqref{embed-0013},  respectively.

\noindent{\bf Case
7:} $\boldsymbol{\big(\frac1p, \frac1q\big)\in\widehat{\mbox{II}}.}$
This is a consequence of the results in Cases 5 and 6 by complex
interpolation.
\end{proof}

\subsection{Embedding between Besov space and $\alpha$-modulation space}

In this section,  we study the   embedding between
1-modulation space and   $\alpha$-modulation spaces.
In an analogous way to the previous subsection,  we start with the embedding
for the same indices $p, q$.
\begin{thm}\label{embedmb}
Let $\alpha\in[0,1)$. Then
$
B_{p,q}^{s_1}\subset M_{p,q}^{s_2,\alpha}
$
holds if and only if $s_1\geqslant s_2+\mbox{R}(p,q;1,\alpha)$. Conversely,
$
M_{p,q}^{s_1,\alpha}\subset B_{p,q}^{s_2}
$
holds if and only if  $s_1 \geqslant s_2+\mbox{R}(p,q;\alpha,1)$.
\end{thm}

\begin{proof}({\bf sufficiency})
For
every $j\in\mathbb{Z}_+$, we introduce
\begin{equation}\label{lambda-j-dingyi}
\Lambda(j)=\{k\in\mathbb{Z}^n:\square_k^{\alpha}\triangle_jf\neq0,\forall
f\in\mathscr{S}'(\mathbb{R}^n)\};
\end{equation}
and for every $k\in\mathbb{Z}^n$, we introduce
\begin{equation}\label{lambda-k-dingyi}
\Lambda(k)=\{j\in\mathbb{Z}_+:\square_k^{\alpha}\triangle_jf\neq0,\forall
f\in\mathscr{S}'(\mathbb{R}^n)\}.
\end{equation}
To a $j\in\mathbb{Z}_+$, for any $k\in\Lambda(j)$, it is easy to see
that the quantitative relationship between $k$ and $j$ is
\begin{equation}\label{embedding-mb-fffffff}
\langle k\rangle^{\frac{1}{1-\alpha}}\sim2^j.
\end{equation}
When $p\geqslant1, q=1$ and $s_2=0$,  we have
\begin{equation}\label{embedding-mb-ffffffg}
\|f\|_{M_{p, 1}^{0, \alpha}}\leqslant
\sum_{k\in\mathbb{Z}^n}\sum_{j\in\Lambda(k)}\|\square_k^{\alpha}\triangle_jf\|_p=
\sum_{j\in\mathbb{Z}_+}\sum_{k\in\Lambda(j)}\|\triangle_j\square_k^{\alpha}f\|_p.
\end{equation}
For any $k\in\mathbb{Z}^n$ and any $j\in\mathbb{Z}_+$,   it is easy to see
\begin{equation}\label{embedding-mb-ffffffh}
\#\Lambda(k)\sim1,  \quad \#\Lambda(j)\sim2^{jn(1-\alpha)}.
\end{equation}
Thus when $p=2$,  combining \eqref{embedding-mb-ffffffg},
 \eqref{embedding-mb-ffffffh},
also with the aid of Jensen's inequality,  we get
\begin{equation}\label{embedding-MB-1}
B_{2, 1}^{s+\frac{n(1-\alpha)}{2}}\hookrightarrow M_{2, 1}^{s, \alpha}.
\end{equation}
If $p=1$ or $\infty$,  combining
  \eqref{embedding-mb-ffffffh}, \eqref{embedding-mb-ffffffg},
we get
\begin{eqnarray}
B_{1, 1}^{s+n(1-\alpha)}&\hookrightarrow&M_{1, 1}^{s, \alpha},  \label{embedding-MB-2}\\
B_{\infty, 1}^{s+n(1-\alpha)}&\hookrightarrow&
M_{\infty, 1}^{s, \alpha}. \label{embedding-MB-3}
\end{eqnarray}

\noindent${\bf
Case 1.} \ \boldsymbol{\big(\frac{1}{p}, \frac{1}{q}\big)\in\mbox{I}}.$
For any $\theta\in[0, 1]$, a complex interpolation between
\eqref{embedding-MB-1} and \eqref{embedding-MB-3} yields
\begin{equation}\label{embedding-MB-4}
B_{\frac{2}{\theta}, 1}^{s+\big(1-\frac{\theta}{2}\big)n(1-\alpha)}\hookrightarrow
M_{\frac{2}{\theta}, 1}^{s, \alpha};
\end{equation}
while combined with Proposition \ref{alpha=2-equivalent-besov} and
\eqref{embedding-MB-3},  yields
\begin{equation}\label{embedding-MB-5}
B_{\frac{2}{\theta}, \frac{2}{2-\theta}}^{s+(1-\theta)n(1-\alpha)}\hookrightarrow
M_{\frac{2}{\theta}, \frac{2}{2-\theta}}^{s, \alpha}.
\end{equation}
In analogy to \eqref{embed-00100},  we get from
\eqref{embedding-MB-4}, \eqref{embedding-MB-5} that
\begin{equation}\label{embedding-MB-6}
B_{p, q}^{s+n\big(\frac{1}{q}-\frac{1}{p}\big)(1-\alpha)}\hookrightarrow
M_{p, q}^{s, \alpha}.
\end{equation}

Conversely,  when we encounter the embedding of $\alpha$-modulation
spaces into Besov spaces,  for $2\leqslant p\leqslant\infty$,
considering \eqref{embedding-mb-ffffffh},  we have
\begin{equation}\label{embedding-MB-7}
\begin{split}
\|f\|_{B_{p, 1}^0}&\leqslant\sum_{j\in\mathbb{Z}^+}\sum_{k\in\Lambda(j)}\|\square_k^{\alpha}\triangle_jf\|_p
\lesssim\sum_{k\in\mathbb{Z}^n}\|\square_k^{\alpha}f\|_p=\|f\|_{M_{p, 1}^{0, \alpha}},
\end{split}
\end{equation}
which gives
\begin{equation}\label{embedding-MB-8}
M_{p, q}^{s, \alpha}\hookrightarrow B_{p, q}^s.
\end{equation}

\noindent${\bf
Case 2.} \ \boldsymbol{\big(\frac{1}{p}, \frac{1}{q}\big)\in\mbox{II}}.$
For any $\theta\in[0, 1]$,  a complex interpolation between
\eqref{embedding-MB-1} and \eqref{embedding-MB-2} yields
\begin{equation}\label{embedding-MB-10}
B_{\frac{2}{2-\theta}, 1}^{s+\big(1-\frac{\theta}{2}\big)n(1-\alpha)}\hookrightarrow
M_{\frac{2}{2-\theta}, 1}^{s, \alpha}.
\end{equation}
From Proposition \ref{alpha=2-equivalent-besov} and
\eqref{embedding-MB-2} it follows that
\begin{equation}\label{embedding-MB-11}
B_{\frac{2}{2-\theta}, \frac{2}{2-\theta}}^{s+(1-\theta)n(1-\alpha)}\hookrightarrow
M_{\frac{2}{2-\theta}, \frac{2}{2-\theta}}^{s, \alpha}.
\end{equation}
Analogous to \eqref{embed-0013},  one can conclude from
\eqref{embedding-MB-10} and \eqref{embedding-MB-11} that
\begin{equation}\label{embedding-MB-12}
B_{p, q}^{s+n\big(\frac{1}{q}+\frac{1}{p}-1\big)(1-\alpha)}\hookrightarrow
M_{p, q}^{s, \alpha}.
\end{equation}
Considering the embedding of $\alpha$-modulation spaces into
Besov spaces, \eqref{embedding-MB-8} still holds if $(1/p,1/q) \in II.$

\noindent${\bf
Case 3.} \ \boldsymbol{\big(\frac{1}{p}, \frac{1}{q}\big)\in\mbox{I}^{\ast}\cup\mbox{II}^{\ast}}.$
Since
$\big(\frac{1}{p^*}, \frac{1}{q^*}\big)\in\mbox{I}\cup\mbox{II}$,
from \eqref{embedding-MB-8}, we see that
$M_{p^{\ast}, q^{\ast}}^{-s, \alpha}\hookrightarrow
B_{p^{\ast}, q^{\ast}}^{-s}$. Thus, the duality between
$B_{p, q}^s$ and $B_{p^{\ast}, q^{\ast}}^{-s}$,  as well as between
$M_{p, q}^{s, \alpha}$ and $M_{p^{\ast}, q^{\ast}}^{-s, \alpha}$,
implies that
\begin{equation}\label{embedding-MB-13}
B_{p, q}^s\hookrightarrow M_{p, q}^{s, \alpha}.
\end{equation}
Conversely,  if one considers the embedding of $\alpha$-modulation
space into Besov space, it follows from Theorem \ref{duality} and \eqref{embedding-MB-6} that
\begin{equation}\label{embedding-mb-fdfdfdfdf}
M_{p, q}^{s+n(\alpha-1)\big(\frac1q-\frac1p\big), \alpha}\subset
B_{p, q}^s, \ \ \left(\frac1p, \frac1q\right)\in\mbox{I}^*;
\end{equation}
and from
\eqref{embedding-MB-12} it follows that
\begin{equation}\label{embedding-mb-dfdfdfdfdf}
M_{p, q}^{s+n(\alpha-1)\big(\frac1p+\frac1q-1\big), \alpha}\subset
B_{p, q}^s, \ \ \left(\frac1p, \frac1q\right)\in\mbox{II}^*.
\end{equation}

\noindent${\bf
Case 4.} \ \boldsymbol{\big(\frac{1}{p}, \frac{1}{q}\big)\in\mbox{III}\cup\mbox{III}^*}.$
If $\big(\frac{1}{p}, \frac{1}{q}\big)\in\mbox{III}^*$,
\eqref{embedding-MB-6} and \eqref{embedding-MB-13} contain that
\begin{equation*}
\begin{split}
B_{p, \frac{p}{p-1}}^{s+n\big(1-\frac{2}{p}\big)(1-\alpha)}\hookrightarrow
M_{p, \frac{p}{p-1}}^{s, \alpha},  \quad B_{p, p}^s\hookrightarrow
M_{p, p}^{s, \alpha}.
\end{split}
\end{equation*}
By interpolation we have
\begin{equation}\label{embedding-MB-15}
B_{p, q}^{s+n\big(\frac{1}{q}-\frac{1}{p}\big)(1-\alpha)}\hookrightarrow
M_{p, q}^{s, \alpha},
\end{equation}
which coincides with \eqref{embedding-MB-6}. If
$\big(\frac{1}{p}, \frac{1}{q}\big)\in\mbox{III}$,
\eqref{embedding-MB-12} and \eqref{embedding-MB-13} imply that
\begin{equation*}
\begin{split}
B_{p, \frac{p}{p-1}}^s\subset M_{p, \frac{p}{p-1}}^{s, \alpha},  \quad
B_{p, p}^{s+n\big(\frac{2}{p}-1\big)(1-\alpha)}\subset
M_{p, q}^{s, \alpha}.
\end{split}
\end{equation*}
By interpolation,
\begin{equation}\label{embedding-MB-16}
B_{p, q}^{s+n\big(\frac{1}{q}+\frac{1}{p}-1\big)(1-\alpha)}\hookrightarrow
M_{p, q}^{s, \alpha},
\end{equation}
which coincides with \eqref{embedding-MB-12}.

Conversely,  considering the embedding of $\alpha$-modulation
space into Besov space,  in view of Theorem \ref{duality} and
\eqref{embedding-MB-15}  we have
\begin{equation}\label{embedding-mb-gggggggggdg}
M_{p, q}^{s+n(\alpha-1)\big(\frac1q-\frac1p\big), \alpha}\subset
B_{p, q}^s, \ \ \left(\frac1p, \frac1q\right)\in\mbox{III};
\end{equation}
while for $\big(\frac1p, \frac1q\big)\in\mbox{III}^*$,  from
\eqref{embedding-MB-16},  we have
\begin{equation}\label{embedding-mb-hggggggggggg}
M_{p, q}^{s+n(\alpha-1)\big(\frac1p+\frac1q-1\big), \alpha}\subset
B_{p, q}^s.
\end{equation}
\eqref{embedding-mb-gggggggggdg} and
\eqref{embedding-mb-hggggggggggg} coincide with
\eqref{embedding-mb-fdfdfdfdf} and \eqref{embedding-mb-dfdfdfdfdf},
respectively.

\noindent${\bf
Case 5:}\boldsymbol{\big(\frac1p, \frac1q\big)\in\widehat{\mbox{I}}.}$
For the embedding of Besov space into $\alpha$-modulation space,
imitating the argument in Theorem \ref{embedding},  we get
\begin{equation*}
M_{\infty, q}^{s+n(1-\alpha)\frac1q, 1}\subset M_{p, q}^{s, \alpha},
\quad
M_{2, q}^{s+n(1-\alpha)\big(\frac1q-\frac12\big), 1}\subset
M_{p, q}^{s, \alpha}.
\end{equation*}
From them,  we interpolate out
\begin{equation}\label{embedding-MB-18}
M_{p, q}^{s+n(1-\alpha)\big(\frac1q-\frac1p\big), 1}\subset
M_{p, q}^{s, \alpha},
\end{equation}
which coincides with \eqref{embedding-MB-6}. Conversely for the
embedding of $\alpha$-modulation space into Besov space,  we have
\begin{equation}\label{embedding-MB-19}
M_{p, q}^{s, \alpha}\subset B_{p, q}^s,
\end{equation}
which coincides with \eqref{embedding-MB-8}.

\noindent{\bf Case
6.} \  $\boldsymbol{\big(\frac1p, \frac1q\big)\in\widehat{\mbox{III}}.}$
If $\Box^\alpha_k  \triangle_j \neq 0$, then
\begin{equation}\label{jjjjjggggggggggg}
\mbox{diam}\;\mbox{supp}\mathscr{F}[\mathscr{F}^{-1}\eta_k^{\alpha}(x-\cdot)\triangle_jf(\cdot)]\lesssim\langle
k\rangle^{\frac{1}{1-\alpha}}\sim 2^j,
\end{equation}
So, in view of Proposition \ref{convolution} we have
\begin{equation}\label{hhhhhhhhhhhh}
\|\square_k^{\alpha}\triangle_jf\|_p\lesssim\langle
k\rangle^{\frac{n}{1-\alpha}\big(\frac1p-1\big)}\langle
k\rangle^{-\frac{n\alpha}{1-\alpha}\big(\frac1p-1\big)}\|\triangle_jf\|_p\sim2^{jn(1-\alpha)\big(\frac1p-1\big)}\|\triangle_jf\|_p.
\end{equation}
From \eqref{hhhhhhhhhhhh}, \eqref{embedding-mb-ffffffh},  $\ell^p\subset\ell^1$ and Jensen's inequality it follows that
\begin{equation}\label{qqqqqqqqqqqqqqqqqqqqqqqqq}
\begin{split}
\|f\|_{M_{p, q}^{s, \alpha}}&\lesssim\left[\sum_{k\in\mathbb{Z}^n}\langle
k\rangle^{\frac{sq}{1-\alpha}}\left(\sum_{j\in\Lambda(k)}\|\square_k^{\alpha}\triangle_jf\|_p^p\right)^{\frac{q}{p}}\right]^{\frac1q}
\\
&\lesssim\left(\sum_k\langle
k\rangle^{\left[\frac{s}{1-\alpha}+n\big(\frac1p-1\big)\right]q}\sum_{j\in\Lambda(k)}\|\triangle_jf\|_p^q\right)^{\frac1q}
\\
&\lesssim\left(\sum_j2^{j\left[s+n(1-\alpha)\big(\frac1p-1\big)\right]q}\sum_{k\in\Lambda(j)}\|\triangle_jf\|_p^q\right)^{\frac1q}
\lesssim\|f\|_{B_{p, q}^{s+n(1-\alpha)\big(\frac1p+\frac1q-1\big)}},
\end{split}
\end{equation}
which implies that
\begin{equation}\label{embedding-MB-20}
B_{p, q}^{s+n(1-\alpha)\big(\frac1p+\frac1q-1\big), \alpha}\subset
M_{p, q}^{s, \alpha}.
\end{equation}
It coincides with \eqref{embedding-MB-12}.

Conversely,  when we study the embedding of $\alpha$-modulation space
into Besov space,  in analogy to \eqref{jjjjjggggggggggg},  we see
\begin{equation*}
\mbox{diam}\;\mbox{supp}\mathscr{F}[\mathscr{F}^{-1}\varphi_j(x-\cdot)\square_k^{\alpha}f(\cdot)]\lesssim2^j
\end{equation*}
In contrast to \eqref{hhhhhhhhhhhh},  we conclude
\begin{equation}\label{hhhhhhhhhhhp}
\|\triangle_j\square_k^{\alpha}f\|_p\lesssim2^{jn\big(\frac1p-1\big)}2^{-jn\big(\frac1p-1\big)}\|\square_k^{\alpha}f\|_p=\|\square_k^{\alpha}f\|_p.
\end{equation}
Inserting \eqref{hhhhhhhhhhhp},  the substitution for
\eqref{qqqqqqqqqqqqqqqqqqqqqqqqq} is
\begin{equation*}
\begin{split}
\|f\|_{B_{p, q}^s}&\lesssim\left[\sum_{j\in\mathbb{Z}^+}2^{jsq}\left(\sum_{j\in\Lambda(j)}\|\triangle_j\square_k^{\alpha}f\|_p^p\right)^{\frac{q}{p}}\right]^{\frac1q}
\\
&\lesssim\left(\sum_{j}2^{jsq}\#\Lambda(j)^{\frac{q}{p}-1}\sum_{k\in\Lambda(j)}\|\triangle_j\square_k^{\alpha}f\|_p^q\right)^{\frac1q}
\\
&\lesssim\left(\sum_k\langle
k\rangle^{\left[s+n(1-\alpha)\big(\frac1p-\frac1q\big)\right]q}\|\square_k^{\alpha}f\|_p^q\right)^{\frac1q}
\lesssim\|f\|_{M_{p, q}^{s+n(1-\alpha)\big(\frac1p-\frac1q\big)}},
\end{split}
\end{equation*}
which implies that
\begin{equation}\label{embedding-MB-21}
M_{p, q}^{s+n(\alpha-1)\big(\frac1q-\frac1p\big), \alpha}\subset
M_{p, q}^{s, 1}.
\end{equation}
It coincides with \eqref{embedding-mb-fdfdfdfdf}.

\noindent${\bf
Case 7:}\boldsymbol{\big(\frac1p, \frac1q\big)\in\widehat{\mbox{II}}.}$
It is interpolated out from Case 5 and Case 6.
\end{proof}

\section{Multiplication algebra}

It is well known that
$B_{p,q}^s$ is a multiplication algebra if $s>n/p$, cf. \cite{triebel}. But for
$\alpha$-modulation space, this issue is much more complicated. The
  regularity indices  for which $M_{p,q}^{s,\alpha}$ constitutes
a multiplication algebra, are quite different from those of Besov and modulation spaces.

We introduce a parameter, denoted by $s_0=s_0(p,q;\alpha)$, to
describe   the regularity for  which
$M_{p,q}^{s,\alpha}$ with $s>s_0$ forms a multiplication algebra.  Denote (see Figure \ref{figure-algebra})
$$
\mbox{D}_1=\left\{\big(\tfrac1p,\tfrac1q\big)\in\mathbb{R}_+^2:
\tfrac1q\geqslant\tfrac2p,\tfrac1p\leqslant\tfrac12\right\}, \quad D_2= \mathbb{R}_+^2\setminus D_1
$$
and
\begin{equation*}
s_0=
\begin{cases}
\frac{n\alpha}{p}+n(1-\alpha)\big(1-1\wedge\frac1q\big)+\frac{n\alpha(1-\alpha)}{2-\alpha}\big(\frac1q-\frac2p\big), & \quad \big(\frac1p,\frac1q\big)\in\mbox{D}_1; \\
\frac{n\alpha}{p}+n(1-\alpha)\big(1\vee\frac1p\vee\frac1q-\frac1q\big)+\frac{n\alpha(1-\alpha)}{2-\alpha}\big(1\vee\frac1p\vee\frac1q-1\big),
& \quad
\big(\frac1p,\frac1q\big)\in \mbox{D}_2.
\end{cases}
\end{equation*}

\begin{figure}
\begin{center}
\includegraphics[width=10.5cm,height=10cm]{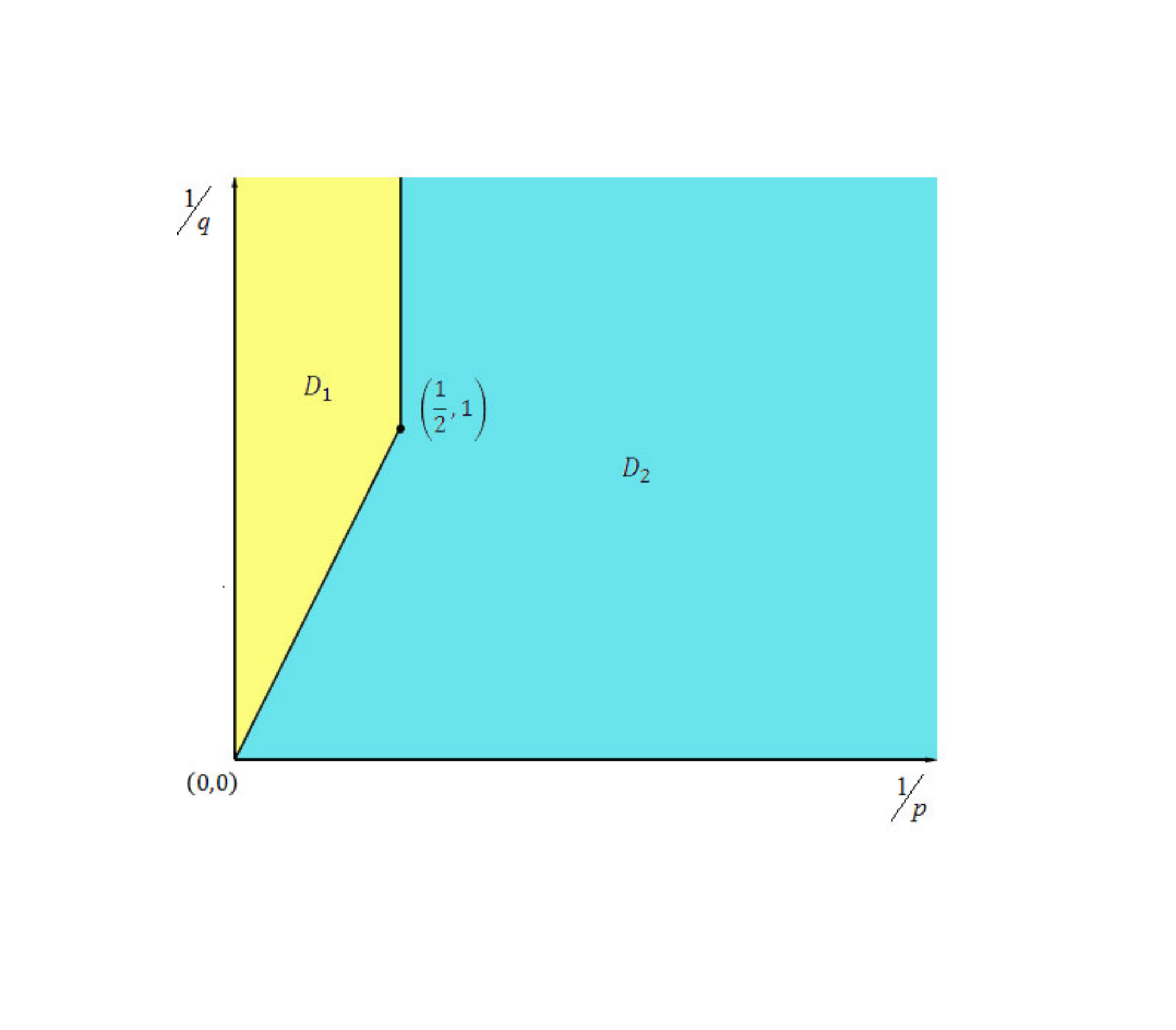}
\vspace*{-2cm}\caption{\small Distribution of $s_0$}\label{figure-algebra}
\end{center}
\end{figure}
\begin{thm}\label{algebra}
If $s>s_0$, then $M_{p,q}^{s,\alpha}$ is a multiplication algebra,
which is equivalent to say that for any $f,g\in M_{p,q}^{s,\alpha}$,
we have
\begin{equation}\label{algebra-inequality}
\|fg\|_{M_{p,q}^{s,\alpha}}\lesssim\|f\|_{M_{p,q}^{s,\alpha}}\|g\|_{M_{p,q}^{s,\alpha}}.
\end{equation}
\end{thm}
In Section 7 we will give some counterexamples to show that $s_0$ is sharp if $(1/p, 1/q) \in  \mbox{D}_2 \cap \{p\ge 1\}$.  When $\big(1/p, \; 1/q \big)\in D_1$, it is not very clear for us to know the sharp low bound of the index $s$ for which $M_{p,q}^{s,\alpha}$ constitutes
a multiplication algebra.
As a straightforward consequence of Theorem \ref{algebra}, we have the following result for which $M^s_{p,q}$ is an algebra.

\begin{cor}\label{algebramod}
Assume that
\begin{equation*}
s>
\begin{cases}
 n \big(1-1\wedge\frac1q\big), & \quad \big(\frac1p,\frac1q\big)\in\mbox{D}_1; \\
 n \big(1\vee\frac1p\vee\frac1q-\frac1q\big),
& \quad
\big(\frac1p,\frac1q\big)\in\mathbb{R}_+^2\backslash\mbox{D}_1.
\end{cases}
\end{equation*}
Then $M_{p,q}^{s}$ is a multiplication algebra,
i.e.,
\begin{equation}\label{algebra-inequalitymod}
\|fg\|_{M_{p,q}^{s}}\lesssim\|f\|_{M_{p,q}^{s}}\|g\|_{M_{p,q}^{s}}
\end{equation}
holds for all $f,g\in M^s_{p,q}.$
\end{cor}

A natural long standing question on modulation, $\alpha$-modulation and Besov spaces is: Can we reformulate $\alpha$-modulation spaces by interpolations  between modulation and Besov spaces, say,
\begin{align} \label{intques}
M^{s, \alpha}_{p,q} = ( M^{s_0, 0}_{p,q}, \ M^{s_1, 1}_{p,q})_\alpha, \quad \mbox{if} \ \   s= (1-\alpha)s_0 + \alpha s_1 ?
\end{align}
The answer is negative  at least for some special cases. Indeed, we see $M^{s_1, 1}_{p,q}$ and $M^{s_0, 0}_{p,q}$ are algebra if $s_1>n/p$ and $s_0> 0$. If \eqref{intques} holds, then $
M^{s, \alpha}_{p,q}  $ is an algebra if $s>n\alpha/p$, however, this is not true if $1<p<2, \ 0< q<1$, see Section 7. \\

\begin{cor}\label{algebraintque}
Let $0<\alpha<1$. Then \eqref{intques} does not hold if $1<p<2, \ 0< q<1$ and $s_0=0_+, \ s_1=n/p$.
\end{cor}

\noindent \begin{proof} [Proof  of Theorem \ref{algebra}.]
We start with some notations and basic conclusions. For every
$(k^{(1)},k^{(2)})\in\mathbb{Z}^{2n}$, we introduce
\begin{equation}
\Lambda\big(k^{(1)},k^{(2)}\big)=\left\{k\in\mathbb{Z}^n:\square_k^{\alpha}(\square_{k^{(1)}}^{\alpha}f\square_{k^{(2)}}^{\alpha} g)\neq 0 \right\};
\end{equation}
and for every $k\in\mathbb{Z}^n$, we introduce
\begin{equation*}
\Lambda(k)=\left\{\big(k^{(1)},k^{(2)}\big)\in\mathbb{Z}^{2n}:\square_k^{\alpha}(\square_{k^{(1)}}^{\alpha}f\square_{k^{(2)}}^{\alpha} g)\neq 0
 \right\}.
\end{equation*}
It is worth to mention that $\Lambda\big(k^{(1)},k^{(2)}\big)$ and $\Lambda(k)$ are independent of $f$ and $g$.  From \eqref{eta-2} we see that for any $k\in\Lambda(k^{(1)},k^{(2)})$, or $(k^{(1)},k^{(2)})  \in\Lambda (k)$,  $k^{(1)},k^{(2)}$ and $k$ satisfy
\begin{subequations}\label{rela}
\begin{align}
&\langle k\rangle^{\frac{\alpha}{1-\alpha}}(k_j-C)<\langle
k^{(1)}\rangle^{\frac{\alpha}{1-\alpha}}(k^{(1)}_j+C)+\langle
k^{(2)}\rangle^{\frac{\alpha}{1-\alpha}}(k^{(2)}_j+C), \label{rela-1}\\
&\langle k\rangle^{\frac{\alpha}{1-\alpha}}(k_j+C)>\langle
k^{(1)}\rangle^{\frac{\alpha}{1-\alpha}}(k^{(1)}_j-C)+\langle
k^{(2)}\rangle^{\frac{\alpha}{1-\alpha}}(k^{(2)}_j-C)
\label{rela-2}
\end{align}
\end{subequations}
for $j=1,2,\cdots,n$. Let $ \langle k_{\max} \rangle    $, $ \langle k_{\rm min} \rangle $  and $ \langle k_{\rm med} \rangle  $ to be the maximal, minimal and medial ones in $\langle k \rangle, \langle k^{(1)} \rangle $ and $\langle k^{(2)} \rangle$, respectively.  If \eqref{rela} holds, then
\begin{equation} \label{maxmed}
 \langle k_{\max} \rangle \sim  \langle k_{\rm med} \rangle, \ \ \# \big\{ k_{\max}: \ k_{\max}, \ k_{\rm med}, \ k_{\min} \ \mbox{satisfy}  \ \eqref{rela} \big\}\lesssim 1.
\end{equation}
We write
\begin{equation}
\begin{split}
K_j\big(k^{(1)},k^{(2)}\big) &=\langle
k^{(1)}\rangle^{\frac{\alpha}{1-\alpha}}k^{(1)}_j+\langle
k^{(2)}\rangle^{\frac{\alpha}{1-\alpha}}k^{(2)}_j; \\
K\big(k^{(1)},k^{(2)}\big) &=\max_{1\leqslant j\leqslant
n}|K_j(k^{(1)},k^{(2)})|.
\end{split}
\end{equation}
In order to get  more precise estimates, we divide
$\mathbb{Z}^{2n}$ of all $(k^{(1)},k^{(2)})$ into
\begin{equation}
\begin{split}
\Omega_0
&=\left\{\big(k^{(1)},k^{(2)}\big)\in\mathbb{Z}^{2n}:\langle
k^{(1)}\rangle\sim\langle k^{(2)}\rangle\right\}, \\
\Omega_1 &=\left\{\big(k^{(1)},k^{(2)}\big)\in\mathbb{Z}^{2n}:\langle k^{(1)}\rangle\gg\langle k^{(2)}\rangle\right\}, \\
\Omega_2
&=\left\{\big(k^{(1)},k^{(2)}\big)\in\mathbb{Z}^{2n}:\langle
k^{(1)}\rangle\ll\langle k^{(2)}\rangle\right\},
\end{split}
\end{equation}
and separate  $\Omega_0$   into
\begin{equation}
\begin{split}
\Omega_{0,1}
&=\left\{\big(k^{(1)},k^{(2)}\big)\in\Omega_0:K\big(k^{(1)},k^{(2)}\big)\lesssim\langle
k^{(1)}\rangle^{\frac{\alpha}{1-\alpha}}\right\}; \\
\Omega_{0,2}
&=\left\{\big(k^{(1)},k^{(2)}\big)\in\Omega_0:K\big(k^{(1)},k^{(2)}\big)\gg\langle
k^{(1)}\rangle^{\frac{\alpha}{1-\alpha}}\right\}.
\end{split}
\end{equation}
If $(k^{(1)},k^{(2)})\in\Omega_{0,1}$, from \eqref{rela} it is easy to see that
\begin{equation} \label{3.8}
\langle k\rangle \lesssim  \langle k^{(1)}\rangle^{\alpha}.
\end{equation}
Let
$(k^{(1)},k^{(2)})\in\Omega_{0,2}$ be fixed.   There exists some $y:=y(k^{(1)},k^{(2)}) \in(\alpha,1]$
such that
\begin{equation}\label{rela-K-k^1}
K_i(k^{(1)},k^{(2)})=K(k^{(1)},k^{(2)})\sim\langle
k^{(1)}\rangle^{\frac{y}{1-\alpha}},
\end{equation}
for some $i $ with $ 1\leqslant i\leqslant n$. We can assume that $K_i(k^{(1)},k^{(2)})>0$.  By
\eqref{rela} and \eqref{rela-K-k^1} we have
\begin{subequations}\label{jjjjjjjjjjjjjjjjj}
\begin{align}
\langle k\rangle^{\frac{\alpha}{1-\alpha}}(k_i-C)
<K_i(k^{(1)},k^{(2)})+C(\langle
k^{(1)}\rangle^{\frac{\alpha}{1-\alpha}}+\langle
k^{(2)}\rangle^{\frac{\alpha}{1-\alpha}}) &\lesssim\langle
k^{(1)}\rangle^{\frac{y}{1-\alpha}}, \label{jjjjjjjjjjjjjjjjj-1}\\
\langle k\rangle^{\frac{\alpha}{1-\alpha}}(k_i+C)
>K_i(k^{(1)},k^{(2)})-C(\langle
k^{(1)}\rangle^{\frac{\alpha}{1-\alpha}}+\langle
k^{(2)}\rangle^{\frac{\alpha}{1-\alpha}}) &\gtrsim\langle
k^{(1)}\rangle^{\frac{y}{1-\alpha}}. \label{jjjjjjjjjjjjjjjjj-2}
\end{align}
\end{subequations}
For every $k\in\Lambda(k^{(1)},k^{(2)})\cap\mathbb{R}^n_{\widetilde{i}}$,
we substitute $\widetilde{i}$ for $j$ in \eqref{rela}, thus
\eqref{rela-1} and \eqref{rela-2} are rewritten as
\begin{subequations}\label{kkkkkkkkkkkkkkkkkkk}
\begin{align}
\langle k\rangle^{\frac{\alpha}{1-\alpha}}(k_{\widetilde{i}}-C)
&<K_{\widetilde{i}}(k^{(1)},k^{(2)})+C(\langle
k^{(1)}\rangle^{\frac{\alpha}{1-\alpha}}+\langle
k^{(2)}\rangle^{\frac{\alpha}{1-\alpha}}); \label{kkkkkkkkkkkkkkkkkkk-1}\\
\langle k\rangle^{\frac{\alpha}{1-\alpha}}(k_{\widetilde{i}}+C)
&>K_{\widetilde{i}}(k^{(1)},k^{(2)})-C(\langle
k^{(1)}\rangle^{\frac{\alpha}{1-\alpha}}+\langle
k^{(2)}\rangle^{\frac{\alpha}{1-\alpha}}).
\label{kkkkkkkkkkkkkkkkkkk-2}
\end{align}
\end{subequations}
For such $k$, we claim that
\begin{equation}\label{algebra-ggggggggggg}
\langle k\rangle\sim\langle k^{(1)}\rangle^y.
\end{equation}
\eqref{algebra-ggggggggggg} is obvious when
$|k^{(1)}|\lesssim1$ and so, it suffices to consider \eqref{algebra-ggggggggggg} in the case
$|k^{(1)}|\gg1$. If either $|k_i|\sim\langle k\rangle$ or
$|K_{\widetilde{i}}(k^{(1)},k^{(2)})|\sim K(k^{(1)},k^{(2)})$
exists, \eqref{algebra-ggggggggggg} follows from
\eqref{jjjjjjjjjjjjjjjjj} or \eqref{kkkkkkkkkkkkkkkkkkk} directly.
Otherwise, we see that $|k_i|\ll\langle k\rangle$ and
$|K_{\widetilde{i}}(k^{(1)},k^{(2)})|\ll K(k^{(1)},k^{(2)})$.
When $k_{\widetilde{i}}>0$, we let
\eqref{kkkkkkkkkkkkkkkkkkk-1}+\eqref{jjjjjjjjjjjjjjjjj-1} and
\eqref{kkkkkkkkkkkkkkkkkkk-2}+\eqref{jjjjjjjjjjjjjjjjj-2}; whereas
when $k_{\widetilde{i}}<0$, we let
\eqref{kkkkkkkkkkkkkkkkkkk-2}$\times(-1)+$\eqref{jjjjjjjjjjjjjjjjj-1}
and
\eqref{jjjjjjjjjjjjjjjjj-1}$\times(-1)+$\eqref{jjjjjjjjjjjjjjjjj-2},
then we get
\begin{equation*}
\langle k\rangle^{\frac{1}{1-\alpha}}\gtrsim\langle
k^{(1)}\rangle^{\frac{y}{1-\alpha}}, \quad \langle
k\rangle^{\frac{1}{1-\alpha}}\lesssim\langle
k^{(1)}\rangle^{\frac{y}{1-\alpha}},
\end{equation*}
which imply \eqref{algebra-ggggggggggg}. Let $\widetilde{k} = (k_1,...,k_{j-1}, \tilde{k}_j,k_{j+1},...,k_n)\in \Lambda (k^{(1)},k^{(2)})$.  In view of (\ref{rela-1}) and (\ref{rela-2}) and $(k^{(1)},k^{(2)})\in\Omega_{0,2}$, we have
\begin{equation*}
\left|\langle
k\rangle^{\frac{\alpha}{1-\alpha}}k_j-\langle\widetilde{k}\rangle^{\frac{\alpha}{1-\alpha}}\widetilde{k}_j\right|\lesssim\langle
k^{(1)}\rangle^{\frac{\alpha}{1-\alpha}}+\langle
k\rangle^{\frac{\alpha}{1-\alpha}}+\langle\widetilde{k}\rangle^{\frac{\alpha}{1-\alpha}}.
\end{equation*}
Thus Taylor's theorem, combined with \eqref{algebra-ggggggggggg},
gives
\begin{equation}\label{algebra-set-length}
\left|k_j-\widetilde{k}_j\right|\lesssim\langle
k^{(1)}\rangle^{\frac{\alpha(1-y)}{1-\alpha}}.
\end{equation}
For $\big(k^{(1)},k^{(2)}\big)\in\Omega_1\cup\Omega_2$, in view of \eqref{maxmed} we have
\begin{equation}\label{algebra-relation-1-2-fff}
\langle k\rangle\sim\langle k^{(1)}\rangle\vee\langle
k^{(2)}\rangle
\end{equation}
and
\begin{equation}\label{algebra-relation-1-2-ggg}
\#\Lambda\big(k^{(1)},k^{(2)}\big)\sim1.
\end{equation}

In what follows, we separate the proof into four steps. In Steps 1-3, we prove
\eqref{algebra-inequality} for certain $p$ and
$q$. In Step 4, applying the complex interpolation together with the
conclusions obtained in the previous three steps,  we can get
\eqref{algebra-inequality}.

\noindent{\bf Step 1.} $\boldsymbol{1\leqslant p\leqslant\infty,q \leqslant 1.}$
Suppose $f,g\in M_{p,q}^{s,\alpha}$, from the triangle inequality
and the embedding $\ell^q\subset\ell^1$, we have
\begin{equation}\label{algebra-I-II-hat}
\begin{split}
\|fg\|_{M_{p,q}^{s,\alpha}} &=\left(\sum_{k\in\mathbb{Z}^n}\langle
k\rangle^{\frac{sq}{1-\alpha}}\|\square_k^{\alpha}(fg)\|_p^q\right)^{\frac1q}
\\
&\leqslant\left[\sum_k\langle
k\rangle^{\frac{sq}{1-\alpha}}\left(\sum_{k^{(1)},k^{(2)}}\|\square_k^{\alpha}(\square_{k^{(1)}}^{\alpha}f\square_{k^{(2)}}^{\alpha}g)\|_p\right)^q\right]^{\frac1q}
\\
&\leqslant\left(\sum_k\langle
k\rangle^{\frac{sq}{1-\alpha}}\sum_{k^{(1)},k^{(2)}}\|\square_k^{\alpha}(\square_{k^{(1)}}^{\alpha}f\square_{k^{(2)}}^{\alpha}g)\|_p^q\right)^{\frac1q}
\\
&=\left(\sum_{\ell=0}^2\sum_{(k^{(1)},k^{(2)})\in\Omega_{\ell}}\sum_{k\in\Lambda(k^{(1)},k^{(2)})}\langle
k\rangle^{\frac{sq}{1-\alpha}}\|\square_k^{\alpha}(\square_{k^{(1)}}^{\alpha}f\square_{k^{(2)}}^{\alpha}g)\|_p^q\right)^{\frac1q}.
\end{split}
\end{equation}
Applying the multiplier estimate and H\"older's inequality, we see that
\begin{align}\label{sum-p-1-1}
\|\square_k^{\alpha}(\square_{k^{(1)}}^{\alpha}f\square_{k^{(2)}}^{\alpha}g)\|_p \lesssim  \| \square_{k^{(1)}}^{\alpha}f\square_{k^{(2)}}^{\alpha}g \|_p  \lesssim \| \square_{k^{(1)}}^{\alpha}f \|_{p} \|\square_{k^{(2)}}^{\alpha}g \|_\infty.
\end{align}
For $(k^{(1)},k^{(2)})\in\Omega_{0,1}$, when $p=1$ and $s\geqslant
n\alpha+\frac{n\alpha(1-\alpha)}{2-\alpha}\big(\frac1q-1\big)$, we have
\begin{equation}\label{algebra-I-II-hat-1}
\begin{split}
&\sum_{(k^{(1)},k^{(2)})\in\Omega_{0,1}}\sum_{k\in\Lambda(k^{(1)},k^{(2)})}\langle
k\rangle^{\frac{sq}{1-\alpha}}\|\square_k^{\alpha}(\square_{k^{(1)}}^{\alpha}f\square_{k^{(2)}}^{\alpha}g)\|_1^q
\\
&\qquad\lesssim\sum_{(k^{(1)},k^{(2)})}\langle
k^{(1)}\rangle^{\frac{sq\alpha}{1-\alpha}+n\alpha}\|\square_{k^{(1)}}^{\alpha}f\|_1^q\langle
k^{(2)}\rangle^{\frac{nq\alpha}{1-\alpha}}\|\square_{k^{(2)}}^{\alpha}g\|_1^q
\leqslant\|f\|_{M_{1,q}^{s,\alpha}}^q\|g\|_{M_{1,q}^{s,\alpha}}^q.
\end{split}
\end{equation}
If $p=2$ and
$s\geqslant\frac{n\alpha}{2}+\frac{n\alpha(1-\alpha)}{2-\alpha}\big(\frac1q-1\big)$, by Plancherel and Jensen's inequality\footnote{$a^\theta_1+...+ a^\theta_N \le N^{1-\theta} (a_1+...+a_N)^\theta$ for $\theta \in (0,1)$},
 we have
\begin{equation}\label{algebra-I-II-hat-2}
\begin{split}
&\sum_{(k^{(1)},k^{(2)})\in\Omega_{0,1}}\sum_{k\in\Lambda(k^{(1)},k^{(2)})}\langle
k\rangle^{\frac{sq}{1-\alpha}}\|\square_k^{\alpha}(\square_{k^{(1)}}^{\alpha}f\square_{k^{(2)}}^{}g)\|_2^q
\\
&\qquad\lesssim\sum_{(k^{(1)},k^{(2)})}\sup_{k\in\Lambda(k^{(1)},k^{(2)})}\langle
k\rangle^{\frac{sq}{1-\alpha}}\sum_{k\in\Lambda(k^{(1)},k^{(2)})}\|\square_k^{\alpha}(\square_{k^{(1)}}^{\alpha}f\square_{(2)}^{\alpha}g)\|_2^q
\\
&\qquad\lesssim\sum_{(k^{(1)},k^{(2)})}\langle
k^{(1)}\rangle^{\frac{sq\alpha}{1-\alpha}} [\#\Lambda(k^{(1)},k^{(2)})]^{1-\frac{q}{2}}\|\square_{k^{(1)}}^{\alpha}f\square_{k^{(2)}}^{\alpha}g\|_2^q
\\
&\qquad\lesssim\sum_{(k^{(1)},k^{(2)})}\langle
k^{(1)}\rangle^{\frac{sq\alpha}{1-\alpha}+n\alpha\big(1-\frac{q}{2}\big)}\|\square_{k^{(1)}}^{\alpha}f\|_2^q\langle
k^{(2)}\rangle^{\frac{nq\alpha}{2(1-\alpha)}}\|\square_{k^{(2)}}^{\alpha}g\|_2^q\leqslant\|f\|_{M_{2,q}^{s,\alpha}}^q\|g\|_{M_{2,q}^{s,\alpha}}^q.
\end{split}
\end{equation}
For $p=\infty$ and
$s\geqslant\frac{n\alpha(1-\alpha)}{2-\alpha}\frac1q$, we have
\begin{equation}\label{algebra-I-II-hat-infty}
\begin{split}
&\sum_{(k^{(1)},k^{(2)})\in\Omega_{0,1}}\sum_{k\in\Lambda(k^{(1)},k^{(2)})}\langle
k\rangle^{\frac{sq}{1-\alpha}}\|\square_k^{\alpha}(\square_{k^{(1)}}^{\alpha}f\square_{k^{(2)}}^{\alpha}g)\|_{\infty}^q
\\
&\qquad\lesssim\sum_{(k^{(1)},k^{(2)})}\langle
k^{(1)}\rangle^{\frac{sq\alpha}{1-\alpha}+n\alpha}\|\square_{k^{(1)}}^{\alpha}f\|_{\infty}^q\|\square_{k^{(2)}}^{\alpha}g\|_{\infty}^q\leqslant\|f\|_{M_{\infty,q}^{s,\alpha}}^q\|g\|_{M_{\infty,q}^{s,\alpha}}^q.
\end{split}
\end{equation}

For $(k^{(1)},k^{(2)})\in\Omega_{0,2}$, when $s\geqslant
n\alpha (1+ (1-\alpha)/q)/(2-\alpha)$, we see that
$$
2s \geqslant sy + n\alpha \frac{(1+q-y)}{q}.
$$
It follows that
\begin{equation}\label{algebra-mid-1}
\begin{split}
&\sum_{(k^{(1)},k^{(2)})\in\Omega_{0,2}}\sum_{k\in\Lambda(k^{(1)},k^{(2)})}\langle
k\rangle^{\frac{s q}{1-\alpha}}\|\square_k^{\alpha}(\square_{k^{(1)}}^{\alpha}f\square_{k^{(2)}}^{\alpha}g)\|^q_1
\\
&\qquad \lesssim \sum_{k^{(1)},k^{(2)}}\langle
k^{(1)}\rangle^{\frac{syq}{1-\alpha}+\frac{n\alpha}{1-\alpha}(1-y)}\|\square_{k^{(1)}}^{\alpha}f\|^q_1
\langle
k^{(2)}\rangle^{\frac{n\alpha q}{1-\alpha}}\|\square_{k^{(2)}}^{\alpha}g\|^q_1
\lesssim\|f\|_{M_{1,q}^{s,\alpha}}\|g\|_{M_{1,q}^{s,\alpha}}.
\end{split}
\end{equation}
When $p=\infty$ and $s\geqslant\frac{n\alpha(1-\alpha)}{q(2-\alpha)}$,
it is suffices to get
\begin{equation}\label{algebra-mid-infty}
\begin{split}
&\sum_{(k^{(1)},k^{(2)})\in\Omega_{0,2}}\sum_{k\in\Lambda(k^{(1)},k^{(2)})}\langle
k\rangle^{\frac{sq}{1-\alpha}}\|\square_k^{\alpha}(\square_{k^{(1)}}^{\alpha}f\square_{k^{(2)}}^{\alpha}g)\|^q_{\infty}
\\
&\qquad\lesssim\sum_{k^{(1)},k^{(2)}}\langle
k^{(1)}\rangle^{\frac{sy q}{1-\alpha} +\frac{n\alpha}{1-\alpha}(1-y)}\|\square_{k^{(1)}}^{\alpha}f\|_{\infty}\|\square_{k^{(2)}}^{\alpha}g\|_{\infty}
\lesssim\|f\|_{M_{\infty,q}^{s,\alpha}}\|g\|_{M_{\infty,q}^{s,\alpha}};
\end{split}
\end{equation}
If $p=2$ and $s\geqslant\frac{n\alpha (\alpha/2+ (1-\alpha)/q)}{2-\alpha}$, by Plancherel and Jensen's inequality,
\begin{equation}\label{algebra-mid-2}
\begin{split}
&\sum_{(k^{(1)},k^{(2)})\in\Omega_{0,2}}\sum_{k\in\Lambda(k^{(1)},k^{(2)})}\langle
k\rangle^{\frac{s q}{1-\alpha}}\|\square_k^{\alpha}(\square_{k^{(1)}}^{\alpha}f\square_{k^{(2)}}^{\alpha}g)\|^q_2
\\
&\qquad \lesssim\sum_{k^{(1)},k^{(2)}}\langle
k^{(1)}\rangle^{\frac{sy q}{1-\alpha}} [\# \Lambda(k^{(1)},k^{(2)})]^{1-q/2} \|\square_{k^{(1)}}^{\alpha}f\square_{k^{(2)}}^{\alpha}g\|^q_2
\\
&\qquad \lesssim \sum_{k^{(1)},k^{(2)}}\langle
k^{(1)}\rangle^{\frac{sy q}{1-\alpha}+\frac{n\alpha}{(1-\alpha)}(1-y)(1-q/2)}\|\square_{k^{(1)}}^{\alpha}f\|^q_2 \langle
k^{(2)}\rangle^{\frac{n\alpha q}{2(1-\alpha)}}\|\square_{k^{(2)}}^{\alpha}g\|^q_2\\
&\qquad \lesssim\|f\|_{M_{2,q}^{s,\alpha}}\|g\|_{M_{2,q}^{s,\alpha}}.
\end{split}
\end{equation}
From
\eqref{algebra-relation-1-2-fff}, \eqref{algebra-relation-1-2-ggg},
\eqref{embede-1}, if $(k^{(1)},k^{(2)})\in\Omega_1$, when
$s\geqslant\frac{n\alpha}{p}$,
\begin{equation}\label{algebra-high-1}
\begin{split}
&\sum_{(k^{(1)},k^{(2)})\in\Omega_1} \ \sum_{k\in\Lambda(k^{(1)},k^{(2)})}\langle
k\rangle^{\frac{s q}{1-\alpha}}\|\square_k^{\alpha}(\square_{k^{(1)}}^{\alpha}f\square_{k^{(2)}}^{\alpha}g)\|^q_p
\\
&\qquad\lesssim \sum_{k^{(1)}\in\mathbb{Z}^n}\langle
k^{(1)}\rangle^{\frac{s q}{1-\alpha}}\|\square_{k^{(1)}}^{\alpha}f\|^q_p \sum_{k^{(2)}\in\mathbb{Z}^n} \|\square_{k^{(2)}}^{\alpha}g\|^q_{\infty}
\lesssim\|f\|_{M_{p,q}^{s,\alpha}}\|g\|_{M_{\infty,q}^{0,\alpha}}
\\
&\qquad\lesssim\|f\|_{M_{p,q}^{s,\alpha}}\|g\|_{M_{p,q}^{s,\alpha}}.
\end{split}
\end{equation}
Similarly,  if $(k^{(1)},k^{(2)})\in\Omega_2$ and
$s\geqslant\frac{n\alpha}{p}$,
 \begin{equation}\label{algebra-I-II-hat-23}
\begin{split}
&\sum_{(k^{(1)},k^{(2)})\in \Omega_2} \ \sum_{k\in\Lambda(k^{(1)},k^{(2)})}\langle
k\rangle^{\frac{sq}{1-\alpha}}\|\square_k^{\alpha}(\square_{k^{(1)}}^{\alpha}f\square_{k^{(2)}}^{\alpha}g)\|_p^q
  \lesssim \|f\|_{M_{p,q}^{s,\alpha}}\|g\|_{M_{p,q}^{s,\alpha}}.
\end{split}
\end{equation}
By complex interpolation,
\eqref{algebra-I-II-hat}-\eqref{algebra-I-II-hat-23} imply that $M_{p,q}^{s,\alpha}$ is
a multiplication algebra
as long as $s\geqslant s_0$ for some $s_0$. More precisely, when $1\leqslant
p\leqslant2$, from \eqref{algebra-I-II-hat-1} and
\eqref{algebra-I-II-hat-2}, we get
$s_0=\frac{n\alpha}{p}+\frac{n\alpha(1-\alpha)}{2-\alpha}\big(\frac1q-1\big)$;
and when $2<p\leqslant\infty$, from \eqref{algebra-I-II-hat-2} and
\eqref{algebra-I-II-hat-infty}, we get $s_0=\frac{n\alpha}{p}
+\frac{n\alpha(1-\alpha)}{2-\alpha}\big(\frac1q-\frac{2}{p}\big)$.

\noindent{\bf Step 2.} $\boldsymbol{0 < p\leqslant\infty,q=\infty}.$
First, we consider the case $1\leqslant p \leqslant \infty$. Suppose $f,g\in M_{p,\infty}^{s,\alpha}$, from the triangle
inequality, we have
\begin{equation}\label{algebra-p>1-q=infty}
\begin{split}
\|fg\|_{M_{p,\infty}^{s,\alpha}} &=\sup_{k\in\mathbb{Z}^n}\langle
k\rangle^{\frac{s}{1-\alpha}}\|\square_k^{\alpha}(fg)\|_p \\
&\leqslant\sup_{k\in\mathbb{Z}^n}\langle
k\rangle^{\frac{s}{1-\alpha}}\sum_{(k^{(1)},k^{(2)})\in\Lambda(k)}\|\square_k^{\alpha}(\square_{k^{(1)}}^{\alpha}f\square_{k^{(2)}}^{\alpha}g)\|_p
\\
&=\sup_{k\in\mathbb{Z}^n}\sum_{\ell=0}^2\sum_{(k^{(1)},k^{(2)})\in\Lambda(k)\cap\Omega_{\ell}}\langle
k\rangle^{\frac{s}{1-\alpha}}\|\square_k^{\alpha}(\square_{k^{(1)}}^{\alpha}f\square_{k^{(2)}}^{\alpha}g)\|_p.
\end{split}
\end{equation}
For a $\Psi\subset\mathbb{Z}^{2n}$, we denote
\begin{equation*}
\begin{split}
\Psi^{\bot}_1 =\left\{k^{(1)}\in\mathbb{Z}^n: \exists
k^{(2)}\in\mathbb{Z}^n\;\mbox{s.t.}\;\big(k^{(1)},k^{(2)}\big)\in\Psi\right\};
\\
\Psi^{\bot}_2 =\left\{k^{(2)}\in\mathbb{Z}^n: \exists
k^{(1)}\in\mathbb{Z}^n\;\mbox{s.t.}\;\big(k^{(1)},k^{(2)}\big)\in\Psi\right\}.
\end{split}
\end{equation*}
Let $s>\frac{n\alpha}{p}+n(1-\alpha)$. For any
$k^{(2)}\in\left\{\{\Omega_0\cup\Omega_1\}\cap\Lambda(k)\right\}_2^{\bot}$ with
every fixed $k$, noticing \eqref{maxmed}, we easily see $\#\Lambda(-k^{(2)},k)\lesssim1$.
Inserting \eqref{sum-p-1-1} and using
\eqref{embede-2}, we obtain
\begin{equation}\label{algebra-p>1-q=infty-1}
\begin{split}
&\sum_{(k^{(1)},k^{(2)})\in\{\Omega_0\cup\Omega_1\}\cap\Lambda(k)}\langle
k\rangle^{\frac{s}{1-\alpha}}\|\square_k^{\alpha}(\square_{k^{(1)}}^{\alpha}f\square_{k^{(2)}}^{\alpha}g)\|_p
\\
&\qquad\lesssim\sup_{k^{(1)}\in\big\{\{\Omega_0\cup\Omega_1\}\cap\Lambda(k)\big\}^{\bot}_1}\langle
k\rangle^{\frac{s}{1-\alpha}}\|\square_{k^{(1)}}^{\alpha}f\|_p\sum_{k^{(1)}\in\big\{\{\Omega_0\cup\Omega_1\}\cap\Lambda(k)\big\}^{\bot}_1}\sum_{k^{(2)}\in\Lambda(-k^{(1)},k)}\|\square_{k^{(2)}}^{\alpha}g\|_{\infty}
\\
&\qquad\lesssim\sup_{k^{(1)}\in\mathbb{Z}^n}\langle
k^{(1)}\rangle^{\frac{s}{1-\alpha}}\|\square_{k^{(1)}}^{\alpha}f\|_p\sum_{k^{(2)}\in\big\{\{\Omega_0\cup\Omega_1\}\cap\Lambda(k)\big\}^{\bot}_2}\sum_{k^{(1)}\in\Lambda(-k^{(2)},k)}\|\square_{k^{(2)}}^{\alpha}g\|_{\infty}
\\
&\qquad\lesssim\|f\|_{M_{p,\infty}^{s,\alpha}}\|g\|_{M_{\infty,1}^{0,\alpha}}\lesssim\|f\|_{M_{p,\infty}^{s,\alpha}}\|g\|_{M_{p,\infty}^{s,\alpha}}.
\end{split}
\end{equation}
For $k^{(1)}\in\left\{\Omega_2\cap\Lambda(k)\right\}_1^{\bot}$ with every fixed
$k$, symmetrically, we have
\begin{equation}\label{algebra-p>1-q=infty-2}
\begin{split}
&\sum_{(k^{(1)},k^{(2)})\in\Omega_2\cap\Lambda(k)}\langle
k\rangle^{\frac{s}{1-\alpha}}\|\square_k^{\alpha}(\square_{k^{(1)}}^{\alpha}f\square_{k^{(2)}}^{\alpha}g)\|_p
\\
&\qquad\lesssim\sup_{k^{(2)}\in \mathbb{Z}^n} \langle
k^{(2)}\rangle^{\frac{s}{1-\alpha}}\|\square_{k^{(2)}}^{\alpha}g\|_p\sum_{k^{(1)}\in \left\{\Omega_2\cap\Lambda(k)\right\}_1^{\bot}} \ \sum_{k^{(2)}\in \Lambda(-k^{(1)},k)}\|\square_{k^{(1)}}^{\alpha}f\|_{\infty}
\\
&\qquad\lesssim\|f\|_{M_{\infty,1}^{0,\alpha}}\|g\|_{M_{p,\infty}^{s,\alpha}}\lesssim\|f\|_{M_{p,\infty}^{s,\alpha}}\|g\|_{M_{p,\infty}^{s,\alpha}}.
\end{split}
\end{equation}
Combining \eqref{algebra-p>1-q=infty}--\eqref{algebra-p>1-q=infty-2},
we know when $s>\frac{n\alpha}{p}+n(1-\alpha)$,
$M_{p,\infty}^{s,\alpha}$ is a multiplication algebra.

Next, we consider the case $0<p<1$ and $q=\infty$.  Suppose that $f,g\in M_{p,\infty}^{s,\alpha}$.  It follows from the
embedding $\ell^p\subset\ell^1$ that
\begin{equation}\label{algebra-p<1-q=infty-0}
\begin{split}
\|fg\|_{M_{p,\infty}^{s,\alpha}} &\leqslant\sup_k\langle
k\rangle^{\frac{s}{1-\alpha}}\left(\sum_{(k^{(1)},k^{(2)})\in\Lambda(k)}\|\square_k^{\alpha}(\square_{k^{(1)}}^{\alpha}f\square_{k^{(2)}}^{\alpha}g)\|_p^p\right)^{\frac1p}
\\
&=\sup_k\left(\sum_{\ell=0}^2\sum_{(k^{(1)},k^{(2)})\in\Lambda(k)\cap\Omega_{\ell}}\langle
k\rangle^{\frac{sp}{1-\alpha}}\|\square_k^{\alpha}(\square_{k^{(1)}}^{\alpha}f\square_{k^{(2)}}^{\alpha}g)\|_p^p\right)^{\frac1p}.
\end{split}
\end{equation}
By Proposition \ref{convolution}, for $(k^{(1)}, k^{(2)}) \in \Lambda (k) \cap \Omega_0$, we have
\begin{align}
\|\square_k^{\alpha}(\square_{k^{(1)}}^{\alpha}f\square_{k^{(2)}}^{\alpha}g)\|_p & \lesssim  \langle k^{(1)}\rangle^{\frac{n\alpha}{(1-\alpha)}(\frac{1}{p}-1)} \|\mathscr{F}^{-1} \eta^\alpha_k \|_p \|\square_{k^{(1)}}^{\alpha}f\square_{k^{(2)}}^{\alpha}g\|_p  \nonumber\\
 & \lesssim  \langle k^{(1)}\rangle^{\frac{n\alpha}{(1-\alpha)}(\frac{1}{p}-1)} \langle k\rangle^{- \frac{n\alpha}{(1-\alpha)}(\frac{1}{p}-1)} \|\square_{k^{(1)}}^{\alpha}f\square_{k^{(2)}}^{\alpha}g\|_p. \label{algebra-p<1-omega-1}
\end{align}
When
$s \ge \frac{n }{p} + \frac{n\alpha(1-\alpha)}{(2-\alpha)}(\frac{1}{p}-1)$,
inserting \eqref{algebra-p<1-omega-1} and using \eqref{embede-2}, we
obtain that
\begin{equation}\label{algebra-p<1-q=infty-1}
\begin{split}
&\sum_{(k^{(1)},k^{(2)})\in\Lambda(k)\cap\Omega_0}\langle
k\rangle^{\frac{sp}{1-\alpha}}\|\square_k^{\alpha}(\square_{k^{(1)}}^{\alpha}f\square_{k^{(2)}}^{\alpha}g)\|_p^p
\\
&\qquad\lesssim \sum_{(k^{(1)},k^{(2)}) \in\Lambda(k)\cap\Omega_0 }
\langle k^{(1)}\rangle^{\frac{spy + n\alpha(1-y)(1-p)}{1-\alpha}} \|\square_{k^{(1)}}^{\alpha}f\|_p^p\|\square_{k^{(2)}}^{\alpha}g\|_{\infty}^p
\\
&\qquad\lesssim \sup_{k^{(1)}\in\mathbb{Z}^n}\langle
k^{(1)}\rangle^{\frac{sp}{1-\alpha}}\|\square_{k^{(1)}}^{\alpha}f\|_p^p
\sum_{k^{(2)} \in \left\{\Lambda(k)\cap\Omega_0\right\}_2^{\bot}} \ \ \sum_{k^{(1)}\in\Lambda(-k^{(2)}, k)}  \langle k^{(2)}\rangle^{\frac{sp(y-1) + n\alpha(1-y)(1-p)}{1-\alpha}} \|\square_{k^{(2)}}^{\alpha}g\|_{\infty}^p
\\
&\qquad \lesssim\|f\|_{M_{p,\infty}^{s,\alpha}}^p\|g\|_{M_{p,\infty}^{s,\alpha}}^p.
\end{split}
\end{equation}
For any $(k^{(1)},k^{(2)})\in\Lambda(k)\cap\{\Omega_1\cup\Omega_2\}$
with every fixed $k$,  imitating the process  as in
\eqref{algebra-p<1-omega-1} and combining
\eqref{algebra-relation-1-2-fff}, we get
\begin{equation}\label{algebra-p<1-omega-2}
\|\square_k^{\alpha}(\square_{k^{(1)}}^{\alpha}f\square_{k^{(2)}}^{\alpha}g)\|_p\lesssim\|\square_{k^{(1)}}^{\alpha}f\square_{k^{(2)}}^{\alpha}g\|_p.
\end{equation}
When $s>\frac{n}{p} $, inserting
\eqref{algebra-p<1-omega-2} and also using \eqref{embede-2}, we
obtain
\begin{equation}\label{algebra-p<1-q=infty-2}
\begin{split}
&\sum_{\ell=1}^2\sum_{(k^{(1)},k^{(2)})\in\Lambda(k)\cap\Omega_{\ell}}\langle
k\rangle^{\frac{sp}{1-\alpha}}\|\square_k^{\alpha}(\square_{k^{(1)}}^{\alpha}f\square_{k^{(2)}}^{\alpha}g)\|_p^p
\\
&\qquad\lesssim\|f\|_{M_{p,\infty}^{s,\alpha}}^p\|g\|_{M_{\infty,p}^{0,\alpha}}^p+\|f\|_{M_{\infty,p}^{0,\alpha}}^p\|g\|_{M_{p,\infty}^{s,\alpha}}^p\lesssim\|f\|_{M_{p,\infty}^{s,\alpha}}^p\|g\|_{M_{p,\infty}^{s,\alpha}}^p.
\end{split}
\end{equation}

Combining
\eqref{algebra-p<1-q=infty-0},\eqref{algebra-p<1-q=infty-1} and
\eqref{algebra-p<1-q=infty-2}, we conclude when
$s \ge \frac{n }{p} + \frac{n\alpha(1-\alpha)}{(2-\alpha)}((1\vee \frac{1}{p})-1)$,
$M_{p,\infty}^{s,\alpha}$ is a multiplication algebra.

\noindent{\bf Step 3.} $\,\boldsymbol{p<1, \ q= p.}$  Suppose $f,g\in M_{p,p}^{s,\alpha}$. From the embedding
$\ell^p\subset\ell^1\subset \ell^{1/p}$ it follows that
\begin{equation}\label{algebra-iii-hat-0}
\begin{split}
\|fg\|_{M_{p,p}^{s,\alpha}} &\leqslant\left(\sum_k\langle
k\rangle^{\frac{sp}{1-\alpha}}\sum_{k^{(1)},k^{(2)}}\|\square_k^{\alpha}(\square_{k^{(1)}}^{\alpha}f\square_{k^{(2)}}^{\alpha}g)\|_p^p\right)^{\frac1p}
\\
&=\left(\sum_{\ell=0}^2\sum_{(k^{(1)},k^{(2)})\in\Omega_{\ell}}\sum_{k\in\Lambda(k^{(1)},k^{(2)})}\langle
k\rangle^{\frac{sp}{1-\alpha}}\|\square_k^{\alpha}(\square_{k^{(1)}}^{\alpha}f\square_{k^{(2)}}^{\alpha}g)\|_p^p\right)^{\frac1p}.
\end{split}
\end{equation}
For $(k^{(1)},k^{(2)})\in\Omega_0$,  if
$s\geqslant\frac{n\alpha}{p} + \frac{n\alpha(1-\alpha)}{(2-\alpha)}(\frac{1}{p}-1)$,
then we see from \eqref{3.8}, \eqref{algebra-ggggggggggg}, \eqref{algebra-set-length} and \eqref{algebra-p<1-omega-1} that
\begin{equation}\label{algebra-iii-hat-1}
\begin{split}
&\sum_{(k^{(1)},k^{(2)})\in\Omega_0} \ \sum_{k\in\Lambda(k^{(1)},k^{(2)})}\langle
k\rangle^{\frac{sp}{1-\alpha}}\|\square_k^{\alpha}(\square_{k^{(1)}}^{\alpha}f\square_{k^{(2)}}^{\alpha}g)\|_p^p
\\
&\qquad\lesssim \sum_{(k^{(1)},k^{(2)})\in\Omega_0}\langle
k^{(1)}\rangle^{\frac{spy}{1-\alpha} - \frac{n\alpha y}{1-\alpha}(1-p) +\frac{n\alpha}{1-\alpha}(1-y) +\frac{n\alpha}{1-\alpha}(1-p) }\|\square_{k^{(1)}}^{\alpha}f\|_p^p\|\square_{k^{(2)}}^{\alpha}g\|_{\infty}^p
\\
&\qquad\lesssim\sum_{(k^{(1)},k^{(2)})\in\Omega_0}\langle
k^{(1)}\rangle^{\frac{spy}{1-\alpha} +\frac{n\alpha}{1-\alpha}(1-p)(1-y) +\frac{n\alpha}{1-\alpha}(2-y)}\|\square_{k^{(1)}}^{\alpha}f\|_p^p\|\square_{k^{(2)}}^{\alpha}g\|_p^p\\
& \qquad \lesssim\|f\|_{M_{p,p}^{s,\alpha}}^p\|g\|_{M_{p,p}^{s,\alpha}}^p.
\end{split}
\end{equation}
For $(k^{(1)},k^{(2)})\in\Omega_1\cup\Omega_2$, when
$s\geqslant\frac{n\alpha}{p}$, in view  of \eqref{embede-1}, we
obtain
\begin{equation}\label{algebra-iii-hat-2}
\begin{split}
&\sum_{(k^{(1)},k^{(2)})\in\Omega_1 \cup \Omega}\sum_{k\in\Lambda(k^{(1)},k^{(2)})}\langle
k\rangle^{\frac{sp}{1-\alpha}}\|\square_k^{\alpha}(\square_{k^{(1)}}^{\alpha}f\square_{k^{(2)}}^{\alpha}g)\|_p^p
\\
&\qquad\lesssim\|f\|_{M_{p,p}^{s,\alpha}}^p\|g\|_{M_{\infty,p}^{0,\alpha}}^p+\|f\|_{M_{\infty,p}^{0,\alpha}}^p\|g\|_{M_{p,p}^{s,\alpha}}^p\lesssim\|f\|_{M_{p,p}^{s,\alpha}}^p\|g\|_{M_{p,p}^{s,\alpha}}^p.
\end{split}
\end{equation}
Combining \eqref{algebra-iii-hat-0}-\eqref{algebra-iii-hat-2}, we
conclude when
$s\geqslant\frac{n\alpha}{p}+\frac{n\alpha(1-\alpha)}{2-\alpha}\big(\frac1p-1\big)$,
$M_{p,p}^{s,\alpha}$ is a multiplication algebra.

\begin{figure}
\begin{center}
\includegraphics[scale=.6]{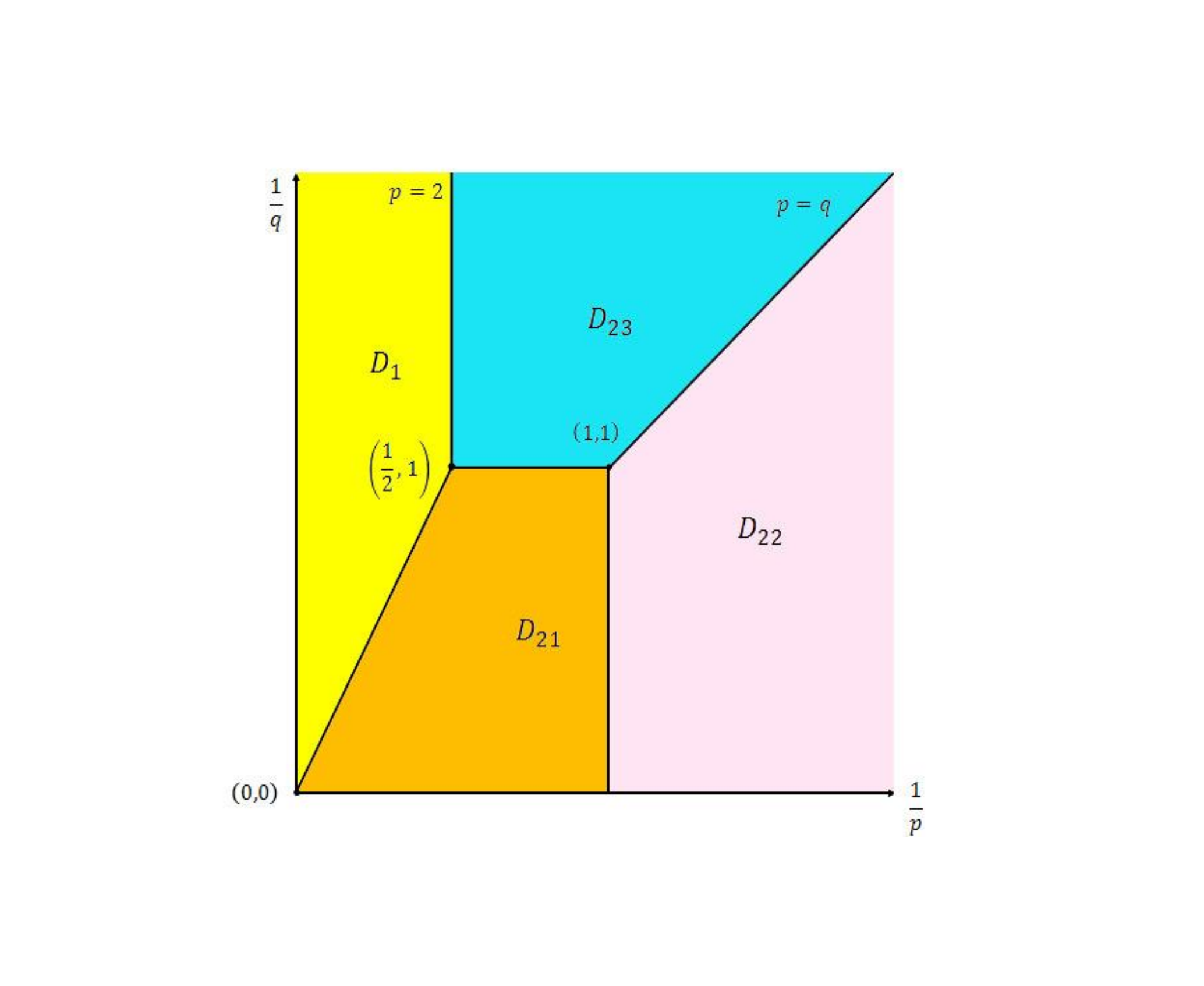}
\vspace*{-2cm}\caption{\small For Step 4 in the proof of Theorem \ref{algebra}. }\label{figure-step-5}
\end{center}
\end{figure}

\noindent{\bf Step 4.} Let
$\big(\frac1p,\frac1q\big)\in\big\{\big(\frac1p,\frac1q\big)\in\mbox{D}_1:\frac1q<1\big\}$. It is easy to see that $\big(\frac1p,\frac1q\big)$
is a point at the line segment connecting
$\big(\frac1p+\frac12\big(1-\frac1q\big),1\big)$ and
$\big(0,\frac1q-\frac2p\big)$, which is parallel to the line connecting
$\big(\frac12,1\big)$ and $(0,0)$.  At the point
$(\frac{1}{\bar{p}}, 1):= \big(\frac1p+\frac12\big(1-\frac1q\big),1\big)$, in Step 1 we have shown
that $M^{s,\alpha}_{\bar{p},1}$ is a multiplication algebra  if $s\geqslant
n\alpha\big[\frac1p+\frac12\big(1-\frac1q\big)\big]+\frac{n\alpha(1-\alpha)}{2-\alpha}\big(\frac1q-\frac2p\big)$. For $(0, \frac{1}{\bar{q}}):=
\big(0,\frac1q-\frac2p\big)$, complex interpolation between $(0,1)$
in Step 1 and $(0,0)$ in Step 3 shows that once
$s>n(1-\alpha)\big(1-\frac1q+\frac2p\big)+\frac{n\alpha(1-\alpha)}{2-\alpha}\big(\frac1q-\frac2p\big)$,
the associated $\alpha$-modulation space $M^{s,\alpha}_{\infty, \bar{q}}$ is a multiplication algebra. Again,
using the complex interpolation and combining
the result in Step 2, and we arrive at \eqref{algebra-inequality} in
$\mbox{D}_1$. Denote (see Fig. \ref{figure-step-5})
\begin{equation*}
\begin{split}
\mbox{D}_{21} &=[0,1]\times[0,1]\backslash\mbox{D}_1; \\
\mbox{D}_{23}
&=\left\{\big(\tfrac1p,\tfrac1q\big)\in\mathbb{R}_+^2:
\tfrac1q\geqslant\tfrac1p\right\}\Big\backslash\mbox{D}_1\cup\mbox{D}_{21}.
\end{split}
\end{equation*}
Through the point $\big(\frac1p,\frac1q\big)\in\mbox{D}_{21}$, one can make a line segment connecting
$\big(\frac{1}{2q},\frac1q\big)$ and $\big(1,\frac1q\big)$. For
$\big(\frac{1}{2q},\frac1q\big)$,  we see that once
$s>\frac{n\alpha}{2q}+n(1-\alpha)\big(1-\frac1q\big)$,
$M_{2q,q}^{s,\alpha}$ is a multiplication algebra. For
$\big(1,\frac1q\big)$, complex interpolation between $(1,1)$ in Step
1 and $(1,0)$ in Step 3 shows that once
$s>n\alpha+n(1-\alpha)\big(1-\frac1q\big)$, the associated
$\alpha$-modulation space is a multiplication algebra. Then we use complex
interpolation to get that
$M_{p,q}^{s,\alpha}$ is a multiplication algebra if
$s>\frac{n\alpha}{p}+n(1-\alpha)\big(1-\frac1q\big)$.  If
$(1/p,1/q) \in \mbox{D}_{23}$, the result can be derived in a similar way.

If
$\big(\frac1p,\frac1q\big)\in  \mbox{D}_{22}$,
then it belongs to the segment by connecting $\big(\frac1p,0\big)$ and $\big(\frac1p,\frac1p\big)$. In
Step 4 we see that once
$s\geqslant\frac{n\alpha}{p}+n(1-\alpha)\frac1p+\frac{n\alpha(1-\alpha)}{2-\alpha}\big(\frac1p-1\big)$,
$M_{p,\infty}^{s,\alpha}$ is a multiplication algebra; and once
$s\geqslant\frac{n\alpha}{p}+\frac{n\alpha(1-\alpha)}{2-\alpha}\big(\frac1p-1\big)$,
$M_{p,p}^{s,\alpha}$ is a multiplication algebra. Then complex
interpolation between them gives once
$s\geqslant\frac{n\alpha}{p}+n(1-\alpha)\big(\frac1p-\frac1q\big)+\frac{n\alpha(1-\alpha)}{2-\alpha}\big(\frac1p-1\big)$,
$M_{p,q}^{s,\alpha}$ is a multiplication algebra.
\end{proof}

~~

\section{Sharpness for the scaling and embedding properties}

In this section we show the necessity of Theorems \ref{dilationP}, \ref{embedding} and \ref{embedmb}. Since the $p$-$BAPU$ has no scaling, it is difficult to calculate the norm for a known function. However, we have the following equivalent norm on $\alpha$-modulation spaces. Let
$\rho$ be a smooth radial bump function supported in $B(0,2)$,
satisfying $\rho(\xi)=1$ as $|\xi|<1$, and $\rho(\xi)=0$ as
$|\xi|\geqslant2$. Let $\rho^\alpha_k $ be as in \eqref{rhok}:
$$
\rho^\alpha_k (\xi)= \rho\left(\frac{\xi- \langle k\rangle^{\alpha/(1-\alpha)}k}{C \langle k\rangle^{\alpha/(1-\alpha)}}\right).
$$
\begin{prop} Let $\rho^\alpha_k$ be as in the above. Then
$$
\|f\|^{\circ}_{M^{s,\alpha}_{p,q}} = \left(\sum_{k\in \mathbb{Z}^n} \langle k\rangle^{ \frac{sq}{1-\alpha}} \|(\mathscr{F}^{-1}\rho^\alpha_k)*f\|^q_p \right)^{1/q}
$$
is an equivalent norm on $M^{s,\alpha}_{p,q}$.
\end{prop}
\begin{proof}
If $p\ge 1$, in view of Young's inequality,
$$
\|(\mathscr{F}^{-1}\rho^\alpha_k)*f\|_p \le \sum_{|\ell|_\infty \le C}\|(\mathscr{F}^{-1}\rho^\alpha_{k+\ell})\|_1 \|\Box^\alpha_k f\|_p \lesssim  \|\Box^\alpha_k f\|_p.
$$
If $p<1$, by Proposition \ref{convolution} and the scaling argument, we have
$$
\|(\mathscr{F}^{-1}\rho^\alpha_k)*f\|_p \le  \langle k\rangle^{n\alpha(1/p-1)/(1-\alpha)} \sum_{|\ell|_\infty \le C}\|(\mathscr{F}^{-1}\rho^\alpha_{k+\ell})\|_p \|\Box^\alpha_k f\|_p \lesssim  \|\Box^\alpha_k f\|_p.
$$
Combining the above two cases, we have $\|f\|^{\circ}_{M^{s,\alpha}_{p,q}} \le \|f\|_{M^{s,\alpha}_{p,q}}$.  On the other hand, noticing that
$$
\eta^\alpha_k = \rho^\alpha_k \frac{\rho\left(\frac{\xi- \langle k\rangle^{\alpha/(1-\alpha)}k}{2 C \langle k\rangle^{\alpha/(1-\alpha)}}\right)}{ \sum_{ \ell
\in \mathbb{Z}^n}\rho\left(\frac{\xi- \langle k +\ell \rangle^{\alpha/(1-\alpha)}(k+\ell)}{C \langle k+\ell\rangle^{\alpha/(1-\alpha)}}\right)}: = \rho^\alpha_k \sigma^\alpha_k,
$$
We have for $p\ge 1$,  in view of Young's inequality,
$$
\|\Box^\alpha_k f\|_p \le \|(\mathscr{F}^{-1}\rho^\alpha_k)*f\|_p \|(\mathscr{F}^{-1}\sigma^\alpha_{k})\|_1 \lesssim  \|(\mathscr{F}^{-1}\rho^\alpha_k)*f\|_p .
$$
If $p<1$, by Proposition \ref{convolution}, the scaling argument and the generalized Bernstein inequality, we have
$$
\|\Box^\alpha_k f\|_p  \lesssim    \langle k\rangle^{n\alpha(1/p-1)/(1-\alpha)} \|(\mathscr{F}^{-1}\rho^\alpha_k)*f\|_p  \|\mathscr{F}^{-1}\sigma^\alpha_{k}\|_p \|\Box^\alpha_k f\|_p \lesssim  \|(\mathscr{F}^{-1}\rho^\alpha_k)*f\|_p.
$$
The result follows.
\end{proof}

\begin{proof} [Proof of Theorem \ref{dilationP}]
({\bf Necessity}) We divide the proof into the following two cases $\lambda\gg 1$ and $\lambda \ll 1$, respectively.

\noindent{\bf Case 1.} $\lambda \gg 1$. One needs to show that
\begin{equation*}
\|f_{\lambda}\|_{M_{p,q}^{s,\alpha}}\gtrsim\lambda^{-\frac{n}{p}+n\big(\frac1p-1\big)\vee (s+s_c)}\|f\|_{M_{p,q}^{s,\alpha}}.
\end{equation*}
\noindent{\bf Case 1.1.} We consider the case $s_p= n(1/p-1) > s+s_c$.  Our aim is to show that there exists a function $f$ satisfying
\begin{equation*}
\|f_{\lambda}\|_{M_{p,q}^{s,\alpha}}\gtrsim\lambda^{- n }\|f\|_{M_{p,q}^{s,\alpha}}.
\end{equation*}
Taking $f= \mathscr{F}^{-1}\rho$, we have
\begin{equation}
\begin{split}
\|f_{\lambda}\|^{\circ}_{M_{p,q}^{s,\alpha}} \ge
\|(\mathscr{F}^{-1} \rho^\alpha_0)* f_\lambda\|_p
 \geqslant \lambda^{-n}\|\mathscr{F}^{-1}\rho\|_p
 \gtrsim \lambda^{-n}\| f \|_{M_{p,q}^{s,\alpha}}.
\end{split}
\end{equation}

\noindent{\bf Case 1.2.} We consider the case $s_p < s+s_c$.  According to the definition of $s_c$, we separate the proof into the following three cases.

{\bf Case 1.2.1.}  $s_c= n(1-\alpha)(1/p+1/q-1)$. Put $f= \mathscr{F}^{-1} \rho$.  Since $\lambda \gg 1$, we see that for some  $0<\varepsilon_0<\varepsilon_1 \ll 1$,
$$
(\mathscr{F}^{-1}\rho^\alpha_k)*f_\lambda = \lambda^{-n} \mathscr{F}^{-1}\rho^\alpha_k,
\quad  \quad |k | \in [\varepsilon_0 \lambda^{1-\alpha}, \ \varepsilon_1 \lambda^{1-\alpha}].
$$
It follows that
\begin{equation}
\begin{split}
\|f_{\lambda}\|^{\circ}_{M_{p,q}^{s,\alpha}} &\gtrsim
\lambda^{-n}\left(\sum_{|l|\in [\varepsilon_0 \lambda^{1-\alpha}, \ \varepsilon_1 \lambda^{1-\alpha}]}    \langle l\rangle^{\frac{sq}{1-\alpha}}\|\mathscr{F}^{-1}\rho_l^{\alpha}\|_p^q\right)^{\frac1q} \\
&\gtrsim\lambda^{-\frac{n}{p}+s+n(1-\alpha)\big(\frac1p+\frac1q-1\big)}
\\
&\gtrsim\lambda^{-\frac{n}{p}+s+n(1-\alpha)\big(\frac1p+\frac1q-1\big)}\|f\|_{M_{p,q}^{s,\alpha}}.
\end{split}
\end{equation}

\noindent{\bf Case 1.2.2.} $s_c=0$. Let us take
$f={\rm e}^{{\rm i} x_1}\mathscr{F}^{-1} \rho_\lambda.$  We have $f_\lambda = \lambda^{-n} e^{{\rm i} x_1 \lambda} \mathscr{F}^{-1}\rho$.  We may assume that there exists $l_0 \in \mathbb{N}$ such that $\langle l_0\rangle^{\alpha/(1-\alpha)} l_0  \sim \lambda$ and $B((\lambda,0,...,0), \ 2) \subset B(\langle l_0\rangle^{\alpha/(1-\alpha)} (l_0,0,...,0), \ C\langle l_0\rangle^{\alpha/(1-\alpha)} )$.
It is easy to see that $\|f\|_p\sim\lambda^{n\big(\frac1p-1\big)}$  and
\begin{equation}
\|f_{\lambda}\|_{M_{p,q}^{s,\alpha}} \sim \lambda^{-n} \langle l_0\rangle^{s/(1-\alpha)}  \sim \lambda^{-n+s}
\gtrsim\lambda^{-\frac{n}{p}+s}\|f\|_p\gtrsim \lambda^{-\frac{n}{p}+s}\|f\|_{M_{p,q}^{s,\alpha}}.
\end{equation}

\noindent{\bf Case 1.2.3.} $s_c = n(1-\alpha)(1/q-1/p)$.  Put
\begin{equation}
f^{(l)}={\rm e}^{{\rm i}x \langle l \rangle^{\frac{\alpha}{1-\alpha}}l /\lambda} \  \tau_{\lambda^{1-\alpha} l}  \left( \rho \left(\frac{\lambda \ \cdot}{ c  \langle l \rangle^{\frac{\alpha}{1-\alpha}}} \right)  \right)^\vee, \quad f=\sum_{|\ell|\sim \lambda^{1-\alpha}} f^{(l)},
\end{equation}
where $|\ell| \sim \lambda^{1-\alpha}$ means that
$ |l| \in [\varepsilon_0 \lambda^{1-\alpha}, \ \varepsilon_1 \lambda^{1-\alpha}]$ for some $0<\varepsilon_0<\varepsilon_1\ll 1.$
If $|l| \sim \lambda^{1-\alpha}$,  then  $\|f^{(l)}\|_p\sim\lambda^{n(1-\alpha)\big(\frac1p-1\big)}$, and
\begin{equation}
(f^{(l)})_{\lambda}=  {\rm e}^{{\rm i}x \langle l \rangle^{\frac{\alpha}{1-\alpha}}l } \  \tau_{\lambda^{1-\alpha} l}  \left( \rho \left(\frac{ \cdot}{ c  \langle l \rangle^{\frac{\alpha}{1-\alpha}}} \right)  \right)^\vee,
\end{equation}
which follows that $(\mathscr{F}^{-1}\rho^\alpha_l)* (f^{(l)})_{\lambda}=(f^{(l)})_{\lambda}$. Since supp $\rho^\alpha_l$ overlaps at most finite many supp $\rho^\alpha_{l+k}$, we see that $\Box^\alpha_{l+k} (f^{(l)})_\lambda =0$ if $|k| \ge C$. Let $A(\lambda) \subset \{l: |l|\sim \lambda^{1-\alpha}\}$ be the set so that for any $l, \tilde{l} \in A(\lambda)$ ($l\neq \tilde{l}$), $|l-\tilde{l}|\ge C$. We have
\begin{equation}
\begin{split}
\left\| f_{\lambda}\right\|_{M_{p,q}^{s,\alpha}} & \gtrsim \left(\sum_{l \in A{(\lambda)}} \langle l \rangle^{\frac{sq}{1-\alpha}}\|(f^{(l)})_{\lambda}\|_p^q\right)^{\frac1q}
\\
&\gtrsim\lambda^{-\frac{n}{p}+s}\left(\sum_{l \in A{(\lambda)}} \|f^{(l)}\|_p^q\right)^{\frac1q}
\\
&\gtrsim\lambda^{-\frac{n}{p}+s+n(1-\alpha)\big(\frac1q+\frac1p -1\big)}.
\end{split}
\end{equation}
Moreover, we easily see that
$$
|\widehat{f}^{(l)}(\xi)|= \rho \left(\frac{\lambda \xi -
\langle l\rangle^{\alpha/(1-\alpha)}l}{c \langle l\rangle^{\alpha/(1-\alpha)}}\right).
$$
It follows that ${\rm supp } \ \widehat{f}^{(l)}$ is included in the unit ball. Hence, we have
$$
\left\|f \right\|_{M_{p,q}^{s,\alpha}} \lesssim \left\|\sum_{l \in A{(\lambda)}}  f^{(l)}\right\|_{p}.
$$
By Plancheral's identity,
$$
\left\|f \right\|_{2} =\left\|\widehat{f} \right\|_{2} = \left(\sum_{l \in A{(\lambda)}}  \int \rho^{2} \left(\frac{\lambda \xi -
\langle l\rangle^{\alpha/(1-\alpha)}l}{c \langle l\rangle^{\alpha/(1-\alpha)}}\right) d\xi  \right)^{1/2} \sim 1.
$$
On the other hand, in view of $\mathscr{F}^{-1}\rho$ is a Schwartz function, we have
$$
|\mathscr{F}^{-1}\rho (x)| \lesssim \langle x\rangle^{-N}.
$$
It follows that for $N \gg n$,
$$
\left |\sum_{l \in A{(\lambda)}}  f^{(l)} (x)\right | \lesssim \lambda^{-n(1-\alpha)} \sum_{l \in A{(\lambda)}} \left(1+ \lambda^{-(1-\alpha)} |x- \lambda^{1-\alpha} l|\right)^{-N} \lesssim  \lambda^{-n(1-\alpha)}.
$$
By H\"older's inequality,
$$
\left\|f \right\|_{M_{p,q}^{s,\alpha}} \lesssim \left\|f \right\|_{p} \lesssim  \lambda^{n(1-\alpha)(2/p -1)}.
$$
The result follows.

\noindent{\bf Case 2.} $\lambda \ll 1$. It suffices to show that for some $f\in M^{s,\alpha}_{p,q}$,
\begin{equation*}
\|f_{\lambda}\|^{\circ}_{M_{p,q}^{s,\alpha}}\gtrsim\lambda^{-\frac{n}{p}} (1\vee \lambda^{ s+s_c})\|f\|_{M_{p,q}^{s,\alpha}}.
\end{equation*}

\noindent{\bf Case 2.1.} $s+s_c \ge 0$. Taking $f= \mathscr{F}^{-1}\rho$, we have
\begin{equation}
\|f_\lambda\|^{\circ}_{M_{p,q}^{s,\alpha}} =\|(\mathscr{F}^{-1}
\rho)(\lambda \, \cdot)\|_p
 \sim \lambda^{-\frac{n}{p}}\|f\|_{M_{p,q}^{s,\alpha}}.
\end{equation}

\noindent{\bf Case 2.2.} $s+s_c < 0$. We divide the proof into the following three cases.

{\bf Case 2.2.1.} We can find some $k_0$ such that $\lambda\langle k_0\rangle^{\frac{1}{1-\alpha}}\sim 1$. Denote
\begin{equation}
f={\rm e}^{{\rm i}x\langle k_0\rangle^{\frac{\alpha}{1-\alpha}}k_0}\mathscr{F}^{-1}\left( \rho \left(\frac{\cdot}{c \langle k_0\rangle^{\frac{\alpha}{1-\alpha}}} \right) \right),
\end{equation}
We have
\begin{equation}
\widehat{f}_\lambda = \lambda^{-n}     \rho \left(\frac{\xi/\lambda - \langle k_0\rangle^{\frac{\alpha}{1-\alpha}}k_0}{c \langle k_0\rangle^{\frac{\alpha}{1-\alpha}}} \right),
\end{equation}
 Therefore,
\begin{equation}
\|f_{\lambda}\|^{\circ}_{M_{p,q}^{s,\alpha}}\gtrsim\|f_{\lambda}\|_p\gtrsim\lambda^{-\frac{n}{p}+s}\langle k_0\rangle^{\frac{s}{1-\alpha}}\|f\|_p\gtrsim\lambda^{-\frac{n}{p}+s}\|f\|_{M_{p,q}^{s,\alpha}}.
\end{equation}

\noindent{\bf Case 2.2.2.} Taking $f= \lambda^n \mathscr{F}^{-1} \rho_\lambda $, we have $f_\lambda = \mathscr{F}^{-1} \rho$. It follows that $\|f_\lambda\|^{\circ}_{M_{p,q}^{s,\alpha}} \sim 1$. On the other hand,
\begin{equation}
\begin{split}
\|f\|^{\circ}_{M_{p,q}^{s,\alpha}}
& \lesssim \left(\sum_{|k| \lesssim \lambda^{\alpha-1} } \langle k \rangle^{\frac{sq}{1-\alpha}}\|(\mathscr{F}^{-1}\rho_{k}^{\alpha})*f \|_p^q\right)^{\frac1q}
\\
& \lesssim \left(\sum_{|k| \lesssim \lambda^{\alpha-1} } \langle k \rangle^{\frac{sq}{1-\alpha}}\| \mathscr{F}^{-1}\rho_{k}^{\alpha})\|^q_p \|f \|_1^q\right)^{\frac1q}
\\
& \lesssim
\lambda^{n-s-n(1-\alpha)\frac1q-n\alpha\big(1-\frac1p\big)}.
\end{split}
\end{equation}
It follows that
\begin{equation*}
\|f_{\lambda}\|^{\circ}_{M_{p,q}^{s,\alpha}}\gtrsim\lambda^{-\frac{n}{p}+s+s_c} \|f\|_{M_{p,q}^{s,\alpha}}.
\end{equation*}

\noindent{\bf Case 2.2.3.} $s_c= n(1-\alpha)(1/q-1/p)$.  Let $A(\lambda)$ be the set so that for any $l, \tilde{l} \in A(\lambda)$ ($l\neq \tilde{l}$), $|l-\tilde{l}|\ge C$ and $|l| \in [\varepsilon_0 \lambda^{\alpha-1}, \  \varepsilon_1 \lambda^{\alpha-1}]$ for some $0<\varepsilon_0< \varepsilon_1 \ll 1$. Take
\begin{equation}
f^{(l)}={\rm e}^{{\rm i}x\langle l \rangle^{\frac{\alpha}{1-\alpha}} l } \tau_{C'l} \mathscr{F}^{-1}\left(\rho \left( \frac{\cdot}{c \langle l\rangle^{\frac{\alpha}{1-\alpha}}} \right)\right)  , \ \ f=\sum_{l\in A(\lambda)} f^{(l)}.
\end{equation}
One easily sees that
\begin{equation}
(f^{(l)})_{\lambda}= \lambda^{-n} {\rm e}^{{\rm i}x\lambda\langle l \rangle^{\frac{\alpha}{1-\alpha}}l } \tau_{C'l} \mathscr{F}^{-1}\left(\rho \left( \frac{\cdot}{c \lambda \langle l\rangle^{\frac{\alpha}{1-\alpha}}} \right)\right).
\end{equation}
We have
\begin{equation}
\begin{split}
\left\| f \right\|^{\circ}_{M_{p,q}^{s,\alpha}} &=\left(\sum_{l \in A{(\lambda)}}
\langle l \rangle^{\frac{sq}{1-\alpha}}\|(f^{(l)})\|_p^q\right)^{\frac1q}
\\
&\lesssim \lambda^{-s}\left(\sum_{l \in A{(\lambda)}}  \langle l \rangle^{\frac{n\alpha}{1-\alpha}(1-\frac{1}{p})q}   \right)^{\frac1q}
\\
&\lesssim \lambda^{- s+n\alpha \big(\frac1p -1\big) + \frac{n(\alpha-1)}{q}}.
\end{split}
\end{equation}
Since ${\rm supp } \widehat{f}^{(l)}_\lambda$ is contained in the unit ball, we see that
$$
\left\| f_\lambda \right\|_{M_{p,q}^{s,\alpha}} \gtrsim  \left\| \sum_{l \in A{(\lambda)}} {f}^{(l)}_\lambda \right\|_{p} = \lambda^{-n/p} \left\| \sum_{l \in A{(\lambda)}} {f}^{(l)} \right\|_{p}.
$$
We note  that
$$
f^{(l)} (x) =  \langle l\rangle^{\frac{n\alpha}{1-\alpha}} {\rm e}^{{\rm i} x \langle l \rangle^{\frac{\alpha}{1-\alpha}}l }  (\mathscr{F}^{-1} \rho ) \left(  c   \langle l\rangle^{\frac{\alpha}{1-\alpha}} (x- C' l)\right).
$$
Due to $\mathscr{F}^{-1}\rho $ is a Schwartz function, we see that $|\mathscr{F}^{-1}\rho (x)| \lesssim \langle x\rangle^{-2n}$. We can assume that $\mathscr{F}^{-1}\rho (0)=1$, which follows that there exists a $\delta>0$ such that
$$
|\mathscr{F}^{-1}\rho (x)| \ge 1/2, \quad |x| \le \delta.
$$
Hence,  for any $k\in A(\lambda)$,  $x\in B(C'k, \delta/c \langle k\rangle^{\frac{\alpha}{1-\alpha}})$,
$$
\left| \sum_{l\in A(\lambda)} f^{(l)} (x) \right|  \ge \frac{\lambda^{-n\alpha}}{2}  - C \lambda^{-n\alpha} \sup_{|y|\lesssim \delta}\sum_{l \in \mathbb{Z}^n}
\left(y + C'c \ l \right)^{-2n}.
$$
We can take $C'c\gg 1$. It follows that for any $k\in A(\lambda)$,
$$
\left|  \sum_{l\in A(\lambda)} f^{(l)} (x) \right|  \ge \frac{\lambda^{-n\alpha}}{2}  - C  \lambda^{-n\alpha} (C'c)^{-n} \ge  \frac{\lambda^{-n\alpha}}{4}, \quad x\in B(C'k, \delta/c \langle k\rangle^{\frac{\alpha}{1-\alpha}}).
$$
So, we have
$$
\left\|  \sum_{l\in A(\lambda)} f^{(l)} (x) \right\|_p   \gtrsim   \lambda^{-n\alpha}    \lambda^{\frac{2n\alpha-n}{p}}.
$$
It follows that
\begin{equation*}
\|f_{\lambda}\|^{\circ}_{M_{p,q}^{s,\alpha}}\gtrsim\lambda^{- n/p +s+n(1-\alpha)(1/q-1/p)} \|f\|_{M_{p,q}^{s,\alpha}}.
\end{equation*}
The result follows.
 \end{proof}

\begin{proof} [Proof of Theorem \ref{embedmb}]
({\bf Necessity}) {\bf Case 1.} Let us assume that $B^{s_1}_{p,q} \subset M^{s_2}_{p,q}$.

\noindent{\bf Case 1.1.} We show that $s_1 \ge s_2$. Let $k\in \mathbb{Z}^n$, $|k| \gg 1$ with $B(\langle k\rangle^{\frac{\alpha}{1-\alpha}}k, C\langle k\rangle^{\frac{\alpha}{1-\alpha}}) \subset \{\xi : \ 5\cdot 2^{j-3} \le |\xi| \le 3\cdot 2^{j-1}\}$.  We see that
\begin{equation*}
\left\|\mathscr{F}^{-1}\left( \rho_k^{\alpha} (2 \ \cdot) \right)\right\|^{\circ}_{M_{p,q}^{s_2,\alpha}}\geqslant \langle k \rangle^{\frac{s_2}{1-\alpha}}\|\mathscr{F}^{-1}\rho_{k}^{\alpha} (2 \ \cdot) \|_p\gtrsim\langle k \rangle^{\frac{s_2}{1-\alpha}+\frac{n\alpha}{1-\alpha}\big(1-\frac1p\big)},
\end{equation*}
\begin{equation*}
\left\|\mathscr{F}^{-1}\left( \rho_k^{\alpha}(2 \ \cdot) \right)\right\|_{B_{p,q}^{s_1}}\leqslant 2^{j s_1}\|\mathscr{F}^{-1}\rho_{k}^{\alpha} (2 \ \cdot)\|_p\lesssim\langle k \rangle^{\frac{s_1 }{1-\alpha}+\frac{n\alpha}{1-\alpha}\big(1-\frac1p\big)}.
\end{equation*}
$B_{p,q}^{s_1}\subset M_{p,q}^{s_2,\alpha}$ follows that $s_1\geqslant s_2$.

\noindent{\bf Case 1.2.} We show that $s_1 \ge s_2 +n(1-\alpha) (1/p+1/q-1)$. One has that
\begin{equation}
\|\mathscr{F}^{-1}\varphi_{j}\|_{B_{p,q}^{s_1}}\lesssim2^{j s_1 +j n\big(1-\frac1p\big)}.
\end{equation}
Denote
$$
A_j = \{k: \ {\rm supp} \rho_k^{\alpha} \subset \{\xi : \ 5\cdot 2^{j-3} \le |\xi| \le 3\cdot 2^{j-1}\} \}.
$$
\begin{equation}
\|\mathscr{F}^{-1}\varphi_{j}\|^\circ_{M_{p,q}^{s_2,\alpha}}\geqslant\left(\sum_{k\in A_j} \langle k\rangle^{\frac{s_2q}{1-\alpha}}\|\mathscr{F}^{-1}\rho_k^{\alpha}\|_p^q\right)^{\frac1q}\gtrsim\left(\sum_{k\in A_j}\langle k\rangle^{\frac{s_2q}{1-\alpha}+\frac{n\alpha}{1-\alpha}\big(1-\frac1p\big)q}\right)^{\frac1q}.
\end{equation}
Noticing that $\# A_j \sim O(2^{nj(1-\alpha)})$, we immediately have
\begin{equation}
\|\mathscr{F}^{-1}\varphi_{j}\|^\circ_{M_{p,q}^{s_2,\alpha}} \gtrsim 2^{ s_2 j+ n\alpha j \big(1-\frac1p\big) + n(1-\alpha)j \frac1q }.
\end{equation}
$B_{p,q}^{s_1}\subset M_{p,q}^{s_2,\alpha}$ implies that $s_1 \geqslant s_2+n(1-\alpha)\big(\frac1p+\frac1q-1\big)$.

\noindent{\bf Case 1.3.} We show that $s_1 \geqslant s_2+n(1-\alpha)\big(\frac1q- \frac1p\big)$.
We denote by $A_j$ the set such that for every $k,l \in \mathbb{Z}^n \cap A_j$ ($l\neq \tilde{l}$), $|k-l|\ge C$ and $|k|^{1/(1-\alpha)} \in [5\cdot 2^{j-3}+C2^{j\alpha}, 3\cdot 2^{j-1}-C2^{j\alpha}]$.  Put
\begin{equation}\label{b-m-necessity-2-2}
f^{(k)}= \tau_{k} \left( \mathscr{F}^{-1}\rho_{k}^{\alpha} \right), \quad \quad f=\sum_{k\in A_j} f^{(k)}.
\end{equation}
Noticing that $\# A_j \sim O(2^{nj(1-\alpha)})$, we have
\begin{equation}
\left\|f \right\|^{\circ}_{M_{p,q}^{s_2,\alpha}}=\left(\sum_{k \in A_j}\langle k \rangle^{\frac{s_2q}{1-\alpha}}\|f^{(k)}\|_p^q\right)^{\frac1q} \gtrsim 2^{j \left(s_2+n(1-\alpha)\frac1q+n\alpha\big(1-\frac1p\big)\right)}.
\end{equation}
On the other hand,
\begin{equation}
\left\|f \right\|_{B_{p,q}^{s_1}}  \lesssim 2^{j s_1}\left\|\sum_{k\in A_j} f^{(k)}\right\|_p.
\end{equation}
By Plancherel's identity,
$$
\|f\|_2 \lesssim \left(\sum_{k\in A_j } \|\rho^\alpha_k (2 \ \cdot )\|^2_2 \right)^{1/2} \lesssim 2^{nj/2}.
$$
Moreover, let us observe that
$$
f^{(k)}= C\langle k\rangle^{\frac{n\alpha}{1-\alpha}} (\mathscr{F}^{-1} \rho)(C\langle k\rangle^{\frac{n\alpha}{1-\alpha}}(x- k)).
$$
Using the same way as in the proof of Case 1.2.3 in Theorem \ref{dilationP},  we have
$$
|f(x)| \lesssim 2^{n\alpha j}.
$$
By H\"older's inequality,
$$
\|f\|_p \lesssim 2^{j n\alpha(1-2/p) + nj/p}.
$$
It follows from $B_{p,q}^{s_1}\subset M_{p,q}^{s_2,\alpha}$ that $s_1\geqslant s_2+n(1-\alpha)\big(\frac1q-\frac1p\big)$.

\noindent{\bf Case 2.} We assume that $M_{p,q}^{s_1, \alpha_1} \subset B_{p,q}^{s_2}$.

\noindent{\bf Case 2.1.} Let $j \gg 1$. One can find some $k \in \mathbb{Z}^n$ verifying
  $\rho_{k}^{\alpha}\varphi_{j}=\rho_{k}^{\alpha}$. It follows that $2^j \sim  \langle k\rangle^{1/(1-\alpha)}$.  Thus,
\begin{equation}
\|\mathscr{F}^{-1}\rho_{k}^{\alpha}\|_{B_{p,q}^{s_2}}\geqslant2^{j s_2}\|\mathscr{F}^{-1}\rho_{k}^{\alpha}\|_p \gtrsim 2^{j s_2}\langle k \rangle^{\frac{n\alpha}{1-\alpha}\big(1-\frac1p\big)},
\end{equation}
\begin{equation}
\|\mathscr{F}^{-1}\rho_{k}^{\alpha}\|^{\circ}_{M_{p,q}^{s_1,\alpha}}\lesssim \langle k \rangle^{\frac{s_1}{1-\alpha}+\frac{n\alpha}{1-\alpha}\big(1-\frac1p\big)}.
\end{equation}
So, we have $s_1 \geqslant s_2$.

\noindent{\bf Case 2.2.} Let $j \gg 1$.
We have
\begin{equation}
\left\|\mathscr{F}^{-1} \varphi_j \right\|_{B_{p,q}^{s_2}}\geqslant 2^{j s_2}\|\mathscr{F}^{-1}\varphi^2_{j}\|_p \gtrsim2^{j s_2+j n\big(1-\frac1p\big)}.
\end{equation}
On the other hand,
\begin{equation}
\left\|\mathscr{F}^{-1} \varphi_j \right\|_{M_{p,q}^{s_1,\alpha}}\lesssim \left(\sum_{|k|^{1/(1-\alpha)}\sim 2^j} \langle k\rangle^{\frac{s_1 q}{1-\alpha}}2^{j n\alpha\big(1-\frac1p\big)q}\right)^{\frac1q}\lesssim2^{j n(1-\alpha)\frac1q+s_1j+j n\alpha\big(1-\frac1p\big)}.
\end{equation}
It follows from $M_{p,q}^{s_1,\alpha}\subset B_{p,q}^{s_2}$  that $s_1 \geqslant s_2+n(1-\alpha)\big(1-\frac1p-\frac1q\big)$.

\noindent{\bf Case 2.3.}
We denote by $A_j$ the set such that for every $k,l \in \mathbb{Z}^n \cap A_j$ ($l\neq k$), $|k-l|\ge C$ and $|k|^{1/(1-\alpha)} \in [5\cdot 2^{j-3}+C2^{j\alpha}, \ 3\cdot 2^{j-1}-C2^{j\alpha}]$.
Put
\begin{equation}
f^{(k)}=  \tau_{C'k} \mathscr{F}^{-1}\left(\rho \left( \frac{\cdot \ - \langle k \rangle^{\frac{\alpha}{1-\alpha}} k}{c \langle k \rangle^{\frac{\alpha}{1-\alpha}}} \right)\right)  , \ \ f=\sum_{k\in A_j} f^{(k)}.
\end{equation}
Using the same way as in the proof of Case 2.2.3  in Theorem \ref{dilationP},
\begin{equation}
\left\|f \right\|^{\circ}_{M_{p,q}^{s_1,\alpha}} \lesssim2^{j \left[s_1+n(1-\alpha)\frac1q+n\alpha\big(1-\frac1p\big)\right]},
\end{equation}
\begin{equation}
\left\|f \right\|_{B_{p,q}^{s_2}}=2^{j s_2}\left\|f \right\|_p \gtrsim2^{j \left[s_2+n(1-\alpha)\frac1p+n\alpha\big(1-\frac1p\big)\right]}.
\end{equation}
From $M_{p,q}^{s_1,\alpha}\subset B_{p,q}^{s_2}$ it follows that $s_1 \geqslant s_2+n(1-\alpha)\big(\frac1p-\frac1q\big)$.
\end{proof}

\begin{proof} [Proof of Theorem \ref{embedding}]
({\bf Necessity}) We separate the proof into the following two cases.

\noindent{\bf Case 1.} We assume that
$M_{p,q}^{s_1,\alpha_1}\subset M_{p,q}^{s_2,\alpha_2}$ with $\alpha_1\geqslant\alpha_2$.

\noindent{\bf Case 1.1.} We show that $s_1\ge s_2$.  Denote
$$
\Lambda_2(k)= \{l \in \mathbb{Z}^n: \ \rho^{\alpha_2}_k \rho^{\alpha_1}_l \neq 0 \}.
$$
We have
\begin{equation}
\left\|\mathscr{F}^{-1} \rho_k^{\alpha_2} \right\|^\circ_{M_{p,q}^{s_2,\alpha_2}} \geqslant \langle k \rangle^{\frac{s_2}{1-\alpha_2}} \|\mathscr{F}^{-1}\rho_{k}^{\alpha_2} \rho^{\alpha_2}_k\|_p \gtrsim \langle k \rangle^{\frac{s_2}{1-\alpha_2}+\frac{n\alpha_2}{1-\alpha_2}\big(1-\frac1p\big)},
\end{equation}
and
\begin{equation}
\left\|\mathscr{F}^{-1} \rho_k^{\alpha_2} \right\|^\circ_{M_{p,q}^{s_1,\alpha_1}}
\lesssim\left(\sum_{l \in \Lambda_2 (k) } \langle l \rangle^{\frac{s_1 q}{1-\alpha_1}}
\| \mathscr{F}^{-1} \rho_k^{\alpha_2} \rho_l^{\alpha_1}\|^q_p  \right)^{\frac1q}.
\end{equation}
Using Young's inequality  and Proposition \ref{convolution}, respectively for $p\ge 1$ and $p<1$,  we have
$$
\| \mathscr{F}^{-1} \rho_k^{\alpha_2} \rho_l^{\alpha_1}\|_p \lesssim \| \mathscr{F}^{-1} \rho_k^{\alpha_2} \|_p \lesssim   \langle k \rangle^{\frac{n\alpha_2}{1-\alpha_2}\big(1-\frac1p\big)}.
$$
Since $\# \Lambda_2(k) $ is finite and $|k|^{\frac{1}{1-\alpha_2}} \sim |l|^{\frac{1}{1-\alpha_1}} $ for all $l\in \lambda_2(k)$, one has that
\begin{equation}
\left\|\mathscr{F}^{-1} \rho_k^{\alpha_2} \right\|^\circ_{M_{p,q}^{s_1,\alpha_1}} \lesssim\left(\sum_{k\in \Lambda_2(k)} \langle k\rangle^{\frac{sq}{1-\alpha_1}}\langle k \rangle^{\frac{n\alpha_2}{1-\alpha_2}\big(1-\frac1p\big)q}\right)^{\frac1q} \lesssim \langle k \rangle^{\frac{s}{1-\alpha_2}+\frac{n\alpha_2}{1-\alpha_2}\big(1-\frac1p\big)}.
\end{equation}
From $M_{p,q}^{s,\alpha_1}\subset M_{p,q}^{s_2,\alpha_2}$ it follows that $s_1 \geqslant s_2$.

\noindent{\bf Case 1.2.}
Let $k \in\mathbb{Z}^n$ with $|k|\gg 1$. Denote
\begin{equation}\label{m-m-necessity-1-2-1}
\Lambda^*(k)=\left\{l\in\mathbb{Z}^n: \rho_{k}^{\alpha_1} \rho_l^{\alpha_2}= \rho_l^{\alpha_2}\right\}.
\end{equation}
We have
\begin{equation}
\|\mathscr{F}^{-1}\rho_{k}^{\alpha_1}\|_{M_{p,q}^{s_1,\alpha_1}}
\lesssim \langle k \rangle^{\frac{s_1}{1-\alpha_1}+\frac{n\alpha_1}{1-\alpha_1}\big(1-\frac1p\big)}.
\end{equation}
Since $\# \Lambda^* (k) \sim  \langle k\rangle^{\frac{n(\alpha_1-\alpha_2)}{1-\alpha_1}} $  and $\langle k \rangle^{\frac{1}{1-\alpha_1}} \sim \langle l \rangle^{\frac{1}{1-\alpha_2}} $ for all $l\in \Lambda^* (k)$, one has that
\begin{equation}
\|\mathscr{F}^{-1}\rho_{k}^{\alpha_1}\|^\circ_{M_{p,q}^{s_2,\alpha_2}}
 \geqslant \left(\sum_{l\in\Lambda^*(k)} \langle l\rangle^{\frac{s_2q}{1-\alpha_2}} \|\mathscr{F}^{-1}\eta_l^{\alpha_2}\|_p^q\right)^{\frac1q} \gtrsim\langle k \rangle^{\frac{n(\alpha_1-\alpha_2)}{q(1-\alpha_1)}+\frac{s_2}{1-\alpha_1}+\frac{n\alpha_2}{1-\alpha_1}\big(1-\frac1p\big)}.
\end{equation}
$M_{p,q}^{s,\alpha_1}\subset M_{p,q}^{s_2,\alpha_2}$ implies that $s_1\geqslant s_2+n(\alpha_1-\alpha_2)\big(\frac1p+\frac1q-1\big)$.

\noindent{\bf Case 1.3.} We show that $s_1\geqslant s_2+n(\alpha_1-\alpha_2)\big(\frac1q-\frac1p\big)$.
Let $\Lambda^*_0(k)$ be the subset of $\Lambda^*(k)$ such that $|l-\tilde{l}| \ge C$ for all $l, \tilde{l} \in \Lambda^*_0(k)$ ($l\neq \tilde{l}$). Denote
\begin{equation}
f^{(l)}= \tau_{l} \mathscr{F}^{-1}\rho^{\alpha_2}_l, \quad \quad  f= \sum_{l\in \Lambda^*_0(k)} f^{(l)}.
\end{equation}
It follows that
\begin{equation}
\left\|f \right\|^\circ_{M_{p,q}^{s_2,\alpha_2}}=\left( \sum_{l\in \Lambda^*_0(k)} \langle l \rangle^{\frac{s_2q}{1-\alpha_2}}\|f^{(l)}\|_p^q\right)^{\frac1q}
\gtrsim \langle k \rangle^{\frac{s_2}{1-\alpha_1}+\frac{n(\alpha_1-\alpha_2)}{q(1-\alpha_1)}+\frac{n\alpha_2}{1-\alpha_1}\big(1-\frac1p\big)}.
\end{equation}
On the other hand,
$$
\|f\|_2 = \left\| \sum_{l\in \Lambda^*_0(k)} \rho_l^{\alpha_2} \right\|_2 \sim  \langle k \rangle^{\frac{n\alpha_2}{2(1-\alpha_1)} +\frac{n(\alpha_1-\alpha_2)}{2(1-\alpha_1)} },
$$
and noticing that $|f^{(l)}(x)| =  \langle l \rangle^{\frac{n\alpha_2}{(1-\alpha_2)} } |(\mathscr{F}^{-1})(C \langle l \rangle^{\frac{\alpha_2}{(1-\alpha_2)} } (x-l))|$, we have
$$
\|f\|_\infty  \lesssim  \langle k \rangle^{\frac{n\alpha_2}{(1-\alpha_1)}}.
$$
Hence,
$$
\|f\|_p \lesssim  \langle k \rangle^{\frac{n\alpha_2}{(1-\alpha_1)}(1-\frac{1}{p}) +\frac{n(\alpha_1-\alpha_2)}{p(1-\alpha_1)} }.
$$
It follows that
\begin{equation}
\left\| f\right\|^\circ_{M_{p,q}^{s_1,\alpha_1}} \lesssim \langle k \rangle^{\frac{s_1}{1-\alpha_1}} \left\|f\right\|_p
\lesssim\langle k \rangle^{\frac{s_1}{1-\alpha_1}+\frac{n(\alpha_1-\alpha_2)}{p(1-\alpha_1)}+\frac{n\alpha_2}{1-\alpha_1}\big(1-\frac1p\big)}.
\end{equation}
$M_{p,q}^{s,\alpha_1}\subset M_{p,q}^{s_2,\alpha_2}$ implies that $s_1\geqslant s_2+n(\alpha_1-\alpha_2)\big(\frac1q-\frac1p\big)$.

\noindent{\bf Case 2.} We assume that
$M_{p,q}^{s_1,\alpha_1}\subset M_{p,q}^{s_2,\alpha_2},\alpha_1<\alpha_2$. The idea is the same as in the proof of Theorem   \ref{embedmb} and we only give a sketch proof.

\noindent{\bf Case 2.1.} We show that $s_1\ge s_2$.
Let $k \in\mathbb{Z}^n$, $|k| \gg 1$.  One can find some $l \in\mathbb{Z}^n$  such that $\rho_{k}^{\alpha_1} \rho_{l}^{\alpha_2}=\rho_{k}^{\alpha_1}$, Then,
\begin{equation}
\|\mathscr{F}^{-1}\rho_{k }^{\alpha_1}\|^\circ_{M_{p,q}^{s_2,\alpha_2}}  \gtrsim \langle k \rangle^{\frac{s_2}{1-\alpha_1}+\frac{n\alpha_1}{1-\alpha_1}\big(1-\frac1p\big)},
\end{equation}
\begin{equation}
\|\mathscr{F}^{-1}\rho_{k }^{\alpha_1}\|^\circ_{M_{p,q}^{s_1,\alpha_1}} \lesssim\langle k \rangle^{\frac{s_1}{1-\alpha_1}+\frac{n\alpha_1}{1-\alpha_1}\big(1-\frac1p\big)}.
\end{equation}
$M_{p,q}^{s_1,\alpha_1}\subset M_{p,q}^{s_2,\alpha_2}$ implies that $s_1\geqslant s_2$.

\noindent{\bf Case 2.2.} We show that $s_1 \geqslant s_2+n(\alpha_1-\alpha_2)\big(\frac1p+\frac1q-1\big)$.
\begin{equation}
\left\|\mathscr{F}^{-1}\rho_k^{\alpha_2} \right\|^\circ_{M_{p,q}^{s_2,\alpha_2}}   \gtrsim\langle k \rangle^{\frac{s_2}{1-\alpha_2}+\frac{n\alpha_2}{1-\alpha_2}\big(1-\frac1p\big)},
\end{equation}
\begin{equation}
\left\|\mathscr{F}^{-1} \rho_k^{\alpha_2} \right\|^\circ_{M_{p,q}^{s_1,\alpha_1}} \lesssim \langle k \rangle^{\frac{n(\alpha_2-\alpha_1)}{q(1-\alpha_2)}+\frac{s_1}{1-\alpha_2}+\frac{n\alpha_1}{1-\alpha_2}\big(1-\frac1p\big)}.
\end{equation}
  $M_{p,q}^{s_1,\alpha_1}\subset M_{p,q}^{s_2,\alpha_2}$ implies that $s_1 \geqslant s_2+n(\alpha_1-\alpha_2)\big(\frac1p+\frac1q-1\big)$.

\noindent{\bf Case 2.3.} Let $k \in\mathbb{Z}^n$ with $|k |\gg1$, and
\begin{equation}
\Lambda^*(k )=\left\{l \in\mathbb{Z}^n:\rho_l^{\alpha_1}\rho_{k}^{\alpha_2}=\rho_l^{\alpha_1}\right\}.
\end{equation}
Let $\Lambda^*_0(k)$ be the subset of $\Lambda^*(k)$ such that $|l-\tilde{l}| \ge C$ for all $l, \tilde{l} \in \Lambda^*_0(k)$ ($l\neq \tilde{l}$). It is easy to see that  $\#\Lambda^*_0 (k) \sim\langle k \rangle^{\frac{n(\alpha_2-\alpha_1)}{1-\alpha_2}}$.   Put
\begin{equation}
f^{(l)}= \tau_{Cl} \mathscr{F}^{-1}\left( \rho_{l}^{\alpha_1}\right), \quad \quad  f= \sum_{l\in \Lambda^*_0(k)} f^{(l)}.
\end{equation}
then,
\begin{equation}
\left\|f\right\|^\circ_{M_{p,q}^{s_1,\alpha_1}} \lesssim \langle k \rangle^{\frac{s_1}{1-\alpha_2}+\frac{n(\alpha_2-\alpha_1)}{q(1-\alpha_2)}+\frac{n\alpha_1}{1-\alpha_2}\big(1-\frac1p\big)},
\end{equation}
\begin{equation}
\left\| f \right\|^\circ_{M_{p,q}^{s_2,\alpha_2}} \gtrsim
\langle k \rangle^{\frac{s_2}{1-\alpha_2}+\frac{n(\alpha_2-\alpha_1)}{p(1-\alpha_2)}+\frac{n\alpha_1}{1-\alpha_2}\big(1-\frac1p\big)}.
\end{equation}
$M_{p,q}^{s_1,\alpha_1}\subset M_{p,q}^{s_2,\alpha_2}$ implies that $s_1 \geqslant s_2+n(\alpha_2-\alpha_1)\big(\frac1p-\frac1q\big)$.
\end{proof}

\section{Counterexamples for the algebra structure}

In order to show that our results are sharp, we need an $\alpha$-covering which is a slightly modified version in \cite{FHW11}. Let $Q(a, r):= \prod^n_{i=1}[a_i-r, a_i+r]$ and we consider the following covering of $\mathbb{R}$:
$$
Q_0 = [-1,1], \
\  Q_j= Q(|j|^{\alpha/(1-\alpha)}j, \ r_j |j|^{\alpha/(1-\alpha)}), \ j\neq 0.
$$

\begin{lem} \label{lem6.1}
Let $ r  >1/2(1-\alpha) $.  There exists $j_0 \in \mathbb{N}$,  such that $\{ Q_j\}_{j\in \mathbb{Z}}$ is an $\alpha$-covering of $\mathbb{R}$, where
$$
r_j =  \left\{
\begin{array}{ll}
 r, \ \ & |j| > j_0, \\
{\rm suitable}, \ \ & |j| \le  j_0.
\end{array}
\right.
$$
Moreover, if $r < 8/15(1-\alpha)$, then
\begin{align}
Q_{j\pm 1}  \cap Q(|j|^{\alpha/(1-\alpha)}j, \ \frac{7}{8} r_j |j|^{\alpha/(1-\alpha)}) = \ \varnothing  \label{alth}
\end{align}
for all $j \in \mathbb{Z}$.
\end{lem}

\begin{proof}
Let $j>100$. Noticing that $Q_{j+1} \cap Q_j \neq \varnothing$ if and only if
\begin{align}
|j+1|^{\alpha/(1-\alpha)}(j+1) - r_{j+1} |j+1|^{\alpha/(1-\alpha)} < |j|^{\alpha/(1-\alpha)} j  - r_j |j|^{\alpha/(1-\alpha)}. \label{j+1j}
\end{align}
In view of mean value theorem, we see that \eqref{j+1j} is equivalent to
\begin{align}
\frac{1}{1-\alpha} |j+\theta|^{\alpha/(1-\alpha)} < r_{j+1} |j+1|^{\alpha/(1-\alpha)}  + r_j |j|^{\alpha/(1-\alpha)}, \label{mj+1j}
\end{align}
where $\theta \in (0,1)$. Take $r_{j+1}=r_j = r $.  Hence, there exists $j_0:= j_0(\alpha)$ such that for any $j>j_0$, \eqref{mj+1j} holds.  Next, if $j>j_0$, we have
 \begin{align}
|j+1|^{\alpha/(1-\alpha)}(j+1) - r  |j+1|^{\alpha/(1-\alpha)} > |j|^{\alpha/(1-\alpha)} j  + \frac{7}{8} r  |j|^{\alpha/(1-\alpha)}, \label{mmj+1j}
\end{align}
which implies \eqref{alth} for $j>j_0$. If $j \le j_0$, one can choose suitable $r_j$ so that the conclusion holds.
\end{proof}

Using Lemma \ref{lem6.1} and the idea as in \cite{FHW11}, we now construct a new $\alpha$-covering of $\mathbb{R}^n$, where the original idea goes back to Lizorkin's dyadic decomposition to $\mathbb{R}^n$. Let $j \in \mathbb{Z}$ with $|j| >j_0$. We may assume $8|j|/7 r  \in \mathbb{N}$.
We divide $[-|j|^{\frac{1}{1-\alpha}}, |j|^{\frac{1}{1-\alpha}}]$
into $16|j|/7 r $ equal intervals:
$$
[-|j|^{\frac{1}{1-\alpha}}, \,
|j|^{\frac{1}{1-\alpha}}]=[r_{j,-N_j}, r_{j, -N_j+1}]\cup ... \cup
[r_{j, N_j-1}, r_{j, N_j}].
$$
Denote
$$
\mathscr{R}=\{r_{j, s}: j\in \mathbb{N}, s=-N_j, \cdots, N_j\}.
$$
We further write
$$
\mathscr{K}^n_j=\{k= (k_1, \cdots, k_n): k_i \in \mathscr{R}, \,
\max_{1\leq i \leq n} |k_i| =|j|^{\frac{1}{1-\alpha}}\}.
$$
For any $k \in \mathscr{K}^n_j$, we write
$$
Q_{k j}=Q(k, r  |j|^{\frac{\alpha}{1-\alpha}}), \quad  |j| >j_0.
$$
From the construction of $Q_{k j}$ one sees that
\begin{equation}
\left\{
\begin{aligned}
& \# \{Q_{k' j'}: \  Q_{k' j'} \cap Q_{k j} \neq \varnothing \} =2n,   \\
 & Q(k, r  |j|^{\frac{\alpha}{1-\alpha}}/2 ) \cap Q_{k' j'} \neq \varnothing   \Leftrightarrow  k'=k, j'=j.
\end{aligned}
\right. \label{Qk>}
\end{equation}
For $|j| \le j_0$, one can choose suitable $r_j$, and in a similar way as above to define
$$
Q_{k j}=Q(k, r_j |j|^{\frac{\alpha}{1-\alpha}}), \quad  1\le |j| \le j_0, \ \ Q_0 =Q(0,r_0)
$$
so that
\begin{equation}
\left\{
\begin{aligned}
& \# \{Q_{k' j'}:\  Q_{k' j'} \cap Q_{k j} \neq \varnothing \} =2n,   \\
 & Q(k, r_j |j|^{\frac{\alpha}{1-\alpha}}/2 ) \cap Q_{k' j'} \neq \varnothing   \Leftrightarrow  k'=k, j'=j \ \ for \ j\neq 0, \\
 & Q(0, r_0 /2 ) \cap Q_{k' j'} \neq \varnothing   \Leftrightarrow   Q_{k' j'} =Q_0.
\end{aligned}
\right. \label{Qk<}
\end{equation}

\begin{lem} \label{lem6.2}
 $\{Q_0\} \cup \{Q_{kj}\}_{j\in \mathbb{Z}\setminus \{0\}, \ k\in \mathscr{K}_j}$ as in \eqref{Qk>} and \eqref{Qk<} is an $\alpha$-covering of $\mathbb{R}^n$.
\end{lem}

Let $\eta: \mathbb{R} \to [0,1]$ be a smooth bump function
satisfying
\begin{align}
\eta(\xi):= \left\{
\begin{array}{ll}
1, & |\xi|\leq 1/2,\\
{\rm smooth},  & 1/2<|\xi| \le 1,\\
 0, & |\xi|\geq 1.
\end{array}
\right.\label{eta0}
 \end{align}

Let $r$ and $r_j$ be as in \eqref{Qk>} and \eqref{Qk<}, respectively. Denote for $i=1,
\cdots, n$, $j\neq 0$,
$$
\phi_{kj}(\xi_i)=\eta \left(\frac{\xi_i-k_i}{r_j |
j |^{\frac{\alpha}{1-\alpha}}}\right) , \quad k=(k_1, \cdots,
k_n) \in \mathscr{K}_j, \ \ \phi_{0}(\xi_i)= \eta \left(\frac{\xi_i}{r_0  }\right),
$$
$$
\phi_{kj}(\xi)=\phi_{kj}(\xi_1)...\phi_{kj}(\xi_n), \ \  \phi_{0}(\xi)=\phi_{0}(\xi_1)...\phi_{0}(\xi_n).
$$
We put
\begin{align}
\psi_{kj}(\xi)=\frac{\phi_{kj}(\xi)}{\phi_0(\xi) + \sum_{k\in\mathscr{K}_j, j\in
\mathbb{Z} \setminus \{0\}}\phi_{kj}(\xi)}, \ \ \psi_{0}(\xi)=\frac{\phi_{0}(\xi)}{\phi_0(\xi) + \sum_{k\in\mathscr{K}_j, j\in
\mathbb{Z} \setminus \{0\}}\phi_{kj}(\xi)}. \label{psikj}
\end{align}

\begin{lem} \label{lem6.3}
 $ \{\psi_0\} \cup \{\psi_{kj}\}_{j\in \mathbb{Z}\setminus \{0\}, \ k\in \mathscr{K}_j}$ as in \eqref{psikj} is a  $p$-BAPU.
\end{lem}

On the basis of the above $p$-$BAPU$, we immediately have

\begin{prop} \label{prop6.4}
Let $0< \alpha <1$, $0<p, q \leqslant \infty$, then
$$
\|f\|_{M^{s, \alpha}_{p,q}}=\left( \|\mathscr{F}^{-1}\psi_{0}\mathscr{F}f\|^q_{L^p(\mathbb{R}^n)} + \sum_{j\in \mathbb{Z}\setminus \{0\}} \langle
j\rangle^{sq/(1-\alpha)} \sum_{k\in \mathscr{K}_j}
\|\mathscr{F}^{-1}\psi_{kj}\mathscr{F}f\|^q_{L^p(\mathbb{R}^n)}\right)^{1/q}
$$
is an equivalent norm on $\alpha$-modulation space.
\end{prop}

\begin{thm} \label{prop6.5}
Let $0 \le \alpha <1$, $ (1/p, 1/q) \in D_2, \ \   {p>1}$. If $s < s_0$, then $M^{s, \alpha}_{p,q}$ is not a Banach algebra.
\end{thm}
\begin{proof} {\bf Step 1.} {\bf $p>1$, $q\le 1$.}
Let $\chi_A$ be the characteristic function on $A$.  Now we take for $J\gg j_0$, $\ell= (J,J,...,J)$,
\begin{align}
\widehat{f} = \chi_{A(J)},  \ \ \widehat{g} = \chi_{-A(J)}, \ \ A(J) = Q(|J|^{\alpha/(1-\alpha)}\ell, \  r|J|^{\alpha/(1-\alpha)}/2 ). \label{A(J)}
\end{align}
Noticing that
\begin{equation}
(\chi_{[B-b, B+b]} * \chi_{[-B-b, -B+b]})(\xi)= \left\{
\begin{aligned}
0, & \ \  |\xi| \ge 2b,\\
2b-|\xi|, &  \ \  |\xi| < 2b.
\end{aligned}
\right. \label{chi}
\end{equation}
Hence,
\begin{equation}
\left(\chi_{A(J)} * \chi_{-A(J)}\right) (\xi)= \left\{
\begin{aligned}
0,  \quad  |\xi_i| \ge r |J|^{\frac{\alpha}{(1-\alpha)} } \ \mbox{for some} \ i=1,...,n,\\
\prod^n_{i=1} (r |J|^{\frac{\alpha}{(1-\alpha)} }- |\xi_i|),  \quad   |\xi_i| < r |J|^{\frac{\alpha}{(1-\alpha)} } \ \mbox{for all} \ i=1,...,n .
\end{aligned}
\right.
\end{equation}
Hence,
$$
{\rm supp } \widehat{f} * \widehat{g} = \{\xi \in \mathbb{R}^n: |\xi_i| \le r |J|^{\frac{\alpha}{(1-\alpha)} }, \ i=1,...,n\}.
$$
One has that
\begin{align*}
\prod^n_{i=1} (r |J|^{\frac{\alpha}{(1-\alpha)} } + \xi_i ) & = (r|J|^{\frac{\alpha}{(1-\alpha)} })^n +  (r|J|^{\frac{\alpha}{(1-\alpha)} })^{n-1} \sum_{i} \xi_i \\
& \quad +  (r|J|^{\frac{\alpha}{(1-\alpha)} })^{n-2} \sum_{i<j} \xi_i \xi_j +...+ \xi_1...\xi_n \\
& :=  (r|J|^{\frac{\alpha}{(1-\alpha)} })^n  + R(\xi, J).
\end{align*}
Let us write
$$
\mathscr{A}_j = \{k\in \mathscr{K}^n_j: \ \mbox{ if } \ \xi, \eta \in {\rm supp }\psi_{k,j}, \ \mbox{then} \ \xi_i \eta_i >0 \ \mbox{ for all } \ i=1,...,n\}.
$$
Let $1\ll j< \varepsilon |J|^\alpha$ $(0<\varepsilon \ll 1)$ and $k \in \mathscr{A}_j$. We may assume that $\xi_i >0$ if $\xi \in {\rm supp} \psi_{kj}$. Noticing that ${\rm supp}\psi_{kj} \subset
Q(0, r|J|^{\alpha/(1-\alpha)})$,  we have
\begin{align}
\|\mathscr{F}^{-1}\psi_{kj}\mathscr{F}(fg)\|_{ p} & =  \left\|\mathscr{F}^{-1}\psi_{kj}  \prod^n_{i=1} (r |J|^{\frac{\alpha}{(1-\alpha)} }-  \xi_i ) \right\|_p \nonumber\\
& \ge (r |J|^{\frac{ \alpha}{1-\alpha}})^n  \|\mathscr{F}^{-1}\psi_{kj}\|_p -   \left\|\mathscr{F}^{-1}(\psi_{kj} R( \xi,J ))   \right\|_p \nonumber\\
& \ge  c |J|^{\frac{n\alpha}{1-\alpha}} |j|^{\frac{n\alpha}{1-\alpha}(1-\frac{1}{p})}  -   \left\|\mathscr{F}^{-1}(\psi_{kj} R( \xi,J ))   \right\|_p.  \label{estfg}
\end{align}
\begin{align}
   \left\|\mathscr{F}^{-1}(\psi_{kj} R( \xi,J ))   \right\|_p &  \lesssim |J|^{\frac{(n-1) \alpha}{1-\alpha} } \left\|\mathscr{F}^{-1}(\psi_{kj}  \sum^n_{i=1} \xi_i)  \right\|_p  \nonumber\\
   & \ \ +  |J|^{\frac{(n-2) \alpha}{1-\alpha}} \left\|\mathscr{F}^{-1}(\psi_{kj}  \sum_{i<j} \xi_i \xi_j)  \right\|_p +...+ \left\|\mathscr{F}^{-1}(\psi_{kj}    \xi_1... \xi_n)  \right\|_p.
\end{align}
For instance, we estimate $ \|\mathscr{F}^{-1}(\psi_{kj} \xi_1 \xi_2)\|_p $. Let $k$ be the center of ${\rm supp} \psi_{kj}$.  We have
\begin{align}
  \left\|\mathscr{F}^{-1}(\psi_{kj}    \xi_1 \xi_2)  \right\|_p  &  \lesssim |k_1||k_2|\left\|\mathscr{F}^{-1} \psi_{kj}   \right\|_p  + \left\|\mathscr{F}^{-1}(\psi_{kj}    (\xi_1-k_1) (\xi_2-k_2)) \right\|_p \nonumber\\
   &  \quad + |k_1| \left\|\mathscr{F}^{-1}(\psi_{kj} (\xi_2-k_2)) \right\|_p +   |k_2| \left\|\mathscr{F}^{-1}(\psi_{kj} (\xi_1-k_1)) \right\|_p \nonumber\\
   & \lesssim |j|^{\frac{2}{1-\alpha} } |j|^{\frac{n\alpha}{1-\alpha}(1-\frac{1}{p})}
      \lesssim \varepsilon^{2} |J|^{\frac{2\alpha}{1-\alpha} } |j|^{\frac{n\alpha}{1-\alpha}(1-\frac{1}{p})}.
\end{align}
So, one has that
\begin{align}
   \left\|\mathscr{F}^{-1}(\psi_{kj} R( \xi,J ))   \right\|_p    \lesssim  \varepsilon  |J|^{\frac{n\alpha}{1-\alpha} } |j|^{\frac{n\alpha}{1-\alpha}(1-\frac{1}{p})}. \label{estR}
\end{align}
It follows from \eqref{estfg} and \eqref{estR} that
\begin{align}
\|\mathscr{F}^{-1}\psi_{kj}\mathscr{F}(fg)\|_{ p}
\gtrsim  |J|^{\frac{n\alpha}{1-\alpha} } |j|^{\frac{n\alpha}{1-\alpha}(1-\frac{1}{p})}.   \label{estfg2}
\end{align}
\eqref{estfg2} yields
\begin{align}
\|fg\|_{M^{s,\alpha}_{p,q}}  & \gtrsim  \left( \sum_{1\ll j \le \varepsilon |J|^\alpha} |j|^{\frac{sq}{1-\alpha}} \sum_{k\in \mathscr{A}_j} \|\mathscr{F}^{-1}\psi_{kj}\mathscr{F}(fg)\|^q_{ p} \right)^{1/q} \nonumber\\
& \gtrsim |J|^{\frac{n\alpha}{1-\alpha}} \left (\sum_{1\ll j \le \varepsilon |J|^\alpha} |j|^{\frac{sq}{1-\alpha}} \sum_{k\in \mathscr{A}_j}  |j|^{\frac{n\alpha q}{1-\alpha}(1-\frac{1}{p})}    \right)^{1/q} \nonumber\\
& \gtrsim |J|^{\frac{n \alpha}{(1-\alpha)} + \alpha \left( \frac{s}{(1-\alpha)}  + \frac{n\alpha}{1-\alpha}(1-\frac{1}{p}) + \frac{n}{q} \right) }.
\end{align}
On the other hand,
\begin{align}
\|f \|_{M^{s,\alpha}_{p,q}} \sim |J|^{  \frac{s}{(1-\alpha)} } \|\mathscr{F}^{-1} \chi_{A(J)}\|_p \sim |J|^{    \frac{s}{(1-\alpha)}  + \frac{n\alpha}{1-\alpha}(1-\frac{1}{p})   }.
\end{align}
Similarly,
\begin{align}
\|g\|_{M^{s,\alpha}_{p,q}}  \sim |J|^{    \frac{s}{(1-\alpha)}  + \frac{n\alpha}{1-\alpha}(1-\frac{1}{p})   }.
\end{align}
Hence, in order to $M^{s,\alpha}_{p,q}$ forms an algebra, one must has that
\begin{align}
   \frac{2s}{(1-\alpha)}  + \frac{2n\alpha}{1-\alpha}(1-\frac{1}{p})  \ge\frac{n \alpha}{(1-\alpha)} + \alpha \left( \frac{s}{(1-\alpha)}  + \frac{n\alpha}{1-\alpha}(1-\frac{1}{p}) + \frac{n}{q} \right).
\end{align}
Namely,
\begin{align}
 s\ge     \frac{ n\alpha}{ p}  +   \frac{n \alpha(1-\alpha)}{(2-\alpha)}\left(\frac{1}{q}-1\right).
\end{align}

{\bf Step 2. } $(1/p,1/q) \in [0,1]^2 \cap D_2$. Let $J\gg 1$. Put
$$
\widehat{f} (\xi) = \chi_{[J^{1/(1-\alpha)}, 3 J^{1/(1-\alpha)}]^n } (\xi), \quad  \widehat{g} (\xi) = \chi_{[-3J^{1/(1-\alpha)}, -J^{1/(1-\alpha)}]^n } (\xi).
$$
In view of \eqref{chi} we have
\begin{equation}
 (\widehat{f} * \widehat{g} ) (\xi)= \left\{
\begin{aligned}
0,  \quad  |\xi_i| \ge  2  J^{\frac{1}{(1-\alpha)} } \ \mbox{for some} \ i=1,...,n,\\
\prod^n_{i=1} (2 J^{\frac{1}{(1-\alpha)} }- |\xi_i|),  \quad   |\xi_i| \le 2 J^{\frac{1}{(1-\alpha)} } \ \mbox{for all} \ i=1,...,n .
\end{aligned}
\right.
\end{equation}
Hence,
$$
{\rm supp} (\widehat{f} * \widehat{g} ) \subset \{\xi: \ |\xi_i| \le 2 J^{\frac{1}{(1-\alpha)} } \   \ i=1,...,n \}.
$$
Using the same way as in \eqref{estfg2}, we have for $|j| \le \varepsilon J$ and $k\in \mathscr{A}_j$,
\begin{align}
\|\mathscr{F}^{-1}\psi_{kj}\mathscr{F}(fg)\|_{ p}
\gtrsim   J^{\frac{n }{1-\alpha} } |j|^{\frac{n\alpha}{1-\alpha}(1-\frac{1}{p})}.   \label{estfg2}
\end{align}
\eqref{estfg2} implies that
\begin{align}
\|fg\|_{M^{s,\alpha}_{p,q}}  & \gtrsim  \left( \sum_{1\ll j \le \varepsilon  J } |j|^{\frac{sq}{1-\alpha}} \sum_{k\in \mathscr{A}_j} \|\mathscr{F}^{-1}\psi_{kj}\mathscr{F}(fg)\|^q_{ p} \right)^{1/q} \nonumber\\
& \gtrsim  J^{\frac{n}{1-\alpha}} \left (\sum_{1\ll j \le \varepsilon  J} |j|^{\frac{sq}{1-\alpha}} \sum_{k\in \mathscr{A}_j}  |j|^{\frac{n\alpha q}{1-\alpha}(1-\frac{1}{p})}    \right)^{1/q} \nonumber\\
& \gtrsim  J^{\frac{n}{1-\alpha} + \frac{s}{1-\alpha}  + \frac{n\alpha}{1-\alpha}(1-\frac{1}{p}) + \frac{n}{q} }.
\end{align}
On the other hand,
\begin{align}
\|f \|_{M^{s,\alpha}_{p,q}}  & \lesssim   \left( \sum_{|j|   \sim   J } |j|^{\frac{sq}{1-\alpha}} \sum_{k\in \mathscr{K}_j} \|\mathscr{F}^{-1}\psi_{kj} \widehat{f} \|^q_{ p} \right)^{1/q} \nonumber\\
& \lesssim  J^{\frac{s}{1-\alpha}  + \frac{n\alpha}{1-\alpha}(1-\frac{1}{p}) + \frac{n}{q} }.
\end{align}
Similarly,
\begin{align}
\|g\|_{M^{s,\alpha}_{p,q}}   \lesssim  J^{\frac{s}{1-\alpha}  + \frac{n\alpha}{1-\alpha}(1-\frac{1}{p}) + \frac{n}{q} }.
\end{align}
Hence, in order to $M^{s,\alpha}_{p,q}$ forms an algebra, one must has that
\begin{align}
 s\ge     \frac{ n\alpha}{ p}  +  n (1-\alpha) \left(1 - \frac{1}{q}\right).
\end{align}
\end{proof}

{\bf Acknowledgement. \rm  The first named author is supported in part by the National Science Foundation of China, grant 11026053. Part of the work was carried out while the second named author was visiting the Beijing International Center for Mathematical Research (BICMR), he would like to thank BICMR for its hospitality. }


\begin{thebibliography}{50} \footnotesize

\bibitem{BoNi06} L. Borup, M. Nielsen,
 Banach frames for multivariate $\alpha$-modulation spaces,
 J. Math. Anal. Appl. \textbf{321} (2006), no. 2, 880--895.


\bibitem{BoNi062} L. Borup, M. Nielsen, Boundedness for pseudodifferential operators on multivariate   $alpha$-modulation spaces, Ark. Mat. \textbf{44} (2006) 241--259.


\bibitem{BoNi06a} L. Borup, M. Nielsen,
 Nonlinear approximation in $\alpha$-modulation spaces.
 Math. Nachr. \textbf{279} (2006), no. 1-2, 101--120.


\bibitem{Cal} A. P. Calder\'{o}n, Intermediate spaces and
interpolation, the complex method, Studia Math., {\bf 24} (1964),
113--190.

\bibitem{CalTor} A. P. Calder\'{o}n and A. Torchinsky,  Parabolic maximal functions
associated with a distribution, I, Advances in  Math., {\bf 16}
(1975), 1--64.
\bibitem{CalTor2} A. P. Calder\'{o}n and A. Torchinsky,  Parabolic maximal functions
associated with a distribution, II, Advances in  Math., {\bf 24}
(1977), 101--171.

 \bibitem{DaFor08} S. Dahlke, M. Fornasier, H. Rauhut, G. Steidl, G. Teschke, Generalized coorbit theory, Banach frames, and the relation to $\alpha$-modulation spaces,  Proc. Lond. Math. Soc.,  \textbf{96} (2008), no. 2, 464--506.

\bibitem{F} H. G. Feichtinger,  {\it Modulation spaces on locally compact Abelian
group},  Technical Report,  University of Vienna,  1983.


\bibitem{FHW11} H. G. Feichtinger, C. Y. Huang, B. X. Wang, Trace operators for modulation, $\alpha$-modulation and Besov spaces, Appl. Comput. Harmon. Anal., {\bf 30} (2011), 110--127.


\bibitem{For07} M.  Fornasier,
 Banach frames for $\alpha$-modulation spaces,
 Appl. Comput. Harmon. Anal. \textbf{22} (2007), no. 2, 157--175.



\bibitem{Gb} P. Gr\"{o}bner,  {\it Banachr\"{a}ume Glatter Funktionen and
Zerlegungsmethoden},  Doctoral thesis,  University of Vienna,  1992.

\bibitem{Gc} K. Gr\"{o}chenig,  {\it Foundations of Time-Frequency
Analysis},  Birkh\"{a}user Boston, MA,  2001.



\bibitem{KoSuTo091} M. Kobayashi, M. Sugimoto,  N. Tomita,  On the $L^2$-boundedness of pseudo-differential operators and their commutators with symbols in $\alpha$-modulation spaces. J. Math. Anal. Appl. \textbf{350} (2009), no. 1, 157--169.

\bibitem{KoSuTo09} M. Kobayashi, M. Sugimoto,  N. Tomita, Trace ideals for pseudo-differential operators and their commutators with symbols in $\alpha$-modulation spaces. J. Anal. Math. \textbf{107} (2009), 141--160.


\bibitem{ST} M. Sugimoto, N. Tomita,  {\it The dilation property of
modulation space and their inclusion relation with Besov spaces},
J. Funct. Anal.,  \textbf{248 } (2007),  79--106.


    \bibitem{T} J. Toft,  {\it Continuity properties for modulation spaces,  with applications to pseudo-differential
calculus, I},  J. Funct. Anal.,  \textbf{207} (2004),  399--429.

\bibitem{ToWa11} J. Toft and P. Wahlberg, Embeddings of $\alpha$-modulation spaces, Arxiv: 1110.2681.


\bibitem{triebel} H. Triebel,  {\it Theory of Function Spaces},
Birkh\"{a}user-Verlag, Basel, 1983.

\bibitem{WaHe07}  B. X. Wang and H. Hudzik,  {\rm The global Cauchy problem for the
NLS and NLKG with small rough data, } J. Differential Equations,  {\bf 232} (2007),  36--73.

\bibitem{WaHu07} B. X. Wang and C. Y. Huang,  Frequency-uniform decomposition method for the generalized BO,  KdV and NLS equations,   J. Differential Equations,  {\bf 239} (2007),  213--250.

\bibitem{WaZhGu06} B. X. Wang  L. Zhao, B.  Guo,  Isometric decomposition operators, function spaces $E^\lambda_{p,q}$   and applications to nonlinear evolution equations. J. Funct. Anal. \textbf{233} (2006), no. 1, 1--39.

\end{thebibliography}
\end{document}